\documentclass[a4paper,11pt,reqno]{amsart}
\usepackage{amssymb,amsmath,amsthm}
\usepackage{booktabs, array, ragged2e, longtable}
\newcolumntype{P}[1]{>{\RaggedRight\arraybackslash}p{#1}}
\numberwithin{equation}{section}
\usepackage{natbib}
\usepackage{paralist}
\usepackage[margin=0.75in]{geometry}
\usepackage{mathrsfs}
\usepackage{setspace}
\usepackage[final]{graphicx}
\usepackage{xcolor}
\usepackage{epstopdf}
\usepackage{subcaption}
\usepackage{float}
\usepackage{enumitem}
\usepackage{placeins}
\usepackage{hyperref}
\hypersetup{
  hidelinks,
  pdftitle={Finite-Sample Metric Non-Collapse for Geometrically Supervised Latent Control Models},
  pdfauthor={Alain Bensoussan, Minh-Nhat Phung, Minh-Binh Tran},
  pdfsubject={Finite-sample metric non-collapse and controlled semiconjugacy for geometrically supervised latent control models},
  pdfkeywords={geometrically supervised latent models, metric non-collapse, geometric coercivity, semiconjugacy, planning transfer}
}

\newcommand{\R}{\mathbb{R}}

\newcommand{\dd}{\mathrm{d}}

\DeclareMathOperator{\Lip}{Lip}
\DeclareMathOperator{\Cov}{Cov}

\theoremstyle{plain}
\newtheorem{theorem}{Theorem}[section]
\newtheorem{lemma}[theorem]{Lemma}
\newtheorem{proposition}[theorem]{Proposition}
\newtheorem{corollary}[theorem]{Corollary}

\theoremstyle{definition}
\newtheorem{definition}[theorem]{Definition}

\theoremstyle{remark}
\newtheorem{remark}[theorem]{Remark}

\newenvironment{assumption}[2]%
  {\par\medskip\noindent%
   \phantomsection\label{ass:#1}{}%
   \textbf{Assumption~#1\ (#2).}\itshape\ }%
  {\par\medskip}

\usepackage{multirow}
\newcolumntype{L}[1]{>{\raggedright\let\newline\\\arraybackslash\hspace{0pt}}m{#1}}
\newcolumntype{C}[1]{>{\centering\let\newline\\\arraybackslash\hspace{0pt}}m{#1}}
\newcolumntype{R}[1]{>{\raggedleft\let\newline\\\arraybackslash\hspace{0pt}}m{#1}}

\begin{document}
\title[Finite-Sample Metric Non-Collapse for Geometrically Supervised Latent World Models]{Finite-Sample Metric Non-Collapse for Geometrically Supervised Latent World Models in Control}
\author[Alain Bensoussan
	]{Alain Bensoussan
	}

	\address[Alain Bensoussan
	]{Naveen Jindal School of Management \\
		University of Texas at Dallas, Richardson, TX, 75080 USA}

	\email[
	]{\href{mailto:alain.bensoussan@utdallas.edu}{alain.bensoussan@utdallas.edu}}

	\author[Minh-Nhat Phung]{Minh-Nhat Phung }

	\address[Minh-Nhat Phung]{Department of Mathematics \\
		Texas A\&M University, College Station, TX, 77843 USA}

	\email[Minh-Nhat Phung]{\href{mailto:pmnt1114@tamu.edu}{pmnt1114@tamu.edu}}

	\author[Minh-Binh Tran]{Minh-Binh Tran}
	\address[Minh-Binh Tran]{Department of Mathematics \\
		Texas A\&M University, College Station, TX, 77843 USA}

	\email[Minh-Binh Tran]{\href{mailto:minhbinh@tamu.edu}{minhbinh@tamu.edu}}
	\thanks{ M.-N. Phung and M.-B. Tran are funded in part by NSF CAREER  DMS-2303146, and NSF Grants DMS-2204795, DMS-2305523,  DMS-2306379. The Supplementary Material is available on arXiv at https://arxiv.org/abs/2608.07265.}

\begin{abstract}
We establish a finite-sample learning-to-control theory for geometrically supervised latent models of nonlinear deterministic systems. Geometric supervision is used only during training: simulator state, proprioception, or state estimates with independently validated metric and directional error bounds supply observable-state distances and tangent directions, while deployment remains observation- and action-conditioned. We introduce an encoder-only local--global metric hinge that enforces directional resolution and separated-state discrimination. Under regular observable-factor, coverage, finite-capacity approximation, and uniform $C^{1,1}$ hypotheses, a computable one-sided regularization regime has a strong selection property: with high probability, every approximate empirical minimizer is simultaneously pointwise co-Lipschitz and uniformly approximately semiconjugate to the controlled dynamics. Approximation, sampling, and optimization errors remain explicit and separate. Norm-constrained tensor-product B-spline classes constructively realize the approximation hypotheses, and the interpolation exponent converting mean residual control into a uniform bound is sharp. A modular deterministic corollary transfers the learned certificates to trajectory, finite-horizon cost, learned-cost-head, and optimizer guarantees, while a validated finite-net result enables sharper model-specific certification. Controlled experiments isolate collapse and folding, quantify the analytic certificate's reserve, and demonstrate the control benefit of restored metric resolution. The principal contribution is a complete finite-sample implication from approximate empirical optimization to metric faithfulness, uniform controlled dynamics, and reliable planning for the same learned model.
\end{abstract}
\maketitle
\allowdisplaybreaks
\raggedbottom

\section{Introduction}
\label{sec:introduction}

A learned latent world model is useful for control only if it does more than fit one-step data on average. It must retain the state-space distinctions that matter under interventions, and its latent transition must remain uniformly compatible with the controlled dynamics. Otherwise a planner may exploit folds, overlaps, or poorly resolved regions of the representation even when the empirical prediction loss is small. The central question of this paper is therefore control-theoretic: under what conditions does an empirically trained latent model become a metrically faithful approximate realization of a nonlinear deterministic control system, with quantitative guarantees for downstream planning?

The model-based reinforcement-learning viewpoint that learned dynamics can support planning and policy improvement goes back to Dyna and underlies modern world-model methods for visual control~\cite{Sutton1991Dyna,Ha2018WorldModels,Hafner2019PlaNet,Hafner2020Dreamer,Hafner2025DreamerV3}. Methods such as PlaNet and Dreamer encode high-dimensional observations into compact latent variables and learn the dynamics in latent space. Joint-embedding predictive architectures (JEPA) adopt a related but distinct principle: instead of reconstructing every detail of the next observation, the model predicts its representation~\cite{LeCun2022AMI}. This principle has led to I-JEPA for images~\cite{Assran2023JEPA}, V-JEPA for video~\cite{Bardes2024VJEPA}, and V-JEPA~2, which combines video pretraining with robot interaction data to construct an action-conditioned latent model for planning~\cite{Assran2025VJEPA2}. These developments motivate a precise mathematical problem: when does action-conditioned latent prediction produce both a non-collapsed representation and a latent transition that faithfully represents the underlying controlled system?

We consider a deterministic controlled system
\[
s_{t+1}=G(s_t,a_t),
\qquad
x_t=H(s_t),
\]
where the deployed model receives the observation $x_t$. An action-conditioned latent model consists of an encoder $\Phi$ and a transition map $F$,
\[
z_t=\Phi(H(s_t)),
\qquad
z_{t+1}\approx F(z_t,a_t),
\]
with residual
\[
R[\Phi,F](s,a)
:=
\Phi(H(G(s,a)))-F(\Phi(H(s)),a).
\]
A uniform bound on this residual is an approximate controlled-semiconjugacy estimate: evolving the physical system and then encoding nearly agrees with encoding first and evolving in latent space. The prediction loss is observation-based, but the metric regularizer studied here uses additional training information. Observable-factor coordinates, or equivalent state distances and tangent directions, are used to evaluate its local and separated-pair terms. The framework is formulated for settings where observable-state metric information is available, including simulator-state, proprioceptive, and state-estimation regimes.

Forward prediction alone cannot yield the representation property needed for deterministic control. The constant encoder $\Phi\equiv z_0$, together with $F\equiv z_0$, has zero prediction loss. More subtly, positive latent variance or covariance does not prevent folding: a representation may have substantial global spread while identifying distinct states or losing arbitrarily much local metric resolution. For the worst-case deterministic planning guarantees pursued here, the relevant non-collapse condition is not merely distributional. It is a pointwise lower metric estimate of the form
\[
\|\Phi(H(s))-\Phi(H(s'))\|_Z
\ge c_\ast |s-s'|,
\qquad s,s'\in S,
\]
which makes the learned representation quantitatively injective on the observable state factor.

To enforce this geometry, we introduce an encoder-only local--global metric hinge. Its local directional component penalizes infinitesimal collapse of $D(\Phi\circ H)$, while its separated-pair component penalizes overlap of states that are a prescribed distance apart in the observable-state metric. The penalty is one-sided: it vanishes after the desired local and global margins are reached.  It is therefore decoder-free, inverse-model-free, and independent of a reconstruction objective. The two components have antecedents in conditioning and lower co-Lipschitz regularization; the contribution here is the finite-sample control theorem obtained from their joint enforcement.

The theorem-level question is whether approximate optimization of the empirical training objective itself forces both a quantitatively faithful representation and a uniformly accurate controlled latent transition.  The proof has the following parallel-output structure:
\[
\boxed{
\begin{array}{c}
\text{empirical objective comparison and margin clearing}\\[1mm]
\begin{array}{c@{\hspace{5em}}c}\swarrow&\searrow\end{array}\\[-1mm]
\begin{array}{c@{\qquad}c}
\begin{array}{c}\text{pointwise}\\[-1mm]\text{co-Lipschitzness}\end{array}
&
\begin{array}{c}\text{uniform approximate}\\[-1mm]\text{controlled semiconjugacy}\end{array}
\end{array}\\[-1mm]
\begin{array}{c@{\hspace{5em}}c}\searrow&\swarrow\end{array}\\[-1mm]
\text{worst-case finite-horizon optimizer transfer}
\end{array}}
\]
Thus the result is a finite-sample realization theorem rather than only an anti-collapse or planning-regret statement.  The representation theorem supplies the metric and dynamics constants in parallel; the deterministic planning corollary then uses both, while approximation, statistical, training-optimization, cost-head, and planner errors remain explicit and separate.

\textbf{Precise novelty statement.} Under explicit observable-geometry, approximation, coverage, regularity, and optimization hypotheses, our principal theorem proves a finite-sample selection principle with quantitative geometric coercivity: throughout a computable one-sided regularization half-line, every approximate empirical minimizer of the fixed objective is pointwise co-Lipschitz and uniformly approximately semiconjugate to the controlled dynamics. This every-minimizer objective-to-certificate implication, followed by a modular deterministic planning transfer, is the paper's central contribution.

The principal contributions are:
\begin{enumerate}[label=(\roman*),itemsep=0pt,topsep=2pt]
\item We identify collapse and folding obstructions for prediction and covariance control, and introduce an encoder-only local--global hinge aligned with pointwise metric faithfulness.
\item We prove the every-minimizer finite-sample co-Lipschitz and uniform controlled-semiconjugacy theorem, with approximation, statistical, and optimization errors separated and an explicit certified regularization regime.
\item We verify the finite-capacity hypotheses constructively for norm-constrained tensor-product B-splines, establish sharpness of the interpolation exponent, and transfer the learned certificates to deterministic planning with compatible or learned costs.
\item We provide controlled experiments that isolate the obstructions, quantify the analytic certificate's uniformity reserve, and demonstrate the control benefit of restored metric resolution.
\end{enumerate}

The explicit lower bound on $\lambda$ is the comparison scale at which the faithful zero-penalty comparator forces every approximate minimizer into the margin-clearing regime. It defines an a priori certified half-line and an error-budget map; after training, the finite-net proposition gives a sharper model-specific a posteriori certificate. The experiments use the same quantities diagnostically without identifying finite-set performance with the theorem's continuum guarantee.

\begin{theorem}[Main representation and planning theorem, informal form]
\label{thm:main-introduction}
Suppose that the controlled system admits a finite-dimensional regular observable factor, that the training and metric-sampling distributions satisfy the stated coverage conditions, and that the finite-capacity model classes approximate a smooth, non-collapsed exact realization with uniform \(C^{1,1}\) bounds. Then, for sufficiently large capacity \(W\) and sample size \(n\), for every confidence level \(\delta\in(0,1)\) and training tolerance \(\varepsilon_{\rm train}\ge0\), there exists an explicit regularization threshold such that, with probability at least \(1-\delta\), every \(\varepsilon_{\rm train}\)-approximate empirical minimizer \((\Phi^\star,F^\star)\) satisfies
\[
\|\Phi^\star(H(s))-\Phi^\star(H(s'))\|_Z
\ge c_\ast |s-s'|
\qquad
\text{for all }s,s'\in S,
\]
for an explicit constant \(c_\ast>0\), and
\[
\sup_{(s,a)\in S\times A}
\left\|
\Phi^\star(H(G(s,a)))
-
F^\star(\Phi^\star(H(s)),a)
\right\|_Z
\le
\eta_{W,n,\delta}(\varepsilon_{\rm train}).
\]
The modulus \(\eta_{W,n,\delta}(\varepsilon_{\rm train})\) tends to zero with the approximation, statistical, and training-optimization errors. 
Separately, if the physical stage and terminal costs are Lipschitz continuous in the state variable and \(\widehat{\mathbf a}\) is a \(\xi_{\rm plan}\)-approximate optimizer of the induced latent problem, then, for horizon \(T\), there exists an explicit bound \(\Delta_T\) such that
\[
J_{\rm true}(s_0;\widehat{\mathbf a})-J^\star(s_0)
\le
2\Delta_T\!\left(\eta_{W,n,\delta}(\varepsilon_{\rm train})\right)
+
\xi_{\rm plan},
\]
where \(J_{\rm true}\) denotes the total physical cost under the true dynamics and \(J^\star\) denotes the corresponding optimal value.

Learned latent cost heads add \(2(T\varepsilon_\ell+\varepsilon_g)\); exact constants and transfer bounds are in Theorem~\ref{thm:representation} and Corollary~\ref{cor:empirical-planning-transfer}.
\end{theorem}

Section~\ref{sec:related-representation} positions the result; Sections~\ref{sec:system-quotient}--\ref{sec:main} develop the model and prove the theorem; Sections~\ref{sec:why-semiconjugacy}--\ref{sec:discussion-certification} give planning transfer, evidence, scope, and outlook.

\section{Relation to coercivity, metric embeddings, abstraction, and planning transfer}
\label{sec:related-representation}

\paragraph{\bf Coercivity and stability gaps.}
A recurring theme in inverse problems and statistical learning is that a data-fit residual becomes scientifically meaningful only after it is paired with a stability mechanism; see, for example, the regularization framework of Engl, Hanke, and Neubauer~\cite{EnglHankeNeubauer1996}.  The present theory brings that perspective to latent control.  Pure prediction exposes a geometric stability gap, and the local--global hinge supplies a conditional lower stability estimate from empirical objective value to pointwise latent geometry under explicit observable-factor, realizability, regularity, coverage, supervision, and optimization hypotheses.

\paragraph{\bf Quantitative injectivity and metric regularization.}
Bi-Lipschitz and co-Lipschitz embeddings quantify injectivity through lower metric distortion, while pairwise and contrastive objectives shape representation geometry~\cite{HadsellChopraLeCun2006}.  The local term controls directional singular values on the observable factor, and the separated-pair term resolves distant folds.  The theorem then propagates sampled lower-margin control to a continuum co-Lipschitz constant under $C^{1,1}$ regularity and coverage, producing a quantitative injectivity certificate for every approximate empirical minimizer.

\paragraph{\bf Observability, abstraction, bisimulation, and semiconjugacy.}
Observable quotients and state abstractions identify the distinctions that must be retained for dynamics and control.  Approximation metrics and approximate bisimulation provide system-level notions of behavioral closeness~\cite{GirardPappas2007,GirardPappas2011,ferns2004metrics,ferns2011bisimulation}, with learned counterparts such as DeepMDP and DBC~\cite{gelada2019deepmdp,zhang2021learning}.  Here the observable factor isolates the recoverable control state, and the theorem equips it simultaneously with a lower metric certificate and a uniform deterministic controlled-semiconjugacy estimate.

\paragraph{\bf Empirical-to-uniform estimates.}
Uniform laws of large numbers and covering-number methods control deviations over function classes~\cite{vanDerVaartWellner1996,Dudley1999}.  We combine these population estimates with capacity-uniform Lipschitz budgets and lower-Ahlfors coverage to obtain pointwise continuum bounds.  The sharp $L^2$-to-$L^\infty$ exponent identifies the exact regularity-limited conversion rate and makes the dimension dependence of the guarantee explicit.

\paragraph{\bf Simulation lemmas and planning transfer.}
Classical simulation lemmas and Lipschitz model-based bounds propagate one-step model error to finite-horizon value or planning error~\cite{kearns2002near,asadi2018lipschitz}.  Our deterministic sup-norm formulation is tailored to optimization over action sequences.  The representation theorem supplies exactly the two constants required by the modular transfer argument: $\eta$ controls the latent dynamics mismatch, and $c_*$ controls the Lipschitz constants of compatible latent costs.

\paragraph{\bf JEPA and learned world models.}
Modern world models and JEPA methods motivate latent action-conditioned prediction~\cite{Sutton1991Dyna,Ha2018WorldModels,Hafner2019PlaNet,Hafner2020Dreamer,LeCun2022AMI,Assran2023JEPA,Bardes2024VJEPA,Assran2025VJEPA2}.  Recent theory establishes complementary results on finite-sample JEPA generalization and planning regret, structured identifiability, planning-aligned geometry, reachability, action sensitivity, and prediction--control mismatch~\cite{Cui2026GeneralizationJEPA,Klindt2026LeJEPAWorldModel,Zhang2026ControlledWorldModels,Destrade2026ValueGuidedJEPA,Li2026RCAux,You2026ControlPredictability,Wang2026CertifiedPredictability,Gan2026ActSWM,Liu2026ProWorld}.  Auxiliary-task JEPA theory characterizes which distinctions an additional prediction target can preserve~\cite{Yu2025AuxJEPA}, while bi-Lipschitz autoencoder regularization targets injectivity and robustness in reconstruction models~\cite{Zhan2026BLAE}.  Controlled-world-model identifiability addresses latent recovery under structural hypotheses~\cite{Zhang2026ControlledWorldModels}.  The present theorem addresses the subsequent objective-level implication: a fixed geometrically supervised finite-sample objective forces pointwise non-collapse and uniform semiconjugacy for all approximate minimizers in the certified regime, and those certificates feed directly into deterministic optimizer transfer.

\section{Deterministic controlled systems and the observable quotient}
\label{sec:system-quotient}

In this section, we mathematically formalize the deterministic system and the observable quotient.

First, let the state space \(S\subset\mathbb R^{d_s}\) be a compact convex Euclidean domain with nonempty interior, and let the action space \(A\subset\mathbb R^{d_a}\) be a compact Euclidean domain. No smoothness of \(\partial S\) is imposed. Whenever derivatives up to order \(C^{1,1}\) are used, the relevant maps are assumed to admit \(C^{1,1}\) extensions to open Euclidean neighborhoods of their compact domains. The metric spaces \((S,d_S)\) and \((A,d_A)\) are Euclidean, with
\begin{align*}
	d_S(s,s'):=|s-s'|,\qquad d_A(a,a'):=|a-a'|.
\end{align*}

The state-action space \(M = S \times A\) is equipped with the usual product metric
\begin{align*}
	d_M((s,a),(s',a')) := d_S(s,s')+d_A(a,a').
\end{align*}
The dimension of the state-action space is denoted by \(d_m:= d_s+d_a\).

The true deterministic dynamics and observation map are
\[
s_{t+1}=G(s_t,a_t),\qquad x_t=H(s_t),
\]
where
\[
G:S\times A\to S,
\qquad
H:S\to X.
\]
Here \(X\) is a separable Banach space, typically \(\R^{d_x}\) with \(d_x\gg d_s\).
We assume throughout that \(G\) and \(H\) are Lipschitz. In the metric-regularized results, Lemma~\ref{lem:smooth-realizer} and those that follow, the neighborhood extensions stipulated above are assumed to be of class \(C^{1,1}\).

If \(H\) is not injective, no representation built from observations can recover more
than the information contained in controlled observation histories. The correct object is
therefore an observable quotient.

\begin{definition}[Observational equivalence]
For \(s,s'\in S\), write \(s\sim s'\) if, for every \(k\ge0\) and every action sequence
\((a_0,\ldots,a_{k-1})\in A^k\),
\[
H(G_{a_{0:k-1}}(s))=H(G_{a_{0:k-1}}(s')),
\]
where \(G_{a_{0:k-1}}\) denotes the \(k\)-fold composition of \(G\) along the prescribed actions, that is
\begin{align*}
	G_{a_{0:-1}}(s)=s,\quad G_{a_{0:0}}(s)= G(s,a_0),\quad G_{a_{0:k}}(s)=G(G_{a_{0:k-1}}(s),a_k).
\end{align*}
This is the deterministic, controlled analogue of bisimulation and state abstraction~\cite{ferns2004metrics,ravindran2003smdp}.
\end{definition}

\begin{definition}[Observable quotient]
\label{def:observable-quotient}
For \(s\in S\), denote its observational equivalence class by
\[
[s]:=\{r\in S:\ r\sim s\}.
\]
The observable quotient is
\[
S_{\mathrm{obs}}:=S/\!\sim:=\{[s]:s\in S\}.
\]
For two equivalence classes \([s],[s']\in S_{\mathrm{obs}}\), define the quotient chain pseudometric\footnote{It is necessary to have chains in the definition, the naive definition \(\tilde d([s],[s'])=\inf\{d_S(u,v):u\in [s],v\in[s']\}\) is not enough. For example, \(S=\{0,1,2,3\}\), which is cut into \(3\) equivalent classes \(\{0\},\{1,2\},\{3\}\); \(\tilde d\) does not define a metric in such case.}
\[
\begin{aligned}
d_{\mathrm{obs}}([s],[s'])
:=
\inf_{N\ge1}\inf
\Bigg\{
\sum_{j=0}^{N-1} d_S(u_j,v_j):\;&
 u_j,v_j\in S,
\quad j=0,\ldots,N-1,\\
& u_0\in [s],
\quad v_{N-1}\in [s'],\\
& v_j\sim u_{j+1},
\quad j=0,\ldots,N-2
\Bigg\}.
\end{aligned}
\]
The condition \(v_j\sim u_{j+1}\) allows zero-cost jumps within observational equivalence
classes. If \(d_{\mathrm{obs}}\) is only a pseudometric, we quotient once more by the
zero-distance relation. Equivalently, in the sequel we work on the metric quotient and
write
\[
d_{S_{\mathrm{obs}}}:=d_{\mathrm{obs}}.
\]
\end{definition}
The dynamics descend to the quotient. Indeed, if \(s\sim s'\), then for every \(a\in A\)
and every future action sequence, the observation sequences generated from \(G(s,a)\) and
\(G(s',a)\) coincide. Hence
\[
\bar G([s],a):=[G(s,a)]
\]
is well-defined. The observation map descends to
\[
\bar H:S_{\mathrm{obs}}\to X,
\qquad
\bar H([s])=H(s),
\]
because \(s\sim s'\) implies \(H(s)=H(s')\). However, this quotient construction alone
need not make \(\bar H\) bi-Lipschitz. Thus the following observability condition is a
substantive structural assumption.

\begin{assumption}{A1}{Regular observable Markov factor}
We assume that there exists a finite-dimensional observable factor
\[
        q:S\to S_{\rm o}\subset\mathbb R^{d_o},
\]
where \(S_{\rm o}\) is a compact convex Euclidean domain with nonempty interior, such that
\[
        q(s)=q(s') \quad\Longleftrightarrow\quad s\sim s'.
\]
Moreover, the dynamics, observation map, and costs factor through \(q\):
there exist
\[
        G_{\rm o}:S_{\rm o}\times A\to S_{\rm o},\qquad
        H_{\rm o}:S_{\rm o}\to X,
\]
and
\[
        \ell_{\rm o}:S_{\rm o}\times A\to\mathbb R,\qquad
        g_{\rm o}:S_{\rm o}\to\mathbb R
\]
such that
\[
        q(G(s,a))=G_{\rm o}(q(s),a),
        \qquad
        H(s)=H_{\rm o}(q(s)),
\]
and
\[
        \ell(s,a)=\ell_{\rm o}(q(s),a),
        \qquad
        g(s)=g_{\rm o}(q(s)).
\]
The maps \(G_{\rm o}\) and \(H_{\rm o}\) admit \(C^{1,1}\) extensions to
open Euclidean neighborhoods of their compact domains, and \(H_{\rm o}\) is
bi-Lipschitz onto its image:
\[
        c_H |y-y'|
        \le
        \|H_{\rm o}(y)-H_{\rm o}(y')\|_X
        \le
        L_H |y-y'|,
	\qquad \text{for all }y,y'\in S_{\rm o}.
\]
After this identification, we work on \(S:=S_{\rm o}\), \(G:=G_{\rm o}\),
and \(H:=H_{\rm o}\).  All derivatives below are taken with respect to the
Euclidean observable coordinate \(y\in S_{\rm o}\).
\end{assumption}

Under Assumption~\hyperref[ass:A1]{A1}, the observable factor is identifiable from the current observation through the bi-Lipschitz map $H_{\rm o}$. Controlled histories motivate the quotient construction, but they are not used by the deployed encoder in the theorem.

\begin{remark}
An arbitrary metric quotient space is not automatically a smooth Euclidean domain. Accordingly, the present Euclidean $C^{1,1}$ proof framework imposes the regular finite-dimensional factor structure of Assumption~\hyperref[ass:A1]{A1}.

Classical indistinguishability theory in control has addressed the manifold structure of the quotient space and has provided a sufficient condition in~\cite[Theorem 3.5]{hermann1977nonlinear}.

Under certain assumptions, the quotient metric $d_{\rm obs}$ endows a topology that agrees with the quotient topology in~\cite[Ex. 3.1.14]{burago2001course}.

Assumption~\hyperref[ass:A1]{A1} supplies a finite-dimensional Euclidean coordinate for the observable factor. The Euclidean factor metric used in all subsequent estimates need not coincide with the chain metric $d_{\rm obs}$. If one wishes to translate the estimates back to the chain metric, an additional bi-Lipschitz comparison between the two metrics must be assumed.
\end{remark}

For the remainder of the paper, the metric on the observable factor is defined by
\[
        \bar d_{\rm obs}([s],[s'])
        :=
        |q(s)-q(s')|,
\]
and after identifying $S$ with $S_{\rm o}$ we write $d_S(s,s'):=|s-s'|$. All lower co-Lipschitz constants below refer to this Euclidean observable-factor metric.

\begin{assumption}{A2}{Observable-factor excitation / coverage}
Let \(\mu\) be the training measure on \(S\times A\), and define the pushforward measure
\[
        \mu_{\rm o}:=(q,\operatorname{Id}_A)_\#\mu
\]
on \(S_{\rm o}\times A\). We assume that \(\mu_{\rm o}\) is lower Ahlfors
\((d_o+d_a)\)-regular: there exist \(m_0>0\) and \(r_0>0\) such that
\[
        \mu_{\rm o}(B_r(s,a))
        \ge m_0 r^{d_o+d_a}
        \qquad
	\text{for all }(s,a)\in S_{\rm o}\times A,\quad 0<r\le r_0.
\]
\end{assumption}

After the identification in Assumption~\hyperref[ass:A1]{A1}, we write \(\mu:=\mu_{\rm o}\), \(d_s:=d_o\), and \(d_m:=d_o+d_a\).

In the common absolutely continuous case in which \(\dd\mu=\rho\,\dd y\) on \(M\subset\mathbb R^{d_m}\), a lower density bound \(\rho\ge\rho_{\min}>0\) implies Assumption~\hyperref[ass:A2]{A2} provided the domain also has a uniform interior-volume lower bound, namely \(|M\cap B_r(y)|\ge c_M r^{d_m}\) for all \(y\in M\) and all sufficiently small \(r\). This geometric thickness condition holds, for example, for compact convex domains and for bounded Lipschitz domains after fixing the small-radius range.

\section{Action-conditioned latent world models}
\label{sec:latent-world-model}
In this section, we define the notation employed for the JEPA model.

Let \(Z=\R^{d_z}\) be the latent space and let \(K_Z\) be a compact convex subset of \(Z\) with nonempty interior. A positive local directional margin requires $D\psi(s):\mathbb R^{d_s}\to\mathbb R^{d_z}$ to be injective, and therefore necessarily $d_z\ge d_s$; we assume this dimension condition throughout the metric results.  Because $S\subset\mathbb R^{d_s}$ is compact and $K_Z$ has nonempty interior, a translated and scaled coordinate inclusion embeds $S$ smoothly and bi-Lipschitzly into $\operatorname{int}K_Z$.  Supplementary Section~S1 gives this construction and a smooth range-enforcing map explicitly.
An action-conditioned latent world model consists of an encoder
\[
\Phi:X\to K_Z
\]
and a latent transition map
\[
F:K_Z\times A\to K_Z.
\]
Only the restriction of \(\Phi\) to the set \(H(S)\) is utilized.
Consequently, we define the induced latent state mapping
\[
	\psi:=\Phi\circ H:S\to K_Z.
\]

To analyze the learning properties of the JEPA model, we fix uniformly bounded Lipschitz classes of ambient maps
\[
\mathcal A_\Phi
\subset
\bigl\{\Phi:X\to K_Z:\ \Phi|_{H(S)}\in C^{0,1}(H(S);K_Z)\bigr\},
\qquad
\mathcal A_F\subset C^{0,1}(K_Z\times A;K_Z),
\qquad
\mathcal A:=\mathcal A_\Phi\times\mathcal A_F.
\]
Only the restrictions of the encoders to \(H(S)\) enter the objective. For compactness of minimizers, we assume that these restrictions and the transition classes are closed in the relevant uniform topology.

For the metric results in Section~\ref{sec:main}, we take \(X=\mathbb R^{d_x}\) and require each admissible encoder to be supplied with a specified \(C^{1,1}\) ambient representative on an open neighborhood of \(H(S)\) (the concrete spline encoders are global maps on \(X\)). For such a representative, we use
\[
\|\Phi\|_{C^1(H(S))}
:=
\sup_{x\in H(S)}\|\Phi(x)\|_Z
+
\sup_{x\in H(S)}\|D\Phi(x)\|_{\mathrm{op}}.
\]
Thus every derivative of \(\Phi\) appearing below is the derivative of its specified ambient representative. Equivalently, the geometric regularity may be expressed through the induced map \(\psi=\Phi\circ H\) on the Euclidean observable factor.

For a horizon \(T\in\mathbb N\), a stage cost \(\ell:S\times A\to\mathbb R\), and a terminal cost  \(g:S\to\mathbb R\), we denote an action sequence by
\[
a_{0:T-1}:=(a_0,\dots,a_{T-1})\in A^T.
\]
When the horizon \(T\) is clear from the context, we also write
\(\mathbf a:=a_{0:T-1}\). The deterministic control problem is
\[
J^\star(s_0)=\inf_{\mathbf a\in A^T} J_{\rm true}(s_0;\mathbf a),
\]
where
\[
J_{\rm true}(s_0;\mathbf a)
=
\sum_{t=0}^{T-1}\ell(s_t^{\mathbf a},a_t)+g(s_T^{\mathbf a}),
\qquad
s_{t+1}^{\mathbf a}=G(s_t^{\mathbf a},a_t).
\]
If $G$, $\ell$, and $g$ are continuous, compactness of $A^T$ implies existence of an optimal control sequence.

The next lemma establishes the mathematical availability of exact compatible latent costs once metric non-collapse has been proved. Section~\ref{sec:simulation-lemma} then extends this construction to learned cost heads through explicit compatibility errors.

\begin{lemma}[Compatible latent costs induced by metric non-collapse]
\label{lem:latent-cost-extension}
Let \(\psi:S\to K_Z\) satisfy
\[
\|\psi(s)-\psi(s')\|_Z
\ge c_\ast|s-s'|
\qquad
\text{for all }s,s'\in S,
\]
with \(c_\ast>0\). Assume
\begin{equation}
\label{eq:separate-cost-lipschitz}
|\ell(s,a)-\ell(s',a')|
\le
L_{\ell,s}|s-s'|+L_{\ell,a}|a-a'|,
\qquad
|g(s)-g(s')|
\le
L_g|s-s'|.
\end{equation}
Then
\[
\ell_\psi(\psi(s),a):=\ell(s,a),
\qquad
g_\psi(\psi(s)):=g(s)
\]
are well-defined. They admit extensions
\[
\widetilde\ell:K_Z\times A\to\mathbb R,
\qquad
\widetilde g:K_Z\to\mathbb R
\]
satisfying exact compatibility
\[
\widetilde\ell(\psi(s),a)=\ell(s,a),
\qquad
\widetilde g(\psi(s))=g(s),
\]
and the separate Lipschitz estimate
\begin{equation}
\label{eq:latent-cost-sharp-constants}
|\widetilde\ell(z,a)-\widetilde\ell(z',a')|
\le
\frac{L_{\ell,s}}{c_\ast}\|z-z'\|_Z
+
L_{\ell,a}|a-a'|,
\qquad
|\widetilde g(z)-\widetilde g(z')|
\le
\frac{L_g}{c_\ast}\|z-z'\|_Z.
\end{equation}
In particular,
\[
L_{\widetilde\ell,z}\le\frac{L_{\ell,s}}{c_\ast},
\qquad
L_{\widetilde g}\le\frac{L_g}{c_\ast}.
\]
\end{lemma}

\begin{proof}
The lower co-Lipschitz estimate implies injectivity and
\[
|s-s'|\le c_\ast^{-1}\|\psi(s)-\psi(s')\|_Z.
\]
Hence, \(\ell_\psi\) is \(1\)-Lipschitz on \(\psi(S)\times A\) with respect to the weighted product distance
\[
d_\ell((z,a),(z',a'))
:=
\frac{L_{\ell,s}}{c_\ast}\|z-z'\|_Z
+L_{\ell,a}|a-a'|.
\]
When both weights are positive this is a metric; if one weight vanishes it is a pseudometric, and the same argument is applied after quotienting the corresponding zero-distance equivalence classes. Likewise, \(g_\psi\) is \(L_g/c_\ast\)-Lipschitz on \(\psi(S)\). Therefore, applying McShane's extension theorem \cite{McShane1934} to the resulting metric quotient (or directly when the weights are positive), we conclude that the functions
\[
\widetilde\ell(z,a)
:=
\inf_{(s',a')\in S\times A}
\left\{
\ell(s',a')
+\frac{L_{\ell,s}}{c_\ast}\|z-\psi(s')\|_Z
+L_{\ell,a}|a-a'|
\right\}
\]
and
\[
\widetilde g(z)
:=
\inf_{s'\in S}
\left\{
g(s')+\frac{L_g}{c_\ast}\|z-\psi(s')\|_Z
\right\}
\]
agree with the original functions on the encoded state set and satisfy
\eqref{eq:latent-cost-sharp-constants}.
\end{proof}

For a representation \((\Phi,F)\) such that \(\psi=\Phi\circ H\) is lower co-Lipschitz, and for Lipschitz costs \(\ell\) and \(g\) as in Lemma~\ref{lem:latent-cost-extension}, we obtain induced latent costs \(\widetilde{\ell}\) and \(\widetilde{g}\). The corresponding total latent cost is defined by
\[
J_{\rm lat}(\psi(s_0);\mathbf a)
:=
\sum_{t=0}^{T-1}\widetilde{\ell}(z_t,a_t)+\widetilde{g}(z_T),
\qquad
z_0=\psi(s_0),\quad
z_{t+1}=F(z_t,a_t),\quad t=0,\dots,T-1.
\]

Let \(Y\subset\mathbb R^d\) be a Euclidean domain and let \(E\) be a normed space. The Lipschitz constant of \(f:Y\to E\) is denoted by \(\Lip(f)\). A function class \(\mathfrak F\) has Lipschitz budget \(L\) if every \(f\in \mathfrak F\) is Lipschitz and \(\Lip(f)\le L\).
Furthermore, \(\mathfrak F\) has \(C^{1,1}\) budget \(\Lambda\) if every \(f\in \mathfrak F\) has a specified extension that is \(C^{1,1}\) on an open neighborhood of \(Y\) and
\[
\|f\|_{C^0(Y)}+\|Df\|_{C^0(Y)}+\Lip(Df;Y)\le \Lambda.
\]
Here, \(Df:Y\to\mathcal L(\mathbb R^d,E)\) denotes the derivative of the specified extension, and
\[
\|Df\|_{C^0(Y)}
:=
\sup_{y\in Y}\|Df(y)\|_{\mathcal L(\mathbb R^d,E)}.
\]

\begin{definition}[JEPA residual and prediction loss]
For \((\Phi,F)\in\mathcal A\), define
\[
R[\Phi,F](s,a)
:=
\Phi(H(G(s,a)))-F(\Phi(H(s)),a)
=
\psi(G(s,a))-F(\psi(s),a).
\]
The population prediction loss is
\[
	\mathcal E_{\rm pred}(\Phi,F)
:=
\int_M \|R[\Phi,F](s,a)\|_Z^2\dd\mu(s,a).
\]
Given samples \((s_i,a_i)_{i=1}^n\sim\mu^{\otimes n}\), the empirical prediction loss is
\[
\widehat{\mathcal E_{\rm pred}}_n(\Phi,F)
:=
\frac1n\sum_{i=1}^n
\|R[\Phi,F](s_i,a_i)\|_Z^2.
\]
\end{definition}

The residual is uniformly Lipschitz on \(M\) whenever \(G,H,\Phi,F\) have uniform Lipschitz constants. We denote a structural Lipschitz bound by \(L_R\).

\begin{proposition}[Existence of regularized empirical minimizers]
\label{prop:existence}
Let $\tau$ be a topology on the admissible class $\mathcal A=\mathcal A_\Phi\times\mathcal A_F$. Assume that $\mathcal A$ is compact in $\tau$, and that the empirical prediction loss and $\widehat{\mathcal N}_n:\mathcal A_\Phi\to[0,\infty]$ are $\tau$-lower semicontinuous. Then, for every $\lambda>0$, the empirical objective
\[
\widehat{\mathcal E}_{\lambda,n}(\Phi,F)
:=
\widehat{\mathcal E_{\rm pred}}_n(\Phi,F)
+\lambda\widehat{\mathcal N}_n(\Phi)
\]
attains its infimum on $\mathcal A$.

For the Lipschitz-only setting one may take $\tau$ to be uniform convergence and use Arzel\`a--Ascoli. For the metric regularizer, which depends on $D(\Phi\circ H)$, one takes $\tau$ to be the $C^1$ topology. At each fixed capacity $W$, a class that is closed in $C^1$ and uniformly bounded in $C^{1,1}$ is compact in $C^1$; the empirical local and global hinge terms are then continuous.
\end{proposition}
\begin{proof}
	Let \(\{(\Phi_j,F_j)\}_{j\ge1}\) be a minimizing sequence for the empirical objective. Since \(\mathcal A\) is \(\tau\)-compact, there exists a $\tau$-convergent subsequence \(\{(\Phi_{j_k},F_{j_k})\}_{k\ge1}\) with limit $(\Phi_0,F_0)\in\mathcal A$. Because the prediction term and the regularizer are \(\tau\)-lower semicontinuous, we have
\[
\widehat{\mathcal E}_{\lambda,n}(\Phi_0,F_0)
\le
\liminf_{k\to\infty}
\widehat{\mathcal E}_{\lambda,n}(\Phi_{j_k},F_{j_k}).
\]
Thus, $(\Phi_0,F_0)$ is a minimizer.
\end{proof}

We refer to the parameter \(\lambda\) as \emph{regularization strength}. The metric regularizer is presented in Section~\ref{sec:metric-regularizer}. The admissible range of \(\lambda\) for the uniform approximate controlled semiconjugacy estimate is discussed in Section~\ref{sec:main}.

\section{Geometric obstructions for pure forward prediction and covariance/spectral spread}
\label{sec:failures}

The following two propositions identify the obstruction that the metric regularizer will
address. They show that forward prediction alone, and forward prediction plus purely
second-moment spread penalties, do not imply a controlled semiconjugacy useful for
planning.

\begin{proposition}[Pure forward JEPA admits collapsed zero-loss minimizers]
\label{prop:pure-fails}
Assume \(K_Z\neq\emptyset\). For any \(z_0\in K_Z\), the constant encoder and transition
map
\[
\Phi(x)\equiv z_0,
\qquad
F(z,a)\equiv z_0
\]
satisfy
\[
R[\Phi,F](s,a)=0
\qquad \text{for all }(s,a)\in M.
\]
Hence pure forward prediction cannot by itself imply non-collapse.
\end{proposition}

\begin{proof}
For every \((s,a)\), we have
\[
\Phi(H(G(s,a)))=z_0=F(\Phi(H(s)),a).
\]
Thus, the residual is identically zero. On the other hand, we also observe that the induced latent map \(\psi\equiv z_0\) has no positive lower co-Lipschitz constant unless \(S\) is a singleton.
\end{proof}

To prevent a total collapse in pure forward JEPA, a natural attempt is to add covariance or spectral spread.

\begin{definition}[State marginal of the training measure]\label{def:state-marginal}
Let \(\mu\) be the training measure on the state--action space \(M=S\times A\). We denote
by
\[
\pi_S:M\to S,
\qquad
\pi_S(s,a)=s,
\]
the canonical projection. The \(S\)-marginal of \(\mu\) is the pushforward probability
measure
\[
\mu_S:=(\pi_S)_\#\mu.
\]
Equivalently, for every Borel set \(B\subset S\),
\[
\mu_S(B)=\mu(B\times A).
\]
Equivalently, if \((\mathbf S,\mathbf A)\) is an \(M\)-valued random variable with
\(\mathcal L(\mathbf S,\mathbf A)=\mu\), then
\[
\mathcal L(\mathbf S)=\mu_S.
\]\end{definition}

For a latent map \(\psi:S\to Z=\mathbb R^{d_z}\), define its \(\mu_S\)-mean by
\[
\bar\psi:=\int_S\psi(s)\dd\mu_S(s)\in\mathbb R^{d_z}.
\]
The covariance matrix of \(\psi(s)\) under \(s\sim\mu_S\) is
\[
\Sigma_\psi
:=
\Cov_{s\sim\mu_S}(\psi(s))
:=
\int_S
\bigl(\psi(s)-\bar\psi\bigr)
\bigl(\psi(s)-\bar\psi\bigr)^\top
\dd\mu_S(s)
\in\R^{d_z\times d_z}.
\]
Equivalently,
\[
(\Sigma_\psi)_{ij}
=
\int_S
\bigl(\psi_i(s)-\bar\psi_i\bigr)
\bigl(\psi_j(s)-\bar\psi_j\bigr)
\dd\mu_S(s).
\]
Variance-type regularizers penalize small coordinate variances, i.e. small diagonal
entries \((\Sigma_\psi)_{jj}\). Spectral regularizers penalize small eigenvalues, for
example through a term such as
\[
\sum_{j=1}^{d_z}
\bigl[\gamma-\lambda_j(\Sigma_\psi)\bigr]_+^2
\qquad\text{or}\qquad
\bigl[\gamma-\lambda_{\min}(\Sigma_\psi)\bigr]_+^2.
\]
Such penalties rule out total collapse in a global second-moment sense, but they do not imply injectivity or a pointwise lower co-Lipschitz bound.

\begin{proposition}[Covariance/Spectral spread does not imply injectivity]
\label{prop:covariance-fails}
There exists a deterministic control system and a non-injective zero-residual latent model whose latent code has positive variance.
Consequently, covariance or spectral spread alone cannot imply the desired pointwise lower co-Lipschitz property without additional structural assumptions.
\end{proposition}

\begin{proof}
We consider the example
\[
S=[-1,1],
\qquad A=\{0\},
\qquad G(s,0)=s,
\qquad H(s)=s.
\]
Let $\mu_S$ be the uniform probability measure on $[-1,1]$. Here the singleton action set is only a notational shorthand for an uncontrolled special case. If one insists that the action domain have nonempty Euclidean interior, one may instead take any compact interval \(A\) and define \(G(s,a)=s\) and \(F(z,a)=z\) independently of \(a\); the argument is unchanged.
Let \(Z=\mathbb R\). We choose a constant \(C>0\), and set
\[
\psi(s)=Cs^2,
\qquad
F(z,0)=z.
\]
Then, we obtain the residual
\[
R[\psi,F](s,0)=\psi(G(s,0))-F(\psi(s),0)=Cs^2-Cs^2=0.
\]
Moreover, we compute the variance
\[
	\operatorname{Var}_{\mu_S}(\psi(s))
=
C^2\operatorname{Var}_{\mu_S}(s^2)
= C^2 (\mathbb E_{\mu_S}[s^4]-(\mathbb E_{\mu_S}[s^2])^2)
=
\frac{4C^2}{45}.
\]
Thus, \(\psi(s)\) has positive variance for every fixed \(C>0\). For any prescribed variance threshold, one may choose a sufficiently large compact latent interval \(K_Z=[0,C]\) and then choose \(C\) so that the threshold is cleared. No claim of unbounded variance is made within a single fixed compact latent range.

However, \(\psi(s)=\psi(-s)\), so
\(\psi\) is not injective and no inequality of the form
\[
\|\psi(s)-\psi(s')\|_Z\ge c\,d_S(s,s')
\]
can hold with \(c>0\).
\end{proof}

These examples identify the precise role of geometric regularization.  Pure forward JEPA leaves non-collapse unconstrained, and covariance or spectral spread leaves pointwise lower geometry unconstrained.  The local--global objective supplies exactly the directional and separated-pair information used by the positive theorem to establish a uniform co-Lipschitz certificate.

\section{The metric hinge regularizer}
\label{sec:metric-regularizer}

We now introduce a soft encoder-only metric regularizer. It requires neither a decoder nor an inverse model and imposes no hard architectural constraint.  Instead, it directly promotes two complementary geometric properties of the latent representation: infinitesimal directional resolution and separation of well-separated states.

Throughout this section, \(S\subset\R^{d_s}\) is a compact convex Euclidean domain and all maps are understood as restrictions of \(C^{1,1}\) maps defined on open neighborhoods.

\paragraph{\bf Training-information convention.}
Although the prediction residual is evaluated from observation--action transitions, the metric penalty is supervised by the observable-factor geometry. Its empirical implementation uses samples of the coordinate $s$ (or an equivalent metric representation), state distances $|s-s'|$, and tangent directions $v$. This directly covers simulator-state and proprioceptive settings while preserving observation--action deployment.  State-estimation or pixel-based geometry may be used when its distance and direction errors are independently bounded and absorbed into reduced effective margins before applying Assumptions~\hyperref[ass:A3]{A3}--\hyperref[ass:A4]{A4}; the theorem is then applied to those certified effective quantities.  No unquantified state-estimation claim is needed, and the deployed encoder--transition interface is unchanged.

Let
\[
\mathcal U:=S\times\mathbb S^{d_s-1},
\]
where \(\mathbb S^{d_s-1}\) is the Euclidean unit sphere. We call \(\mathcal U\) the unit-direction sampling set.
It is equipped with the product probability measure
\[
\omega:=\mu_S\otimes\sigma,
\]
where \(\sigma\) is normalized surface measure on \(\mathbb S^{d_s-1}\). Fix a scale
\(0<\rho<{\rm diam}(S)\). Let
\[
P_\rho:=\{(s,s')\in S\times S:\ |s-s'|\ge \rho\}.
\]
Assuming \((\mu_S\otimes\mu_S)(P_\rho)>0\), we define \(\nu_\rho\) to be the normalized restriction of \(\mu_S\otimes\mu_S\) to \(P_\rho\), namely
\[
\nu_\rho(E)
:=
\frac{(\mu_S\otimes\mu_S)(E\cap P_\rho)}
{(\mu_S\otimes\mu_S)(P_\rho)}
\qquad
\text{for every Borel set }E\subset S\times S.
\]
Equivalently, \(\nu_\rho\) is the conditional law of two independent
\(\mu_S\)-distributed states \(s,s'\), conditioned on the event
\(|s-s'|\ge \rho\).

The spaces \(M, \mathcal U\) and \(P_\rho\) are endowed with the product metrics
\[
d_M((s,a),(s',a')):=|s-s'|+|a-a'|,
\]
\[
d_{U}((s,v),(s',v')):=|s-s'|+|v-v'|,
\]
and
\[
d_P((s_1,s_2),(s_1',s_2')):=|s_1-s_1'|+|s_2-s_2'|.
\]

Fix margins \(\kappa>0\) and \(\alpha>0\). For \(\psi=\Phi\circ H\), we define the metric regularizers
\[
	\mathcal N_{\rm loc}(\Phi)
:=
\int_{\mathcal U}
\bigl[\kappa-\|D\psi(s)[v]\|_Z\bigr]_+^2
\dd\omega(s,v),
\]
\[
	\mathcal N_{\rm glob}(\Phi)
:=
\int_{P_\rho}
\bigl[\alpha-\|\psi(s)-\psi(s')\|_Z\bigr]_+^2
\dd\nu_\rho(s,s'),
\]
and
\[
	\mathcal N_{\rm met}(\Phi):=\mathcal N_{\rm loc}(\Phi)+\mathcal N_{\rm glob}(\Phi).
\]
The local term is an averaged hinge version of the pointwise infinitesimal condition
\[
\sigma_{\min}(D\psi(s))\ge\kappa,
\qquad
\sigma_{\min}(D\psi(s))=
\inf_{|v|=1}\|D\psi(s)[v]\|_Z.
\]
In the same perspective, the global term is an averaged hinge version of the metric condition
\[
	\|\psi(s)-\psi(s')\|_{Z}\ge\alpha,\quad\text{when } |s-s'|\ge\rho.
\]

We now define the empirical version of the metric regularizer. Let
\[
\{(s_i,a_i)\}_{i=1}^{n_M}
\]
be independent samples from the training measure \(\mu\) on \(M=S\times A\). Let
\[
\{(\xi_j,v_j)\}_{j=1}^{n_U}
\]
be independent samples from the unit-direction measure
\[
\omega=\mu_S\otimes \sigma
\qquad\text{on}\qquad
\mathcal U=S\times\mathbb S^{d_s-1},
\]
where \(\sigma\) is normalized surface measure on
\(\mathbb S^{d_s-1}\). Finally, let
\[
\{(\zeta_k,\zeta'_k)\}_{k=1}^{n_P}
\]
be independent samples from the separated-pair measure \(\nu_\rho\) on
\[
P_\rho=\{(s,s')\in S\times S:\ |s-s'|\ge \rho\}.
\]
We assume that these three sample families are mutually independent.

We denote by \(\mathcal A(W)\) an admissible class of encoder--latent-transition pairs depending on some trainable parameter \(W\). For \((\Phi,F)\in\mathcal A(W)\), with \(\psi=\Phi\circ H\), we define the empirical prediction loss by
\[
	\widehat{\mathcal E_{\rm pred}}_{n_M}(\Phi,F)
:=
\frac1{n_M}\sum_{i=1}^{n_M}
\|
\psi(G(s_i,a_i))-F(\psi(s_i),a_i)
\|_Z^2.
\]
We define the empirical local metric penalty by
\[
	\widehat{\mathcal N_{\rm loc}}_{n_U}(\Phi)
:=
\frac1{n_U}\sum_{j=1}^{n_U}
\bigl[
\kappa-\|D\psi(\xi_j)[v_j]\|_Z
\bigr]_+^2,
\]
and define the empirical separated-pair metric penalty by
\[
	\widehat{\mathcal N_{\rm glob}}_{n_P}(\Phi)
:=
\frac1{n_P}\sum_{k=1}^{n_P}
\bigl[
\alpha-\|\psi(\zeta_k)-\psi(\zeta'_k)\|_Z
\bigr]_+^2.
\]
The empirical metric regularizer is then
\[
	\widehat{\mathcal N_{\rm met}}_{n_U,n_P}(\Phi)
:=
\widehat{\mathcal N_{\rm loc}}_{n_U}(\Phi)
+
\widehat{\mathcal N_{\rm glob}}_{n_P}(\Phi).
\]
Consequently, for a regularization parameter \(\lambda>0\), the metric-regularized empirical JEPA objective is
\[
	\widehat{\mathcal E}^{\mathrm{met}}_{\lambda,n_M,n_U,n_P}(\Phi,F)
:=
\widehat{\mathcal E_{\rm pred}}_{n_M}(\Phi,F)
+
\lambda\widehat{\mathcal N_{\rm met}}_{n_U,n_P}(\Phi).
\]
When no distinction between the three sample sizes is needed, we take
\[
n_M=n_U=n_P=n
\]
and write simply
\[
	\widehat{\mathcal E}^{\mathrm{met}}_{\lambda,n}(\Phi,F)
:=
\widehat{\mathcal E_{\rm pred}}_{n}(\Phi,F)
+
\lambda\widehat{\mathcal N_{\rm met}}_{n}(\Phi),
\qquad
\widehat{\mathcal N_{\rm met}}_{n}(\Phi)
:=
\widehat{\mathcal N_{\rm loc}}_{n}(\Phi)+\widehat{\mathcal N_{\rm glob}}_{n}(\Phi).
\]

\begin{assumption}{A3}{Metric sampling geometry}
The metric sampling spaces \((\mathcal U,d_U,\omega)\) and \((P_\rho,d_P,\nu_\rho)\) are lower Ahlfors regular.
That is, there exist constants \(m_{U},m_{P},r_{U},r_{P}>0\) and dimensions
\[
q_U=2d_s-1,
\qquad
q_P=2d_s,
\]
such that
\[
	\begin{split}
\omega(B_r(s,v))\ge m_{U} r^{q_{U}}
\qquad \text{for all }(s,v)\in \mathcal U,
\quad 0<r\le r_U,\\
\nu_{\rho}(B_{\bar r}(s,s'))\ge m_{P} \bar r^{q_{P}}
\qquad \text{for all }(s,s')\in P_\rho,
\quad 0<\bar r\le r_P.
	\end{split}
\]
\end{assumption}
The lower Ahlfors conditions in Assumptions~\hyperref[ass:A2]{A2} and~\hyperref[ass:A3]{A3} yield the following \(L^2\to L^\infty\) interpolation property.

\begin{lemma}[Lipschitz--Ahlfors \(L^2\)-to-\(L^\infty\) interpolation]
\label{lem:l2-linf}
Let \((Y,d,\nu)\) be a compact metric probability space. Assume that there exist
constants \(m>0\), \(r_0>0\), and an exponent \(q>0\) such that
\[
\nu(B_r(y))\ge m r^q
\qquad \text{for all }y\in Y,
\quad 0<r\le r_0.
\]

If \(u:Y\to[0,\infty)\) is \(0\)-Lipschitz, then it is constant and
\[
\|u\|_{L^\infty(Y)}^2
=
\int_Yu^2\,\dd\nu.
\]
Thus the constant case is immediate. In the nonconstant case, let \(u\) be \(L\)-Lipschitz with \(L>0\), and suppose
\[
\int_Y u^2\dd\nu\le \varepsilon.
\]
Then
\[
\|u\|_{L^\infty(Y)}
\le
\Theta_Y(\varepsilon),
\]
where
\[
\Theta_Y(\varepsilon)
:=
2\max\left\{
\left(\frac{L^q\varepsilon}{m}\right)^{1/(q+2)},
\quad
\left(\frac{\varepsilon}{m r_0^q}\right)^{1/2}
\right\}.
\]
In particular, \(\Theta_Y(\varepsilon)\to0\) as \(\varepsilon\to0\). Furthermore, we set
\[\Theta_Y^{-1}(\tau):=\min\left\{ \frac{m}{L^q}\left( \frac{\tau}{2}\right)^{q+2},mr_0^q\left( \frac{\tau}{2} \right)^{2} \right\},\]
which satisfies \(\Theta_Y(\Theta_Y^{-1}(\tau))=\tau\) for \(\tau>0\). In the case \(L=0\), we set \(1/L^q=+\infty\) by convention.
\end{lemma}

\begin{proof}
We set \(h=\|u\|_{L^\infty(Y)}\), and choose \(y_0\in Y\) such that \(u(y_0)=h\).
We suppose first that \(h/(2L)\le r_0\), then the Lipschitz bound implies
\[
u(y)\ge \frac h2
\qquad \text{for all } y\in B_{h/(2L)}(y_0).
\]
Consequently, using the hypothesis, we obtain
\[
\varepsilon
\ge \int_Y u^2\dd\nu
\ge \frac{h^2}{4}\nu(B_{h/(2L)}(y_0))
\ge \frac{h^2}{4}m\left(\frac{h}{2L}\right)^q
=
\frac{m}{2^{q+2}L^q}h^{q+2}.
\]
Hence, in the first case, we arrive at
\[
h\le
2\left(\frac{L^q}{m}\varepsilon\right)^{1/(q+2)}.
\]

In a similar way, we suppose now that \(h/(2L)>r_0\), then \(u\ge h/2\) on \(B_{r_0}(y_0)\), and therefore
\[
\varepsilon
\ge \frac{h^2}{4}\nu(B_{r_0}(y_0))
\ge \frac{h^2}{4}m r_0^q.
\]
Thus, in the second case, we obtain
\[
h\le
2\left(\frac{\varepsilon}{m r_0^q}\right)^{1/2}.
\]
Combining the two cases gives the stated estimate.
\end{proof}

\section{Finite-sample geometric coercivity and controlled semiconjugacy}
\label{sec:main}

This section proves the principal finite-sample implication.  The main text retains the comparator, margin-clearing, empirical-process, and representation arguments that drive the theorem.  Detailed spline localization, coefficient stability, and architecture-specific covering-number calculations are provided in Supplementary Sections~S1--S2.

\begin{table}[t]
\centering
\small
\begin{tabular}{P{0.27\textwidth}P{0.66\textwidth}}
\toprule
\textbf{Hypothesis} & \textbf{Exact role in the proof}\\
\midrule
\hyperref[ass:A1]{A1}: regular observable Markov factor & Supplies a finite-dimensional Euclidean observable state factor, a bi-Lipschitz observation embedding, and descended dynamics and costs.\\[1mm]
Smooth realizability hypotheses of Lemma~\ref{lem:smooth-realizer} & Supply a $C^{1,1}$ bi-Lipschitz latent embedding, a smooth range-enforcing map, and hence the exact non-collapsed comparator used in the objective comparison.\\[1mm]
\hyperref[ass:A2]{A2}: lower-Ahlfors dynamics coverage & Converts the population $L^2$ prediction residual into a uniform semiconjugacy defect through Lemma~\ref{lem:l2-linf}.\\[1mm]
\hyperref[ass:A3]{A3}: tangent and separated-pair coverage & Converts population local and global hinge defects into uniform directional and separated-pair margins.\\[1mm]
\hyperref[ass:A4]{A4}: finite-capacity approximation and uniform budgets & Provides a margin-clearing comparator, compactness/equi-Lipschitz control, and the second-order remainder estimate for short chords.\\[1mm]
Approximate empirical global minimization & Compares the learned objective with the zero-metric-penalty comparator and turns the regularization threshold into population defect bounds.\\
\bottomrule
\end{tabular}
\caption{Assumption-to-conclusion map.  The hypotheses provide geometry, regularity, coverage, and optimization control; the theorem proves stability of all approximate empirical minimizers within that regime.}
\label{tab:assumption-proof-map}
\end{table}

\subsection{A smooth non-collapsed comparator and a concrete finite-capacity class}
The representation theorem requires a faithful comparator and finite-capacity classes that approximate it without losing regularity. Supplementary Section~S1 contains the affine latent embedding, local inverse charts, extension and range-enforcement arguments, tensor-product spline construction, compactness proof, coefficient counts, margin-clearing proof, and nonempty-parameter calculation; Supplementary Section~S2 contains the architecture-specific covering estimates.

\begin{lemma}[Exact non-collapsed Lipschitz realization]\label{lem:exact-Lip}
Under Assumption~\hyperref[ass:A1]{A1} and the standing conditions $d_z\ge d_s$ and $\operatorname{int}K_Z\ne\varnothing$, there are a smooth affine bi-Lipschitz embedding $\psi_0:S\to\operatorname{int}K_Z$, a $C^{1,1}$ range-enforcing map $\Pi_Z:Z\to K_Z$ that is the identity near $\psi_0(S)$, and Lipschitz maps $\Phi_0,F_0$ such that $\Phi_0\circ H=\psi_0$ and $F_0(\psi_0(s),a)=\psi_0(G(s,a))$.
\end{lemma}
\begin{lemma}[Quantitative inverse regularity]\label{lem:inverse-C11}
Every $C^{1,1}$ map on a neighborhood of a compact convex Euclidean domain that is quantitatively co-Lipschitz on the domain admits finitely many $C^{1,1}$ inverse charts near its compact image, with constants determined by the co-Lipschitz constant and the fixed-neighborhood $C^{1,1}$ budget.
\end{lemma}
\begin{lemma}[$C^{1,1}$ exact representatives]\label{lem:smooth-realizer}
Assume \hyperref[ass:A1]{A1} and $X=\mathbb R^{d_x}$, and use the affine embedding and smooth range-enforcing map furnished by Lemma~\ref{lem:exact-Lip}.  Then the exact realization may be chosen with $\Phi_0\in C^{1,1}(X;K_Z)$ and $F_0\in C^{1,1}(K_Z\times A;K_Z)$, while retaining the exact identities on $H(S)$ and $\psi_0(S)\times A$.
\end{lemma}
\begin{assumption}{A4}{Finite-capacity $C^1$ approximation with uniform $C^{1,1}$ budgets}
The exact pair is approximated by closed classes $\mathcal A_\Phi(W)\times\mathcal A_F(W)$ mapping into $K_Z$, with capacity-independent Lipschitz and $C^{1,1}$ budgets and
\[
\|\Phi_W-\Phi_0\|_{C^1(H(S))}+\|F_W-F_0\|_{C^1(\psi_0(S)\times A)}\le\delta(W),\qquad \delta(W)\to0.
\]
\end{assumption}
\begin{proposition}[Concrete finite-capacity realization]\label{prop:spline-A4}
For $X=\mathbb R^{d_x}$ and the exact representatives and range-enforcing map of Lemma~\ref{lem:smooth-realizer}, norm-constrained tensor-product cubic B-spline classes on fixed rectangular neighborhoods of $H(S)$ and $K_Z\times A$ satisfy Assumption~\hyperref[ass:A4]{A4}. They may be chosen so that
\[
\delta(W)\le C_\Phi W^{-1/d_x}+C_F W^{-1/(d_z+d_a)},
\]
their Lipschitz and $C^{1,1}$ budgets are independent of $W$, and each class is compact in $C^1$.
\end{proposition}
\begin{proposition}[Margin-clearing competitors]\label{prop:margin-clearing}
Choose $0<\kappa<\inf_{\mathcal U}\|D\psi_0(s)[v]\|_Z$ and $0<\alpha<\inf_{P_\rho}\|\psi_0(s)-\psi_0(s')\|_Z$. For all sufficiently large $W$, the approximants clear both margins, have zero population and empirical metric penalty, and satisfy
\[
\mathcal E_{\rm pred}(\Phi_W,F_W)\le\beta(W):=C_\sharp\delta(W)^2\to0.
\]
\end{proposition}
\begin{lemma}[Nonempty parameter regime]\label{lem:nonempty-parameters}
Let $\kappa_0:=\inf_{(s,v)\in\mathcal U}\|D\psi_0(s)[v]\|_Z>0$ and let $B_{\rm loc}$ be the capacity-uniform second-order budget.  Every $\kappa\in(0,\kappa_0)$ admits
\[
0<\rho<\min\!\left\{\operatorname{diam}(S),\frac{\kappa}{4B_{\rm loc}}\right\},
\qquad
0<\alpha<\inf_{P_\rho}\|\psi_0(s)-\psi_0(s')\|_Z,
\]
with the usual $+\infty$ convention when $B_{\rm loc}=0$.  Hence the geometric margin regime required by Theorem~\ref{thm:representation} is nonempty.
\end{lemma}
These statements supply the only construction input used below: a zero-penalty comparator with prediction error $\beta(W)$ and uniform regularity. Their complete statements and proofs are in Supplementary Section~S1; the architecture-specific entropy estimates are in Supplementary Section~S2.

\subsection{Uniform statistical control}
\label{sub:Stat-devi}

Fix $\kappa,\alpha>0$ and define
\[
\mathcal G_M(W)=\{(s,a)\mapsto\|R[\Phi,F](s,a)\|_Z^2:(\Phi,F)\in\mathcal A(W)\},
\]
\[
\mathcal G_U(W)=\{(s,v)\mapsto[\kappa-\|D(\Phi\circ H)(s)[v]\|_Z]_+^2:\Phi\in\mathcal A_\Phi(W)\},
\]
and
\[
\mathcal G_P(W)=\{(s,s')\mapsto[\alpha-\|(\Phi\circ H)(s)-(\Phi\circ H)(s')\|_Z]_+^2:\Phi\in\mathcal A_\Phi(W)\}.
\]
Let $N_M(r;W),N_U(r;W),N_P(r;W)$ denote their $C^0$ covering numbers.  The generic deviation argument is retained here; the spline coefficient and parameter-Lipschitz verification is in Supplementary Section~S2.
\begin{lemma}[Equi-Lipschitz integrand classes]
\label{lem:equi-lipschitz-classes}
Adopt the uniform \(C^{1,1}\) budgets of Assumption~\hyperref[ass:A4]{A4} and the standing Lipschitz regularity of \(G\) and \(H\). By Assumption~\hyperref[ass:A4]{A4}, there are constants \(C_\psi,L_F^{\mathrm{bud}}<\infty\), independent of \(W\), such that for every \(\Phi\in\mathcal A_\Phi(W)\), writing \(\psi:=\Phi\circ H\), we have
\[
\|\psi\|_{C^0(S)}+\|D\psi\|_{C^0(S)}+\Lip(D\psi;S)\le C_\psi,
\]
and for every \(F\in\mathcal A_F(W)\),
\[
F:K_Z\times A\to K_Z,\qquad \Lip(F)\le L_F^{\mathrm{bud}}.
\]
Then \(\mathcal G_M(W),\mathcal G_U(W),\mathcal G_P(W)\) are uniformly bounded and equi-Lipschitz on their respective compact domains: there exist constants
\[B_M,B_U,B_P<\infty\quad \text{and} \quad\widetilde L_M,\widetilde L_U,\widetilde L_P<\infty,\]
independent of \(W\), such that every \(g\in\mathcal G_Y(W)\) satisfies
\[
0\le g\le B_Y,\qquad \Lip(g)\le\widetilde L_Y,\qquad Y\in\{M,U,P\}.
\]
Consequently, \(N_Y(r;W)\) is finite for every \(r>0\) and \(Y\in\{M,U,P\}\).

Furthermore, there exist constants \(\overline{L}_M, \overline{L}_U, \overline{L}_P<\infty\) such that
\[\Lip(\sqrt{g})\le \overline{L}_Y,\qquad \text{for all }g\in \mathcal G_Y(W),\ Y\in\{M,U,P\}.\]
\end{lemma}
\begin{proof}
	We first treat the case \(Y=M\).

Since \(K_Z\subset Z\) is compact, we set
\[
	D_Z:={\rm diam}\,(K_Z)<\infty.
\]
For every admissible encoder \(\Phi\), the induced map
\[
\psi=\Phi\circ H:S\to K_Z
\]
takes values in \(K_Z\). Hence
\[
\|\psi(s)-\psi(s')\|_Z\le D_Z
\qquad \text{for all } s,s'\in S.
\]
The uniform \(C^{1,1}\) budget gives constants, independent of \(W\), such that
\[
\Lip(\psi)\le L_\psi,
\qquad
\|D\psi\|_{C^0(S)}\le B_D,
\qquad
\Lip(D\psi)\le L_D.
\]
We may take these constants to be bounded by \(C_\psi\).

	Let
\[
g_M(s,a)
:=
\|R[\Phi,F](s,a)\|_Z^2,
\]
where
\[
R[\Phi,F](s,a)
=
\psi(G(s,a))-F(\psi(s),a).
\]
Since both \(\psi(G(s,a))\) and \(F(\psi(s),a)\) lie in \(K_Z\), we have
\[
	\|R[\Phi,F](s,a)\|_Z\le D_Z
	\qquad \text{for all } (s,a)\in M.
\]
Thus, we obtain the bound
\[
	0\le g_M(s,a)\le D_Z^2.
\]
We may therefore take
\[
B_M:=D_Z^2.
\]

We next prove a uniform Lipschitz bound for \(R[\Phi,F]\). Let
\[
p=(s,a),
\qquad
p'=(s',a').
\]
Then, by the triangle inequality, we estimate
\[
\|R[\Phi,F](p)-R[\Phi,F](p')\|_Z
\le
\|\psi(G(s,a))-\psi(G(s',a'))\|_Z
+
\|
F(\psi(s),a)-F(\psi(s'),a')
\|_Z.
\]
For the first term, using the Lipschitz bounds for \(\psi\) and \(G\), we have
\[
\begin{aligned}
\|\psi(G(s,a))-\psi(G(s',a'))\|_Z
&\le
L_\psi\, d_S(G(s,a),G(s',a'))  \\
&\le
L_\psi \Lip(G)\, d_M(p,p').
\end{aligned}
\]
For the second term, using the Lipschitz budget of \(F\), we have
\[
\begin{aligned}
\|
F(\psi(s),a)-F(\psi(s'),a')
\|_Z
&\le
L_F^{\mathrm{bud}}
\bigl(
\|\psi(s)-\psi(s')\|_Z+d_A(a,a')
\bigr)  \\
&\le
L_F^{\mathrm{bud}}
\bigl(
	L_\psi |s-s'|+d_A(a,a')
\bigr)  \\
&\le
L_F^{\mathrm{bud}}\max\{L_\psi,1\}
d_M(p,p').
\end{aligned}
\]
Therefore, we arrive at
\[
\Lip(R[\Phi,F])
\le
L_R,
\]
where
\[
L_R
:=
L_\psi \Lip(G)
+
L_F^{\mathrm{bud}}\max\{L_\psi,1\}.
\]
The constant \(L_R\) is independent of \(W\).

Now, we use the elementary inequality
\[
\bigl|\|x\|_Z^2-\|y\|_Z^2\bigr|
\le
(\|x\|_Z+\|y\|_Z)\|x-y\|_Z.
\]
Since \(\|R[\Phi,F]\|\le D_Z\), we obtain
\[
\begin{aligned}
|g_M(p)-g_M(p')|
&=
\left|
\|R[\Phi,F](p)\|_Z^2
-
\|R[\Phi,F](p')\|_Z^2
\right| \\
&\le
2D_Z
\|R[\Phi,F](p)-R[\Phi,F](p')\|_Z  \\
&\le
2D_Z L_R\,d_M(p,p').
\end{aligned}
\]
By the reverse triangle inequality, we have
\[
|\sqrt{g_M(p)}-\sqrt{g_M(p')}|
=
\bigl|\|R[\Phi,F](p)\|_Z-\|R[\Phi,F](p')\|_Z\bigr|
\le
\|R[\Phi,F](p)-R[\Phi,F](p')\|_Z
\le L_R d_M(p,p').
\]
Hence every element of \(\mathcal G_M(W)\) satisfies
\[
\Lip(g_M)\le \widetilde L_M,
\qquad
\widetilde L_M:=2({\rm diam}\, K_Z) L_R,
\qquad
\overline{L}_M:= L_R.
\]

The cases \(Y=U,P\) are analogous, using that \(t\mapsto[c-t]_+^2\) is \(2c\)-Lipschitz on \([0,\infty)\) and that \((s,v)\mapsto\|D\psi(s)[v]\|_Z\) and \((s,s')\mapsto\|\psi(s)-\psi(s')\|_Z\) are Lipschitz with constants \(\max\{B_D,L_D\}\) and \(L_\psi\), respectively.
The uniform bounds and the Lipschitz budgets are
\begin{align*}
	&B_U = \kappa^2,\quad \widetilde{L}_U = 2\kappa \max\{B_D,L_D\},\quad \overline{L}_U=\max\{B_D,L_D\},\\
	&B_P = \alpha^2,\quad \widetilde{L}_P = 2\alpha L_\psi,\quad \overline{L}_P=L_\psi.
\end{align*}

Finally, we show that the \(C^0\)-covering numbers \(N_Y(r;W)\) are finite for every \(r>0\) and each \(Y\in\{M,U,P\}\).

Indeed, for any \(Y\in\{M,U,P\}\), the class \(\mathcal G_Y(W)\) is uniformly bounded and equi-Lipschitz on the corresponding compact domain, with constants independent of \(W\).
By the Arzel\`a--Ascoli theorem, any such class is relatively compact in
\(C^0(Y)\). In particular, it is totally bounded in the \(C^0\)-norm. Therefore, for
every \(r>0\), there exists a finite \(r\)-net in \(C^0(Y)\), and hence \(N_Y(r;W)<\infty\).
\end{proof}
\begin{lemma}[Uniform statistical deviation]
\label{lem:uniform-stat}
Assume the hypotheses of Lemma~\ref{lem:equi-lipschitz-classes}. For each $Y\in\{M,U,P\}$, choose in advance a deterministic radius
\[
r_Y=r_Y(n,\delta,W)\in(0,B_Y),
\]
and define
\[
\varepsilon_Y(n,\delta;W)
:=
2r_Y
+
B_Y\sqrt{
\frac{\log(6N_Y(r_Y;W)/\delta)}{2n}
}.
\]
Set
\[
\varepsilon_{{\rm stat},W}(n,\delta)
:=
\max\left\{
\varepsilon_M(n,\delta;W),
\varepsilon_U(n,\delta;W)+\varepsilon_P(n,\delta;W)
\right\}.
\]
Then, with probability at least $1-\delta$, simultaneously for all $(\Phi,F)\in\mathcal A(W)$,
\[
\left|
\widehat{\mathcal E_{\rm pred}}_n(\Phi,F)
-
\mathcal E_{\rm pred}(\Phi,F)
\right|
\le
\varepsilon_{{\rm stat},W}(n,\delta),
\]
and
\[
\left|
\widehat{\mathcal N_{\rm met}}_n(\Phi)
-
\mathcal N_{\rm met}(\Phi)
\right|
\le
\varepsilon_{{\rm stat},W}(n,\delta).
\]
For fixed $W$ and $\delta$, one may choose deterministic radii $r_Y(n)\downarrow0$ slowly enough that
\[
\frac{\log N_Y(r_Y(n);W)}{n}\to0,
\]
in which case $\varepsilon_{{\rm stat},W}(n,\delta)\to0$.
\end{lemma}
\begin{proof}
Fix $Y\in\{M,U,P\}$ and let $\{g_1,\ldots,g_N\}$ be an $r_Y$-net of $\mathcal G_Y(W)$ in $C^0(Y)$, where $N=N_Y(r_Y;W)$. If $\|g-g_j\|_{C^0(Y)}\le r_Y$, then
\[
| (\mu_{Y,n}-\mu_Y)g |
\le
| (\mu_{Y,n}-\mu_Y)g_j |+2r_Y.
\]
Since $0\le g_j\le B_Y$, Hoeffding's inequality and a union bound give
\[
\mathbb P\!\left(
\max_{1\le j\le N}|(\mu_{Y,n}-\mu_Y)g_j|>t
\right)
\le
2N\exp\!\left(-\frac{2nt^2}{B_Y^2}\right).
\]
Taking
\[
t=B_Y\sqrt{\frac{\log(6N_Y(r_Y;W)/\delta)}{2n}}
\]
shows that the stated bound for class $Y$ fails with probability at most $\delta/3$. A union bound over $Y=M,U,P$ yields the simultaneous event. The prediction loss is the integral of an element of $\mathcal G_M(W)$, while the two metric penalties are integrals of elements of $\mathcal G_U(W)$ and $\mathcal G_P(W)$; the displayed estimates follow by addition.

For fixed $W$, total boundedness gives $N_Y(r;W)<\infty$ for every $r>0$. A diagonal choice of deterministic radii can therefore be made so that $r_Y(n)\downarrow0$ and $\log N_Y(r_Y(n);W)=o(n)$, proving convergence.
\end{proof}
\begin{corollary}[Finite-parametric estimate]
\label{cor:parametric-estimate}
Assume the hypotheses of Lemma~\ref{lem:uniform-stat}. Since the three classes \(\mathcal G_M(W)\), \(\mathcal G_U(W)\), \(\mathcal G_P(W)\) are all determined by the parameters of the pair \((\Phi,F)\in\mathcal A(W)\), they share a common parametrization
\[
\theta\in\Theta_W\subset[-R_W,R_W]^{p_W}.
\]
Suppose that, for each \(Y\in\{M,U,P\}\), there is a constant \(L_{\theta,Y,W}>0\)
with
\[
\|g_\theta-g_{\theta'}\|_{C^0(Y)}\le L_{\theta,Y,W}\|\theta-\theta'\|_{\ell^\infty}
\qquad\text{for all }\theta,\theta'\in\Theta_W.
\]

Then, for every \(r>0\),
\[
N_Y(r;W)\le\Bigl(1+\tfrac{2R_W L_{\theta,Y,W}}{r}\Bigr)^{p_W},
\]
and hence, for every deterministic choice $r_Y\in(0,B_Y)$,
\[
\varepsilon_Y(n,\delta;W)
\le
2r_Y+B_Y\sqrt{\frac{p_W\log\bigl(1+\tfrac{2R_W L_{\theta,Y,W}}{r_Y}\bigr)+\log(6/\delta)}{2n}}.
\]
In particular, taking \(r_Y=B_Y n^{-1/2}\) gives, for \(n>1\),
\[
\varepsilon_Y(n,\delta;W)
\le\frac{2B_Y}{\sqrt n}
+B_Y\sqrt{\frac{p_W\log\bigl(1+\tfrac{2R_W L_{\theta,Y,W}\sqrt n}{B_Y}\bigr)+\log(6/\delta)}{2n}}.
\]
If \(R_W L_{\theta,Y,W}\le C_Y W^{q_Y}\) for constants \(C_Y,q_Y\) independent of \(n\), then
\[
\varepsilon_Y(n,\delta;W)
=O\!\left(\sqrt{\frac{p_W(\log n+\log W)+\log(1/\delta)}{n}}\right).
\]
The same rate holds for
\(\varepsilon_{{\rm stat},W}(n,\delta)=\max\{\varepsilon_M,\ \varepsilon_U+\varepsilon_P\}\):
\[
\varepsilon_{{\rm stat},W}(n,\delta)
=O\!\left(\sqrt{\frac{p_W(\log n+\log W)+\log(1/\delta)}{n}}\right).
\]
In particular, if \(W\) is fixed, or if the capacity schedule satisfies
\(W=W(n)\le Cn^q\) for some \(C,q>0\), then \(\log W=O(\log n)\) and
\[
\varepsilon_{{\rm stat},W}(n,\delta)
=O\!\left(\sqrt{\frac{p_W\log n+\log(1/\delta)}{n}}\right).
\]
\end{corollary}
\begin{proof}
Fix \(Y\in\{M,U,P\}\); we bound the covering number of \(\mathcal G_Y(W)\) in the \(C^0(Y)\)-norm.

We partition each coordinate interval $[-R_W,R_W]$ into at most
\[
1+\frac{2R_WL_{\theta,Y,W}}{r}
\]
subintervals of length at most $r/L_{\theta,Y,W}$. This partitions the parameter box into at most
\[
N:=\left(1+\frac{2R_WL_{\theta,Y,W}}{r}\right)^{p_W}
\]
cells. For every cell that intersects $\Theta_W$, choose one representative $\theta_j\in\Theta_W$ from that intersection. Any two points in the same cell are at $\ell^\infty$-distance at most $r/L_{\theta,Y,W}$. Hence, for each $\theta\in\Theta_W$, the representative of its cell satisfies
\[
\|g_\theta-g_{\theta_j}\|_{C^0(Y)}
\le
L_{\theta,Y,W}\|\theta-\theta_j\|_{\ell^\infty}
\le r.
\]
Therefore the selected $g_{\theta_j}$ form an $r$-net of \(\mathcal G_Y(W)\) in \(C^0(Y)\), so
\[
N_Y(r;W)\le N
\le\Bigl(1+\frac{2R_W L_{\theta,Y,W}}{r}\Bigr)^{p_W}.
\]
Substituting this bound into Lemma~\ref{lem:uniform-stat} gives the stated bound on \(\varepsilon_Y(n,\delta;W)\); the remaining estimates follow by the elementary computations in the statement.
\end{proof}

\subsection{Principal finite-sample representation theorem}

We now fix the three interpolation moduli used in the estimates. As in
Subsection~\ref{sub:Stat-devi}, we consider the function classes
\(\mathcal G_M(W)\), \(\mathcal G_U(W)\), and \(\mathcal G_P(W)\) with fixed
\(\kappa,\alpha>0\). By Lemma~\ref{lem:equi-lipschitz-classes}, let
\(\overline L_M\), \(\overline L_U\), and \(\overline L_P\) denote the
uniform Lipschitz bounds for the square-root integrands
\[
\|R[\Phi,F]\|_Z,
\qquad
[\kappa-\|D(\Phi\circ H)(s)[v]\|_Z]_+,
\qquad
[\alpha-\|(\Phi\circ H)(s)-(\Phi\circ H)(s')\|_Z]_+,
\]
respectively. These are the functions to which the
Lipschitz--Ahlfors \(L^2\)-to-\(L^\infty\) lemma is applied.

The moduli \(\Theta_M,\Theta_{\mathcal U},\Theta_{P_\rho}\) are therefore the
maps furnished by Lemma~\ref{lem:l2-linf} with data
\[
(d_m,m_0,r_0,\overline L_M),
\qquad
(q_U,m_U,r_U,\overline L_U),
\qquad
(q_P,m_P,r_P,\overline L_P),
\]
respectively.

\begin{definition}[Approximate empirical minimizer]
\label{def:approx-empirical-minimizer}
For \(\lambda>0\), define
\[
\widehat{\mathcal J}_{\lambda,n}(\Phi,F)
:=
\widehat{\mathcal E_{\rm pred}}_n(\Phi,F)
+
\lambda\widehat{\mathcal N_{\rm met}}_n(\Phi).
\]
Given \(\varepsilon_{\rm train}\ge0\), a pair
\((\Phi^\star,F^\star)\in\mathcal A(W)\) is called an
\(\varepsilon_{\rm train}\)-approximate empirical minimizer if
\[
\widehat{\mathcal J}_{\lambda,n}(\Phi^\star,F^\star)
\le
\inf_{(\Phi,F)\in\mathcal A(W)}
\widehat{\mathcal J}_{\lambda,n}(\Phi,F)
+
\varepsilon_{\rm train}.
\]
The case \(\varepsilon_{\rm train}=0\) is the exact global-minimizer case.
\end{definition}

\begin{theorem}[Finite-sample metric non-collapse and controlled semiconjugacy]
\label{thm:representation}
Assume \hyperref[ass:A1]{A1}, \hyperref[ass:A2]{A2}, \hyperref[ass:A4]{A4}, the hypotheses of Lemma~\ref{lem:smooth-realizer}, and the statistical setup of Lemmas~\ref{lem:equi-lipschitz-classes}--\ref{lem:uniform-stat} at capacity $W$.  Fix $(\kappa,\rho,\alpha)$ as in Lemma~\ref{lem:nonempty-parameters}, and assume \hyperref[ass:A3]{A3} for this $\rho$.  In particular,
\[
0<\rho<\operatorname{diam}(S),\qquad B_{\rm loc}\rho\le\frac\kappa4,
\quad
B_{\rm loc}:=\sup_W\sup_{\Phi\in\mathcal A_\Phi(W)}\Lip(D(\Phi\circ H)).
\]
Let $(\Phi_\sharp,F_\sharp)$ be the margin-clearing competitor of Proposition~\ref{prop:margin-clearing}, with
\[
\mathcal E_{\rm pred}(\Phi_\sharp,F_\sharp)\le\beta(W).
\]
Let $\varepsilon_{{\rm stat},W}(n,\delta)$ be the error in Lemma~\ref{lem:uniform-stat}, let $\varepsilon_{\rm train}\ge0$, and set
\[
P^*:=\min\{\Theta_{\mathcal U}^{-1}(\kappa/2),\Theta_{P_\rho}^{-1}(\alpha/2)\}.
\]
Suppose $\varepsilon_{{\rm stat},W}(n,\delta)<P^*$ and choose
\begin{equation}
\label{eq:lambda-min-training}
\lambda\ge\lambda_{\min}(W,n,\delta;\varepsilon_{\rm train})
:=\frac{\beta(W)+\varepsilon_{{\rm stat},W}(n,\delta)+\varepsilon_{\rm train}}
{P^*-\varepsilon_{{\rm stat},W}(n,\delta)}.
\end{equation}
Then, with probability at least $1-\delta$, every $\varepsilon_{\rm train}$-approximate empirical minimizer $(\Phi^\star,F^\star)$ satisfies
\begin{equation}
\label{eq:main-colipschitz}
\|\Phi^\star(H(s))-\Phi^\star(H(s'))\|_Z\ge c_*d_S(s,s')
\qquad(s,s'\in S),
\end{equation}
where
\[
c_*:=\min\left\{\frac\kappa4,\frac\alpha{2\operatorname{diam}(S)}\right\}>0,
\]
and
\begin{equation}
\label{eq:main-semiconjugacy}
\sup_{(s,a)\in M}\|\Phi^\star(H(G(s,a)))-F^\star(\Phi^\star(H(s)),a)\|_Z
\le \eta_{W,n,\delta}(\varepsilon_{\rm train}),
\end{equation}
with
\begin{equation}
\label{eq:eta-with-training-error}
\eta_{W,n,\delta}(\varepsilon_{\rm train})
:=\Theta_M\!\left(\beta(W)+2\varepsilon_{{\rm stat},W}(n,\delta)+\varepsilon_{\rm train}\right).
\end{equation}
If $W=W(n)$ and the approximation, statistical, and training errors vanish, then $\eta_{W(n),n,\delta}(\varepsilon_{\rm train}(n))\to0$.  In the finite-parametric setting of Corollary~\ref{cor:parametric-estimate}, it is enough that
\[
\beta(W(n))\to0,\qquad \varepsilon_{\rm train}(n)\to0,
\qquad \frac{p_{W(n)}(\log n+\log W(n))}{n}\to0.
\]
\end{theorem}

\begin{proof}
Write $\varepsilon:=\varepsilon_{{\rm stat},W}(n,\delta)$.  On the event in Lemma~\ref{lem:uniform-stat}, the prediction and metric empirical/population deviations are at most $\varepsilon$ uniformly over $\mathcal A(W)$.  Approximate optimality against the margin-clearing comparator gives
\begin{equation}
\label{eq:population-errors-with-training}
\mathcal E_{\rm pred}(\Phi^\star,F^\star)
\le\beta(W)+2\varepsilon+\varepsilon_{\rm train},
\qquad
\mathcal N_{\rm met}(\Phi^\star)
\le
\frac{\beta(W)+\varepsilon+\varepsilon_{\rm train}}{\lambda}+\varepsilon.
\end{equation}
Equation~\eqref{eq:lambda-min-training} makes the second bound at most $P^*$.

Set $\psi^\star:=\Phi^\star\circ H$.  The local defect
\[
q_{\rm loc}^\star(s,v):=[\kappa-\|D\psi^\star(s)[v]\|_Z]_+
\]
is uniformly Lipschitz and has squared $L^2$ norm at most $P^*$.  Lemma~\ref{lem:l2-linf} and the definition of $P^*$ yield
\begin{equation}
\label{eq:local-derivative-margin-trained}
\|D\psi^\star(s)[v]\|_Z\ge\frac\kappa2
\qquad((s,v)\in\mathcal U).
\end{equation}
The same argument for the separated-pair defect gives
\begin{equation}
\label{eq:global-pair-margin-trained}
\|\psi^\star(s)-\psi^\star(s')\|_Z\ge\frac\alpha2
\qquad\text{when }d_S(s,s')\ge\rho.
\end{equation}
For pairs at distance at least $\rho$, this implies the second term in $c_*$.  For $h:=|s'-s|<\rho$, let $e=(s'-s)/h$ and $\gamma(t)=s+t(s'-s)$.  Convexity, the fundamental theorem of calculus, \eqref{eq:local-derivative-margin-trained}, and the second-order budget give
\[
\begin{aligned}
\|\psi^\star(s')-\psi^\star(s)\|_Z
&\ge h\|D\psi^\star(s)[e]\|_Z
-h\int_0^1\|(D\psi^\star(\gamma(t))-D\psi^\star(s))[e]\|_Z\,\dd t\\
&\ge\frac\kappa2h-\frac{B_{\rm loc}}2h^2
\ge\frac{3\kappa}{8}h\ge c_*h.
\end{aligned}
\]
This proves \eqref{eq:main-colipschitz}.

Finally, the residual norm is uniformly Lipschitz on $M$, and \eqref{eq:population-errors-with-training} bounds its squared population $L^2$ norm.  Assumption~\hyperref[ass:A2]{A2} and Lemma~\ref{lem:l2-linf} therefore give \eqref{eq:main-semiconjugacy}.  The convergence statements follow from $\Theta_M(r)\to0$ as $r\to0$ and Corollary~\ref{cor:parametric-estimate}.
\end{proof}

\begin{corollary}[Explicit admissible regularization range]
\label{cor:lambda-range}
Under the hypotheses of Theorem~\ref{thm:representation}, the admissible strengths form the nonempty half-line
\[
[\lambda_{\min}(W,n,\delta;\varepsilon_{\rm train}),\infty).
\]
There is no finite upper endpoint.  For fixed $W$ and $\varepsilon_{\rm train}$,
\[
\lambda_{\min}(W,n,\delta;\varepsilon_{\rm train})
\longrightarrow\frac{\beta(W)+\varepsilon_{\rm train}}{P^*}\qquad(n\to\infty).
\]
If all three errors vanish along $W=W(n)$, then $\lambda_{\min}\to0$, so every fixed $\lambda>0$ is eventually admissible.
\end{corollary}
\begin{proof}
The comparison uses only the lower bound on $\lambda$, while the comparator has zero empirical metric penalty.  The limits follow directly from the formula.
\end{proof}

\subsection{Sharpness of the interpolation exponent}
\label{sec:exponent}
The exponent supplied by Lipschitz--Ahlfors interpolation cannot be improved under the stated regularity alone.
\begin{proposition}[Sharpness of the $1/(q+2)$ exponent]\label{prop:sharp-exponent}
For $Y=[0,1]^q$ and fixed $L>0$, there are nonnegative $L$-Lipschitz functions $u_h$ with $\|u_h\|_\infty=h$ and $\int_Yu_h^2\le C_qL^{-q}h^{q+2}$. Consequently no uniform estimate with exponent larger than $1/(q+2)$ holds on this class.
\end{proposition}
The truncated-cone construction and calculation are given in Supplementary Section~S3.

\subsection{A finite-net a posteriori certificate and use of the threshold}
\label{sub:finite-net-certificate}
\label{sec:using-the-theorem}

After training, the theorem's a priori design guarantee is complemented by a model-specific route that certifies the uniform defects directly.

\begin{proposition}[Validated finite-net certificate]
\label{prop:finite-net-certificate}
Let $\psi=\Phi\circ H$ and $F$ satisfy the same compact-range and second-order budget as in Theorem~\ref{thm:representation}.  Define
\[
\mathfrak g_U(s,v):=[\kappa-\|D\psi(s)[v]\|_Z]_+,
\quad \mathfrak g_P(s,s'):=[\alpha-\|\psi(s)-\psi(s')\|_Z]_+,
\]
and $\mathfrak g_M(s,a):=\|\psi(G(s,a))-F(\psi(s),a)\|_Z$.  Assume
\[
\Lip(\mathfrak g_U)\le L_U,\qquad \Lip(\mathfrak g_P)\le L_P,\qquad \Lip(\mathfrak g_M)\le L_M.
\]
Let $\mathcal Q_U,\mathcal Q_P,\mathcal Q_M$ be finite nets with fill distances $h_U,h_P,h_M$.  If
\[
\max_{\mathcal Q_U}\mathfrak g_U+L_Uh_U\le\frac\kappa2,
\qquad
\max_{\mathcal Q_P}\mathfrak g_P+L_Ph_P\le\frac\alpha2,
\qquad B_{\rm loc}\rho\le\frac\kappa4,
\]
then
\[
\|\psi(s)-\psi(s')\|_Z\ge
\min\left\{\frac\kappa4,\frac\alpha{2\operatorname{diam}(S)}\right\}d_S(s,s')
\quad(s,s'\in S).
\]
Moreover,
\[
\sup_{(s,a)\in M}\mathfrak g_M(s,a)\le\max_{\mathcal Q_M}\mathfrak g_M+L_Mh_M.
\]
Consequently, the deterministic planning bounds of Corollary~\ref{cor:empirical-planning-transfer} apply with these certified constants.
\end{proposition}
\begin{proof}
Choose a nearest net point in each compact domain.  Lipschitz continuity bounds each continuum defect by its net maximum plus the corresponding fill-distance remainder.  The first two inequalities give the uniform margins used in the short-chord/long-chord proof of Theorem~\ref{thm:representation}; the third defect bound is the uniform semiconjugacy estimate.
\end{proof}

The analytic threshold and Proposition~\ref{prop:finite-net-certificate} give complementary certification routes.  The class-uniform formula for $\lambda_{\min}$ proves an a priori one-sided design half-line for every approximate empirical minimizer and displays the dependence on coverage, regularity, approximation, sampling, and optimization.  The finite-net proposition is model-specific and is typically sharper: validated interval bounds, exact spline or finite-element estimates, or controlled fill-distance remainders convert measured local, separated-pair, residual, and Lipschitz quantities into the constants used by the same deterministic planning corollary.

For calibration, one may sweep $\lambda$ upward and retain the smallest value for which held-out directional and separated-pair margins are stable across samples and restarts while prediction residual and upper distortion remain controlled.  This is a certificate-aligned diagnostic, not a monotonicity claim for nonconvex optimization.  It becomes a rigorous a posteriori certificate only after the Lipschitz and fill-distance remainders in Proposition~\ref{prop:finite-net-certificate} have been validated.

\section{Deterministic planning as a downstream corollary}
\label{sec:why-semiconjugacy}
\label{sec:simulation-lemma}

Theorem~\ref{thm:representation} supplies the representation and dynamics certificates.  Adding Lipschitz costs, a finite horizon, and a planner converts those certificates into deterministic trajectory, cost, and optimizer-transfer guarantees.  Isolating this modular step makes the reusable role of the learned geometric and semiconjugacy constants explicit.

For $T\in\mathbb N$ and $L\ge0$, define
\begin{equation}
\label{eq:CD_def}
C(T,L):=\sum_{j=0}^{T-1}L^j,
\qquad D(T,L):=\sum_{t=0}^{T-1}\sum_{j=0}^{t-1}L^j,
\end{equation}
where the inner sum is zero at $t=0$.  Put
\[
\Delta_{T,L}(\eta):=\eta[L_{\widetilde\ell,z}D(T,L)+L_{\widetilde g}C(T,L)],
\quad
\widehat\Delta_{T,L}(\eta):=\eta[L_{\widehat\ell,z}D(T,L)+L_{\widehat g}C(T,L)].
\]

\begin{lemma}[Deterministic simulation estimate]
\label{lem:simulation}
Let $\psi:S\to K_Z$ and $F:K_Z\times A\to K_Z$ satisfy
\[
\sup_{(s,a)}\|\psi(G(s,a))-F(\psi(s),a)\|_Z\le\eta,
\]
and suppose $F$ is $L_F$-Lipschitz in its latent-state argument.  For a common action sequence, let $s_{t+1}=G(s_t,a_t)$, $z_{t+1}=F(z_t,a_t)$, and $z_0=\psi(s_0)$.  Then
\[
\|\psi(s_t)-z_t\|_Z\le\eta C(t,L_F).
\]
If $\widetilde\ell$ and $\widetilde g$ are Lipschitz in the latent variable, the true-image and latent-rollout costs differ by at most $\Delta_{T,L_F}(\eta)$.
\end{lemma}
\begin{proof}
The recurrence $e_{t+1}\le\eta+L_Fe_t$, $e_0=0$, gives $e_t\le\eta C(t,L_F)$.  Summing the stage-cost discrepancies and adding the terminal discrepancy gives the stated bound.
\end{proof}

\begin{corollary}[Uniform finite-horizon planning transfer]
\label{cor:planning-transfer}
Assume $\psi$ is $c$-co-Lipschitz, the physical costs are Lipschitz in state, and let $\widetilde\ell,\widetilde g$ be the compatible latent extensions of Lemma~\ref{lem:latent-cost-extension}.  If $\widehat{\mathbf a}$ is $\xi_{\rm plan}$-optimal for the latent horizon-$T$ problem, then
\[
J_{\rm true}(s_0;\widehat{\mathbf a})-J^\star(s_0)
\le2\Delta_{T,L_F}(\eta)+\xi_{\rm plan},
\]
with $L_{\widetilde\ell,z}\le L_{\ell,s}/c$ and $L_{\widetilde g}\le L_g/c$.
For learned cost heads satisfying the compatibility bounds
\[
\sup_{s,a}|\widehat\ell(\psi(s),a)-\ell(s,a)|\le\varepsilon_\ell,
\qquad
\sup_s|\widehat g(\psi(s))-g(s)|\le\varepsilon_g,
\]
and the latent-state Lipschitz estimates
\[
|\widehat\ell(z,a)-\widehat\ell(z',a)|
\le L_{\widehat\ell,z}\|z-z'\|_Z,
\qquad
|\widehat g(z)-\widehat g(z')|
\le L_{\widehat g}\|z-z'\|_Z,
\]
the right-hand side becomes
\[
2[\widehat\Delta_{T,L_F}(\eta)+T\varepsilon_\ell+\varepsilon_g]+\xi_{\rm plan}.
\]
\end{corollary}
\begin{proof}
Lemma~\ref{lem:simulation} bounds the true/latent objective discrepancy uniformly over the common action space.  Add and subtract the true and latent costs of a latent optimizer and a true optimizer.  Approximate latent optimization adds $\xi_{\rm plan}$; learned-head compatibility contributes $T\varepsilon_\ell+\varepsilon_g$ before the two-sided optimizer comparison.
\end{proof}

\begin{corollary}[Finite-sample planning consequence]
\label{cor:empirical-planning-transfer}
Under Theorem~\ref{thm:representation}, assume the physical stage and terminal costs satisfy
\[
|\ell(s,a)-\ell(s',a')|\le L_{\ell,s}d_S(s,s')+L_{\ell,a}d_A(a,a'),
\qquad |g(s)-g(s')|\le L_gd_S(s,s').
\]
For every fixed horizon $T$ and every $\xi_{\rm plan}$-optimal latent action sequence,
\begin{equation}
\label{eq:main-theorem-planning-bound}
J_{\rm true}(s_0;\widehat{\mathbf a})-J^\star(s_0)
\le2\Delta_{T,L_F^{\rm bud}}(\eta_{W,n,\delta}(\varepsilon_{\rm train}))+\xi_{\rm plan}.
\end{equation}
For learned heads satisfying
\[
\sup_{s,a}|\widehat\ell(\psi^\star(s),a)-\ell(s,a)|\le\varepsilon_\ell,
\qquad \sup_s|\widehat g(\psi^\star(s))-g(s)|\le\varepsilon_g,
\]
and
\[
|\widehat\ell(z,a)-\widehat\ell(z',a)|
\le L_{\widehat\ell,z}\|z-z'\|_Z,
\qquad
|\widehat g(z)-\widehat g(z')|
\le L_{\widehat g}\|z-z'\|_Z
\]
on the relevant compact latent domain, one instead has
\begin{equation}
\label{eq:main-theorem-learned-cost-bound}
J_{\rm true}(s_0;\widehat{\mathbf a})-J^\star(s_0)
\le2[\widehat\Delta_{T,L_F^{\rm bud}}(\eta_{W,n,\delta}(\varepsilon_{\rm train}))+T\varepsilon_\ell+\varepsilon_g]+\xi_{\rm plan}.
\end{equation}
\end{corollary}
\begin{proof}
Combine Theorem~\ref{thm:representation}, Lemma~\ref{lem:latent-cost-extension}, and Corollary~\ref{cor:planning-transfer}.
\end{proof}

\paragraph{Horizon dependence.}
If $L_F^{\rm bud}=1$, then $C(T,1)=T$ and $D(T,1)=T(T-1)/2$, so the running-cost contribution is generally $O(\eta T^2/c_*)$.  If $L_F^{\rm bud}<1$, rollout error is bounded by $\eta/(1-L_F^{\rm bud})$; if $L_F^{\rm bud}>1$, the worst-case envelope grows geometrically.  These formulas furnish certified horizon-dependent envelopes; trajectory-local or contractive structure can produce substantially tighter realized errors.

\paragraph{Mean-to-uniform residual control.}
The simulation lemma requires a supremum residual, not only a population mean. A localized Lipschitz bump can have fixed pointwise height while its $L^2$ norm tends to zero with its support radius, and a deterministic optimizer may visit that high-error region. Theorem~\ref{thm:representation} closes this gap by combining regularity with coverage. The exact bump calculation, an adversarial-planner diagnostic, and additional Lyapunov and reachability consequences are in Supplementary Section~S3.

\IfFileExists{numbers_macros.tex}{
\newcommand{\AthreeSeeds}{12}
\newcommand{\AthreeSteps}{2000}
\newcommand{\AthreeRZ}{2.0}
\newcommand{\AthreeGrid}{4001}
\newcommand{\AEvenPureNoninj}{12}
\newcommand{\AEvenPureCmin}{0}
\newcommand{\AEvenPureCmax}{0.12}
\newcommand{\AEvenPureKgeom}{\ensuremath{1.20\!\times\!10^{11}}}
\newcommand{\AEvenPureMloc}{\ensuremath{1.20\!\times\!10^{-5}}}
\newcommand{\AEvenPureEta}{\ensuremath{5.34\!\times\!10^{-4}}}
\newcommand{\AEvenCovNoninj}{10}
\newcommand{\AEvenCovCmin}{\ensuremath{3.75\!\times\!10^{-7}}}
\newcommand{\AEvenCovCmax}{1.57}
\newcommand{\AEvenCovKgeom}{\ensuremath{4.37\!\times\!10^{6}}}
\newcommand{\AEvenCovMloc}{\ensuremath{2.32\!\times\!10^{-4}}}
\newcommand{\AEvenCovEta}{\ensuremath{6.24\!\times\!10^{-3}}}
\newcommand{\AEvenLocalNoninj}{9}
\newcommand{\AEvenLocalCmin}{\ensuremath{1.50\!\times\!10^{-7}}}
\newcommand{\AEvenLocalCmax}{3.82}
\newcommand{\AEvenLocalKgeom}{\ensuremath{3.49\!\times\!10^{7}}}
\newcommand{\AEvenLocalMloc}{0.02}
\newcommand{\AEvenLocalEta}{\ensuremath{4.84\!\times\!10^{-3}}}
\newcommand{\AEvenGlobalNoninj}{1}
\newcommand{\AEvenGlobalCmin}{0.82}
\newcommand{\AEvenGlobalCmax}{1.33}
\newcommand{\AEvenGlobalKgeom}{1.59}
\newcommand{\AEvenGlobalMloc}{0.82}
\newcommand{\AEvenGlobalEta}{\ensuremath{7.83\!\times\!10^{-3}}}
\newcommand{\AEvenHingeNoninj}{9}
\newcommand{\AEvenHingeCmin}{\ensuremath{3.36\!\times\!10^{-6}}}
\newcommand{\AEvenHingeCmax}{16.45}
\newcommand{\AEvenHingeKgeom}{\ensuremath{4.55\!\times\!10^{6}}}
\newcommand{\AEvenHingeMloc}{0.01}
\newcommand{\AEvenHingeEta}{0.01}
\newcommand{\AEvenHingeupNoninj}{3}
\newcommand{\AEvenHingeupCmin}{0.70}
\newcommand{\AEvenHingeupCmax}{1.36}
\newcommand{\AEvenHingeupKgeom}{1.97}
\newcommand{\AEvenHingeupMloc}{0.70}
\newcommand{\AEvenHingeupEta}{\ensuremath{8.06\!\times\!10^{-3}}}
\newcommand{\AEvenTwosidedNoninj}{3}
\newcommand{\AEvenTwosidedCmin}{0.46}
\newcommand{\AEvenTwosidedCmax}{0.55}
\newcommand{\AEvenTwosidedKgeom}{1.19}
\newcommand{\AEvenTwosidedMloc}{0.46}
\newcommand{\AEvenTwosidedEta}{\ensuremath{4.64\!\times\!10^{-3}}}
\newcommand{\APlainPureNoninj}{10}
\newcommand{\APlainPureCmin}{0}
\newcommand{\APlainPureCmax}{0.16}
\newcommand{\APlainPureKgeom}{\ensuremath{9.04\!\times\!10^{10}}}
\newcommand{\APlainPureMloc}{\ensuremath{1.62\!\times\!10^{-5}}}
\newcommand{\APlainPureEta}{\ensuremath{5.72\!\times\!10^{-4}}}
\newcommand{\APlainCovNoninj}{1}
\newcommand{\APlainCovCmin}{0.25}
\newcommand{\APlainCovCmax}{1.17}
\newcommand{\APlainCovKgeom}{5.19}
\newcommand{\APlainCovMloc}{0.25}
\newcommand{\APlainCovEta}{\ensuremath{9.25\!\times\!10^{-3}}}
\newcommand{\APlainLocalNoninj}{1}
\newcommand{\APlainLocalCmin}{0.50}
\newcommand{\APlainLocalCmax}{0.54}
\newcommand{\APlainLocalKgeom}{1.10}
\newcommand{\APlainLocalMloc}{0.50}
\newcommand{\APlainLocalEta}{\ensuremath{4.62\!\times\!10^{-3}}}
\newcommand{\APlainGlobalNoninj}{0}
\newcommand{\APlainGlobalCmin}{0.83}
\newcommand{\APlainGlobalCmax}{1.29}
\newcommand{\APlainGlobalKgeom}{1.55}
\newcommand{\APlainGlobalMloc}{0.83}
\newcommand{\APlainGlobalEta}{\ensuremath{9.57\!\times\!10^{-3}}}
\newcommand{\APlainHingeNoninj}{0}
\newcommand{\APlainHingeCmin}{0.90}
\newcommand{\APlainHingeCmax}{1.31}
\newcommand{\APlainHingeKgeom}{1.48}
\newcommand{\APlainHingeMloc}{0.90}
\newcommand{\APlainHingeEta}{\ensuremath{9.12\!\times\!10^{-3}}}
\newcommand{\APlainHingeupNoninj}{0}
\newcommand{\APlainHingeupCmin}{0.90}
\newcommand{\APlainHingeupCmax}{1.30}
\newcommand{\APlainHingeupKgeom}{1.48}
\newcommand{\APlainHingeupMloc}{0.90}
\newcommand{\APlainHingeupEta}{\ensuremath{9.12\!\times\!10^{-3}}}
\newcommand{\APlainTwosidedNoninj}{0}
\newcommand{\APlainTwosidedCmin}{0.49}
\newcommand{\APlainTwosidedCmax}{0.53}
\newcommand{\APlainTwosidedKgeom}{1.12}
\newcommand{\APlainTwosidedMloc}{0.49}
\newcommand{\APlainTwosidedEta}{\ensuremath{4.21\!\times\!10^{-3}}}
\newcommand{\AEvenHingeOkN}{3}
\newcommand{\AEvenHingeOkCmin}{0.86}
\newcommand{\AEvenHingeOkCmax}{1.27}
\newcommand{\AEvenHingeOkKgeom}{1.47}
\newcommand{\AEvenHingeBadN}{9}
\newcommand{\AEvenHingeBadCmin}{\ensuremath{2.59\!\times\!10^{-6}}}
\newcommand{\AEvenHingeBadCmax}{20.78}
\newcommand{\AEvenHingeBadKgeom}{\ensuremath{7.87\!\times\!10^{6}}}
\newcommand{\AEvenHingeupOkN}{9}
\newcommand{\AEvenHingeupOkCmin}{0.76}
\newcommand{\AEvenHingeupOkCmax}{1.34}
\newcommand{\AEvenHingeupOkKgeom}{1.78}
\newcommand{\AEvenHingeupBadN}{3}
\newcommand{\AEvenHingeupBadCmin}{\ensuremath{2.23\!\times\!10^{-6}}}
\newcommand{\AEvenHingeupBadCmax}{2.55}
\newcommand{\AEvenHingeupBadKgeom}{\ensuremath{1.15\!\times\!10^{6}}}
\newcommand{\BClearFirstN}{2048}
\newcommand{\BClearLams}{0.3}
\newcommand{\BClearNs}{2048}
\newcommand{\CpureRatioMin}{4.03}
\newcommand{\CpureRatioMax}{25.79}
\newcommand{\BPstar}{\ensuremath{4.81\!\times\!10^{-6}}}
\newcommand{\BEpsMin}{5.19}
\newcommand{\BEpsMax}{18.46}
\newcommand{\BBeta}{\ensuremath{7.01\!\times\!10^{-4}}}
\newcommand{\BRatioMin}{\ensuremath{1.08\!\times\!10^{6}}}
\newcommand{\BBestLam}{0.0}
\newcommand{\BBestEta}{\ensuremath{5.22\!\times\!10^{-8}}}
\newcommand{\BBestCmin}{0}
\newcommand{\BBestCleared}{0}
\newcommand{\BBestSeeds}{2}
\newcommand{\BZeroCmin}{0}
\newcommand{\BZeroEta}{\ensuremath{5.22\!\times\!10^{-8}}}
\newcommand{\BZeroNOneTwoEightCmin}{0}
\newcommand{\BZeroNOneTwoEightEta}{\ensuremath{4.31\!\times\!10^{-8}}}
\newcommand{\BZeroNOneTwoEightNloc}{0.16}
\newcommand{\BZeroNOneTwoEightCleared}{0}
\newcommand{\BZeroNOneTwoEightSeeds}{2}
\newcommand{\BLoNOneTwoEightCmin}{0.67}
\newcommand{\BLoNOneTwoEightEta}{0.11}
\newcommand{\BLoNOneTwoEightNloc}{0}
\newcommand{\BLoNOneTwoEightCleared}{0}
\newcommand{\BLoNOneTwoEightSeeds}{2}
\newcommand{\BMidNOneTwoEightCmin}{0.36}
\newcommand{\BMidNOneTwoEightEta}{0.37}
\newcommand{\BMidNOneTwoEightNloc}{\ensuremath{4.77\!\times\!10^{-3}}}
\newcommand{\BMidNOneTwoEightCleared}{0}
\newcommand{\BMidNOneTwoEightSeeds}{2}
\newcommand{\BZeroNFiveOneTwoCmin}{0}
\newcommand{\BZeroNFiveOneTwoEta}{\ensuremath{4.58\!\times\!10^{-8}}}
\newcommand{\BZeroNFiveOneTwoNloc}{0.16}
\newcommand{\BZeroNFiveOneTwoCleared}{0}
\newcommand{\BZeroNFiveOneTwoSeeds}{2}
\newcommand{\BLoNFiveOneTwoCmin}{0.71}
\newcommand{\BLoNFiveOneTwoEta}{0.01}
\newcommand{\BLoNFiveOneTwoNloc}{0}
\newcommand{\BLoNFiveOneTwoCleared}{0}
\newcommand{\BLoNFiveOneTwoSeeds}{2}
\newcommand{\BMidNFiveOneTwoCmin}{0.38}
\newcommand{\BMidNFiveOneTwoEta}{0.25}
\newcommand{\BMidNFiveOneTwoNloc}{\ensuremath{5.15\!\times\!10^{-3}}}
\newcommand{\BMidNFiveOneTwoCleared}{0}
\newcommand{\BMidNFiveOneTwoSeeds}{2}
\newcommand{\BZeroNTwoZeroFourEightCmin}{0}
\newcommand{\BZeroNTwoZeroFourEightEta}{\ensuremath{5.22\!\times\!10^{-8}}}
\newcommand{\BZeroNTwoZeroFourEightNloc}{0.16}
\newcommand{\BZeroNTwoZeroFourEightCleared}{0}
\newcommand{\BZeroNTwoZeroFourEightSeeds}{2}
\newcommand{\BLoNTwoZeroFourEightCmin}{0.68}
\newcommand{\BLoNTwoZeroFourEightEta}{0.02}
\newcommand{\BLoNTwoZeroFourEightNloc}{0}
\newcommand{\BLoNTwoZeroFourEightCleared}{0}
\newcommand{\BLoNTwoZeroFourEightSeeds}{2}
\newcommand{\BMidNTwoZeroFourEightCmin}{0.88}
\newcommand{\BMidNTwoZeroFourEightEta}{0.03}
\newcommand{\BMidNTwoZeroFourEightNloc}{0}
\newcommand{\BMidNTwoZeroFourEightCleared}{1}
\newcommand{\BMidNTwoZeroFourEightSeeds}{2}
\newcommand{\BSeeds}{2}
\newcommand{\CrefFailTFive}{0.69}
\newcommand{\CrefSpreadTFive}{0.07}
\newcommand{\CoptimizerSpreadTFive}{\ensuremath{9.98\!\times\!10^{-3}}}
\newcommand{\CrefFailTOneZero}{0.66}
\newcommand{\CrefSpreadTOneZero}{0.11}
\newcommand{\CoptimizerSpreadTOneZero}{0.05}
\newcommand{\CrefFailTTwoZero}{0.72}
\newcommand{\CrefSpreadTTwoZero}{0.28}
\newcommand{\CoptimizerSpreadTTwoZero}{0.28}
\newcommand{\CrefRestarts}{6}
\newcommand{\CrhoPooledRatio}{0.62}
\newcommand{\CrhoPooledBT}{-0.06}
\newcommand{\CrhoPooledEta}{-0.34}
\newcommand{\CrhoPooledInvC}{0.55}
\newcommand{\CrhoPooledRatioLo}{0.30}
\newcommand{\CrhoPooledRatioHi}{0.80}
\newcommand{\CrhoTFiveRatio}{0.51}
\newcommand{\CrhoTFiveBT}{-0.23}
\newcommand{\CrhoTFiveEta}{-0.35}
\newcommand{\CrhoTFiveInvC}{0.44}
\newcommand{\CrhoTTwoZeroRatio}{0.69}
\newcommand{\CrhoTTwoZeroBT}{0.15}
\newcommand{\CrhoTTwoZeroEta}{-0.45}
\newcommand{\CrhoTTwoZeroInvC}{0.65}
\newcommand{\CrhoNoPurePooledRatio}{0.42}
\newcommand{\CrhoNoPurePooledBT}{0.06}
\newcommand{\CrhoNoPurePooledEta}{0.07}
\newcommand{\CrhoNoPurePooledInvC}{0.30}
\newcommand{\CrhoNoPurePooledRatioLo}{0.04}
\newcommand{\CrhoNoPurePooledRatioHi}{0.69}
\newcommand{\CrhoNoPureTFiveRatio}{0.20}
\newcommand{\CrhoNoPureTFiveBT}{-0.21}
\newcommand{\CrhoNoPureTFiveEta}{0.09}
\newcommand{\CrhoNoPureTFiveInvC}{0.09}
\newcommand{\CrhoNoPureTTwoZeroRatio}{0.54}
\newcommand{\CrhoNoPureTTwoZeroBT}{0.59}
\newcommand{\CrhoNoPureTTwoZeroEta}{-0.11}
\newcommand{\CrhoNoPureTTwoZeroInvC}{0.48}
\newcommand{\CproxyLooseMed}{81.95}
\newcommand{\CproxyLooseMax}{\ensuremath{1.18\!\times\!10^{5}}}
\newcommand{\CchordMin}{0.21}
\newcommand{\CchordMax}{1.18}
\newcommand{\CsigMin}{0.20}
\newcommand{\CsigMax}{1.18}
\newcommand{\CSeeds}{[0, 1, 2]}
\newcommand{\CInits}{32}
\newcommand{\CPop}{48}
\newcommand{\CIters}{4}
\newcommand{\CClSteps}{12}
\newcommand{\CpureEta}{0.01}
\newcommand{\CpureCmin}{\ensuremath{1.75\!\times\!10^{-4}}}
\newcommand{\CpureKgeom}{177.21}
\newcommand{\CpureEpsL}{1.82}
\newcommand{\CcovEta}{0.13}
\newcommand{\CcovCmin}{0.19}
\newcommand{\CcovKgeom}{5.69}
\newcommand{\CcovEpsL}{0.29}
\newcommand{\ClocalEta}{0.08}
\newcommand{\ClocalCmin}{0.12}
\newcommand{\ClocalKgeom}{5.62}
\newcommand{\ClocalEpsL}{0.28}
\newcommand{\CglobalEta}{0.04}
\newcommand{\CglobalCmin}{0.06}
\newcommand{\CglobalKgeom}{4.15}
\newcommand{\CglobalEpsL}{0.71}
\newcommand{\ChingeEta}{0.09}
\newcommand{\ChingeCmin}{0.04}
\newcommand{\ChingeKgeom}{19.72}
\newcommand{\ChingeEpsL}{0.29}
\newcommand{\CtwosidedEta}{0.06}
\newcommand{\CtwosidedCmin}{0.14}
\newcommand{\CtwosidedKgeom}{3.92}
\newcommand{\CtwosidedEpsL}{0.34}
\newcommand{\CpuredJmedTFive}{1.41}
\newcommand{\CpuredJmaxTFive}{5.81}
\newcommand{\CpureFracNegTFive}{0}
\newcommand{\CpuredJoracleTFive}{0.68}
\newcommand{\CpuredJmedTOneZero}{1.54}
\newcommand{\CpuredJmaxTOneZero}{5.90}
\newcommand{\CpureFracNegTOneZero}{0}
\newcommand{\CpuredJoracleTOneZero}{0.63}
\newcommand{\CpuredJmedTTwoZero}{1.55}
\newcommand{\CpuredJmaxTTwoZero}{5.22}
\newcommand{\CpureFracNegTTwoZero}{0}
\newcommand{\CpuredJoracleTTwoZero}{0.50}
\newcommand{\CcovdJmedTFive}{0.10}
\newcommand{\CcovdJmaxTFive}{3.48}
\newcommand{\CcovFracNegTFive}{0.03}
\newcommand{\CcovdJoracleTFive}{0.06}
\newcommand{\CcovdJmedTOneZero}{0.10}
\newcommand{\CcovdJmaxTOneZero}{1.33}
\newcommand{\CcovFracNegTOneZero}{0.12}
\newcommand{\CcovdJoracleTOneZero}{0.07}
\newcommand{\CcovdJmedTTwoZero}{0.23}
\newcommand{\CcovdJmaxTTwoZero}{1.89}
\newcommand{\CcovFracNegTTwoZero}{0.06}
\newcommand{\CcovdJoracleTTwoZero}{0.10}
\newcommand{\ClocaldJmedTFive}{0.08}
\newcommand{\ClocaldJmaxTFive}{2.37}
\newcommand{\ClocalFracNegTFive}{0.22}
\newcommand{\ClocaldJoracleTFive}{0.06}
\newcommand{\ClocaldJmedTOneZero}{0.07}
\newcommand{\ClocaldJmaxTOneZero}{2.63}
\newcommand{\ClocalFracNegTOneZero}{0.16}
\newcommand{\ClocaldJoracleTOneZero}{0.06}
\newcommand{\ClocaldJmedTTwoZero}{0.38}
\newcommand{\ClocaldJmaxTTwoZero}{3.74}
\newcommand{\ClocalFracNegTTwoZero}{0.19}
\newcommand{\ClocaldJoracleTTwoZero}{0.13}
\newcommand{\CglobaldJmedTFive}{0.08}
\newcommand{\CglobaldJmaxTFive}{3.86}
\newcommand{\CglobalFracNegTFive}{0.09}
\newcommand{\CglobaldJoracleTFive}{0.05}
\newcommand{\CglobaldJmedTOneZero}{0.16}
\newcommand{\CglobaldJmaxTOneZero}{2.46}
\newcommand{\CglobalFracNegTOneZero}{0.09}
\newcommand{\CglobaldJoracleTOneZero}{0.07}
\newcommand{\CglobaldJmedTTwoZero}{0.20}
\newcommand{\CglobaldJmaxTTwoZero}{2.29}
\newcommand{\CglobalFracNegTTwoZero}{0.25}
\newcommand{\CglobaldJoracleTTwoZero}{0.10}
\newcommand{\ChingedJmedTFive}{0.20}
\newcommand{\ChingedJmaxTFive}{3.01}
\newcommand{\ChingeFracNegTFive}{0.06}
\newcommand{\ChingedJoracleTFive}{0.05}
\newcommand{\ChingedJmedTOneZero}{0.18}
\newcommand{\ChingedJmaxTOneZero}{3.00}
\newcommand{\ChingeFracNegTOneZero}{0.16}
\newcommand{\ChingedJoracleTOneZero}{0.06}
\newcommand{\ChingedJmedTTwoZero}{0.32}
\newcommand{\ChingedJmaxTTwoZero}{2.94}
\newcommand{\ChingeFracNegTTwoZero}{0.09}
\newcommand{\ChingedJoracleTTwoZero}{0.11}
\newcommand{\CtwosideddJmedTFive}{0.09}
\newcommand{\CtwosideddJmaxTFive}{1.49}
\newcommand{\CtwosidedFracNegTFive}{0.16}
\newcommand{\CtwosideddJoracleTFive}{0.06}
\newcommand{\CtwosideddJmedTOneZero}{0.06}
\newcommand{\CtwosideddJmaxTOneZero}{1.68}
\newcommand{\CtwosidedFracNegTOneZero}{0.16}
\newcommand{\CtwosideddJoracleTOneZero}{0.06}
\newcommand{\CtwosideddJmedTTwoZero}{0.18}
\newcommand{\CtwosideddJmaxTTwoZero}{1.32}
\newcommand{\CtwosidedFracNegTTwoZero}{0.19}
\newcommand{\CtwosideddJoracleTTwoZero}{0.11}
}{}
\section{Numerical evidence}
\label{sec:numerics}\label{sec:experiments}

The experiments provide three complementary forms of evidence for the mathematical mechanism: exact collapse and folding obstructions, a controlled finite-capacity study of threshold ingredients and margin formation, and a nonlinear latent-MPC benchmark linking representation geometry to true-system control.  The high-probability result is established analytically; the numerical study connects its principal quantities to observable finite-model behavior.  All trained suprema are evaluated on declared finite sets with adversarial refinement.  Supplementary Section~S4 provides the protocol, summarized tables, threshold ingredients, and additional figures; complete per-seed outputs, objective and gradient histories, evaluation arrays, and other supporting numerical documents are retained by the authors and will be provided upon request.

Every trained encoder uses the same compact output range, and all objectives share fixed metric arrays, matched initializations, and update counts.  Lower geometric diagnostics are reported together with upper chord and derivative diagnostics, so improved resolution is separated cleanly from scale inflation.  Multi-restart objective values, gradient diagnostics, and held-out geometric margins provide a common certificate-aligned view of optimization stability.

\begin{table}[t]\centering\footnotesize
\resizebox{\textwidth}{!}{%
\begin{tabular}{lccccccc}\toprule
experiment & $R_Z$ & upper control & $n_M/n_U/n_P$ & runs per objective & architecture & params & preprocessing time (s)\\\midrule
A3 folding & 2.0 & bounded range; upper hinge in objective (vi) & 512/512/256 & 12 seeds $\times$ 2 regimes & MLP $1$-$128$-$128$-$1$ & -- & --\\
B spline (main grid) & 2.0 & coefficient box $\pm3.0$ & $n$/$n$/$n$, $n\le2048$ & 2 seeds $\times$ 3 restarts & cubic B-spline & -- & --\\
B spline (capacity sweep) & 2.0 & coefficient box & $512$/$512$/$512$ & 1 seed(s) & cubic B-spline & -- & --\\
B bounded MLP & 2.0 & approx. PI spectral target 1.6 & $n$/$n$/$n$ & 2 seeds & MLP width 64 & -- & --\\
C pendulum + planning & 2.0 & approx. PI spectral target 1.5 & 2048/1024/512 & 3 seeds & MLP width 64 & 13896 & 20.64 [14.87, 50.40]\\
\bottomrule\end{tabular}
}
\caption{Protocol alignment. Every trained encoder in a given experiment maps into the same compact latent box $K_Z=[-R_Z,R_Z]^{d_z}$, uses the same fixed $D_M,D_U,D_P$ arrays, and receives the same number of updates. Run counts are read from the authors' retained result records, which will be provided upon request. The timing column is the median [min, max] over individual representation preprocessing runs, including representation training, cost-head fitting, and diagnostics but excluding MPC planning; it is not a total experiment wall-clock.}\label{tab:protocol}
\end{table}

\subsection{Exact and trained folding obstructions}
\label{sec:expA}
For $S=[-1,1]$, identity dynamics, and $\psi_C(s)=Cs^2$, the JEPA residual is zero and the latent variance is $4C^2/45$, yet $\psi_C(s)=\psi_C(-s)$.  In two dimensions,
\[
\psi_{C,D}(s_1,s_2)=(Cs_1^2,Ds_2)
\]
has full-rank covariance $\operatorname{diag}(4C^2/45,D^2/3)$ while remaining non-injective.  These exact examples establish the geometric obstruction that motivates the local--global penalty.

The bounded trained stress test compares pure prediction, covariance spread, local-only, global-only, the local--global hinge, the hinge plus an upper directional penalty, and a two-sided baseline.  Under plain initialization the combined hinge is injective in all $\AthreeSeeds$ runs.  Under strongly symmetry-biased even pretraining, adding upper control reduces the number of non-injective outcomes from $\AEvenHingeNoninj$ to $\AEvenHingeupNoninj$ out of $\AthreeSeeds$, demonstrating how the lower and upper diagnostics can guide optimization and architectural refinement even from a deliberately challenging initialization.

\begin{figure}[t]\centering
\includegraphics[width=\textwidth]{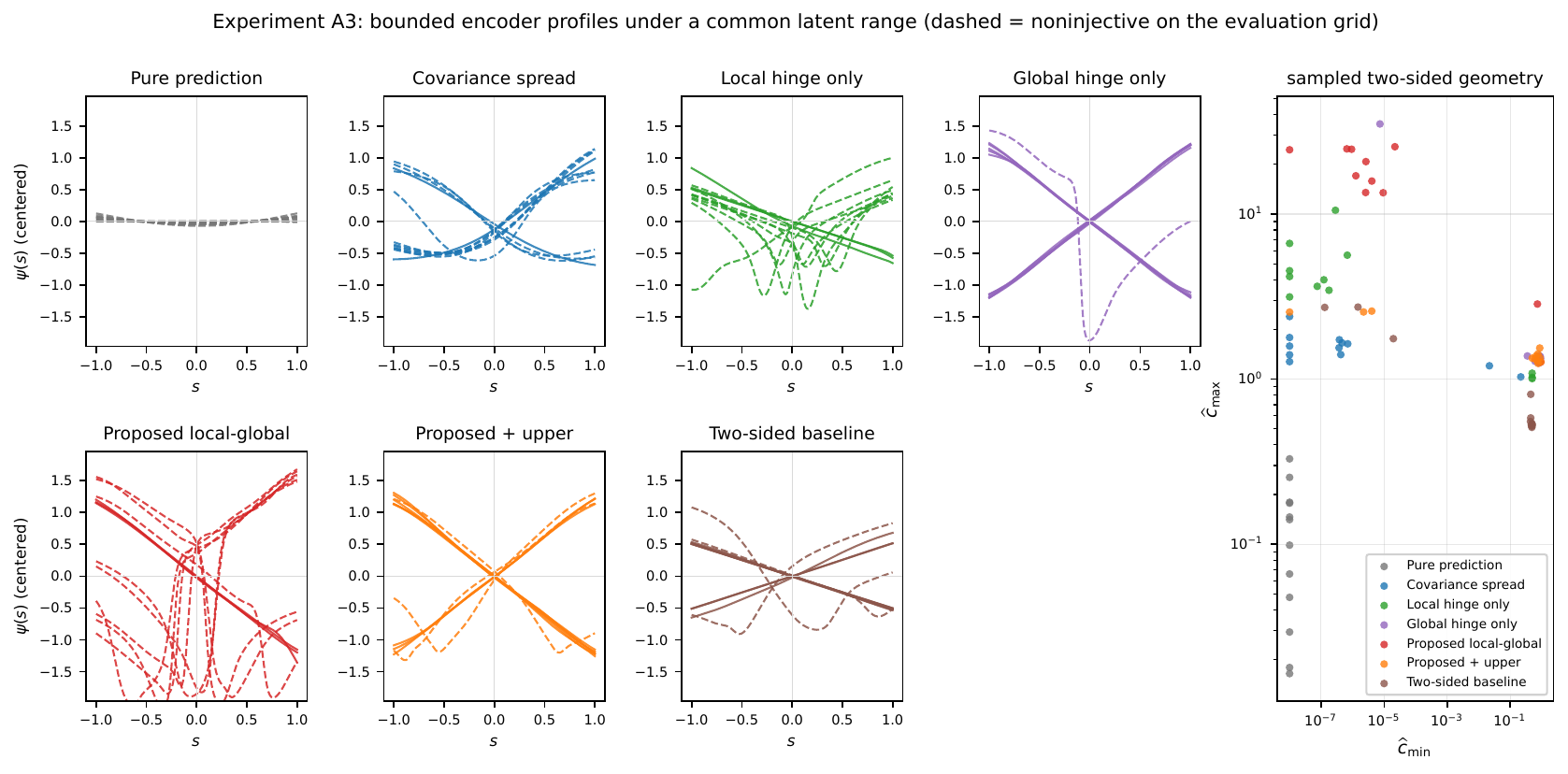}
\caption{Bounded encoder profiles and sampled two-sided geometry in the folding stress test.  Dashed profiles are non-injective on the evaluation grid.  Reporting lower and upper geometry together enables direct comparison of resolution, folding, and scale across objectives.}
\label{fig:folding}
\end{figure}

\subsection{Finite-capacity threshold diagnostic}
\label{sec:expB}
Experiment~B uses a norm-constrained coefficient-box spline proxy with fixed nested samples and repeated deterministic restarts, providing a controlled finite-dimensional laboratory for the threshold ingredients.  The interpolation threshold is $P^*=\BPstar$, while, for the medium-capacity family, the class-uniform statistical deviations range from $\BEpsMin$ to $\BEpsMax$; their separation from the sufficient sample-size margin quantifies the uniformity reserve carried by the class-wide guarantee.  At the same time, the trained models exhibit a clear intermediate regime in which sampled local and global margins strengthen while the uniform evaluation residual remains controlled.  This comparison turns the threshold formula into a diagnostic map and identifies the covering step as a concrete target for sharper data-dependent calibration.

\begin{figure}[t]\centering
\includegraphics[width=\textwidth]{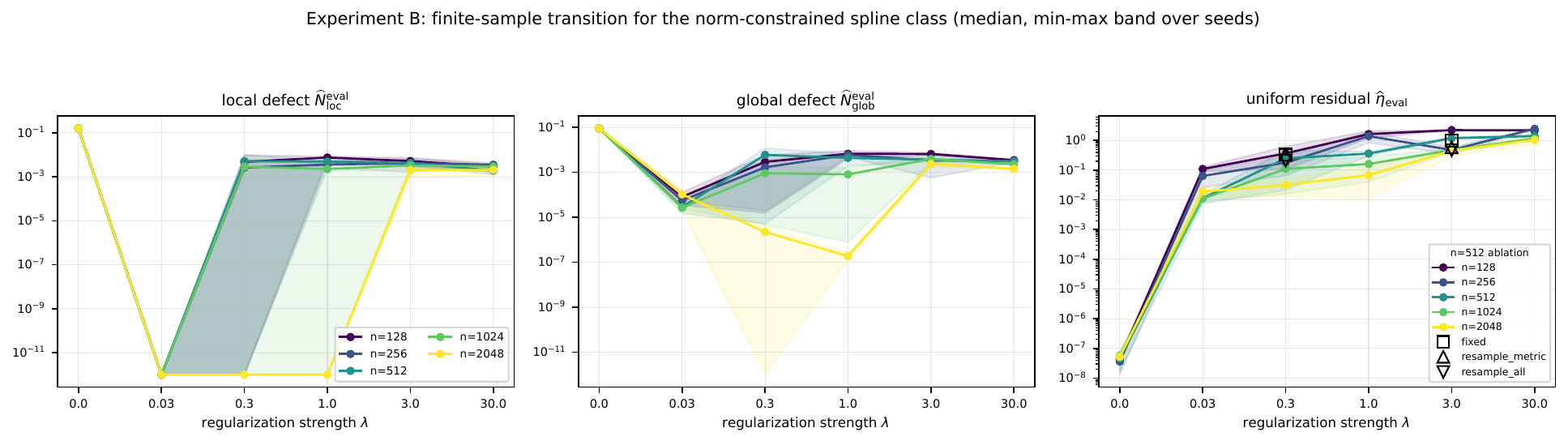}
\caption{Finite-capacity spline proxy.  Local and global empirical margins strengthen in an intermediate regularization regime, while the residual tradeoff identifies a certificate-aligned operating range.  The analytic sufficient bound retains a substantial uniformity reserve across the full class at the tested sample sizes.}
\label{fig:finitesample}
\end{figure}

\subsection{Nonlinear latent-MPC benchmark}
\label{sec:expC}
A smoothly saturated controlled pendulum with a nonlinear differentiable observation embedding is used to compare exact-model MPC and latent MPC.  All learned objectives share architecture, bounded latent range, data, optimization budget, learned cost-head form, and planning budget.  The exact-model reference is the best of $\CrefRestarts$ independent restarts and receives more total optimization effort than a learned arm, providing a stringent common baseline for the paired cost comparisons.

Pure prediction attains the smallest evaluated residual, $\widehat\eta_{\rm eval}=\CpureEta$, while the geometry-aware objectives raise the median chord constant from $\widehat c_{\min}=\CpureCmin$ to a control-relevant range.  The corresponding median paired closed-loop cost difference for pure prediction is $\CpuredJmedTFive$ at $T=5$ and $\CpuredJmedTTwoZero$ at $T=20$, between $\CpureRatioMin$ and $\CpureRatioMax$ times the medians of the regularized arms.  This is the central empirical finding: restoring the representation geometry singled out by the theorem converts accurate latent self-prediction into substantially more reliable control.

Pooled over methods, seeds, and horizons, the Spearman correlation between worst-case paired cost difference and $\widehat\eta_{\rm eval}/\widehat c_{\min}$ is $\CrhoPooledRatio$ with cluster-bootstrap interval $[\CrhoPooledRatioLo,\CrhoPooledRatioHi]$, whereas the correlation with $\widehat\eta_{\rm eval}$ alone is $\CrhoPooledEta$.  Within the compact group of geometry-aware models, the corresponding correlation is $\CrhoNoPurePooledRatio$ with interval $[\CrhoNoPurePooledRatioLo,\CrhoNoPurePooledRatioHi]$.  The ratio therefore provides a strong diagnostic gate between collapsed and non-collapsed models and a secondary, horizon-dependent diagnostic within the already reliable regime.

Covariance spread, local--global, and two-sided regularization have overlapping per-seed ranges and all achieve competitive performance on this controlled nonlinear benchmark.  Their shared success demonstrates that geometry-aware training robustly restores control-relevant representations.  The local--global hinge is mathematically distinguished because its directional and separated-pair terms are exactly the quantities for which the paper proves the complete finite-sample pointwise certificate and deterministic transfer theorem.  The evaluation-set theorem proxy carries a substantial certification reserve---its median envelope factor is $\CproxyLooseMed$ and its maximum is $\CproxyLooseMax$---and provides a transparent decomposition of the regularity and horizon-amplification terms available for trajectory-local and model-specific calibration.

\begin{table}[t]\centering\small
\begin{tabular}{lcccc}
\toprule
Objective & $\widehat\eta_{\rm eval}$ & $\widehat c_{\min}$ & median $\Delta J$, $T=5$ & median $\Delta J$, $T=20$\\
\midrule
Pure prediction & $\CpureEta$ & $\CpureCmin$ & $\CpuredJmedTFive$ & $\CpuredJmedTTwoZero$\\
Covariance spread & $\CcovEta$ & $\CcovCmin$ & $\CcovdJmedTFive$ & $\CcovdJmedTTwoZero$\\
Local hinge only & $\ClocalEta$ & $\ClocalCmin$ & $\ClocaldJmedTFive$ & $\ClocaldJmedTTwoZero$\\
Global hinge only & $\CglobalEta$ & $\CglobalCmin$ & $\CglobaldJmedTFive$ & $\CglobaldJmedTTwoZero$\\
Local--global hinge & $\ChingeEta$ & $\ChingeCmin$ & $\ChingedJmedTFive$ & $\ChingedJmedTTwoZero$\\
Two-sided baseline & $\CtwosidedEta$ & $\CtwosidedCmin$ & $\CtwosideddJmedTFive$ & $\CtwosideddJmedTTwoZero$\\
\bottomrule
\end{tabular}
\caption{Selected Experiment~C medians.  Geometry-aware training raises the lower metric constant and substantially reduces paired cost differences relative to the residual-only reference.  The regularized objectives achieve comparable control performance, and the local--global hinge pairs that competitiveness with the pointwise geometric quantities used by the theorem.}
\label{tab:planning-summary}
\end{table}

\begin{figure}[t]\centering
\includegraphics[width=\textwidth]{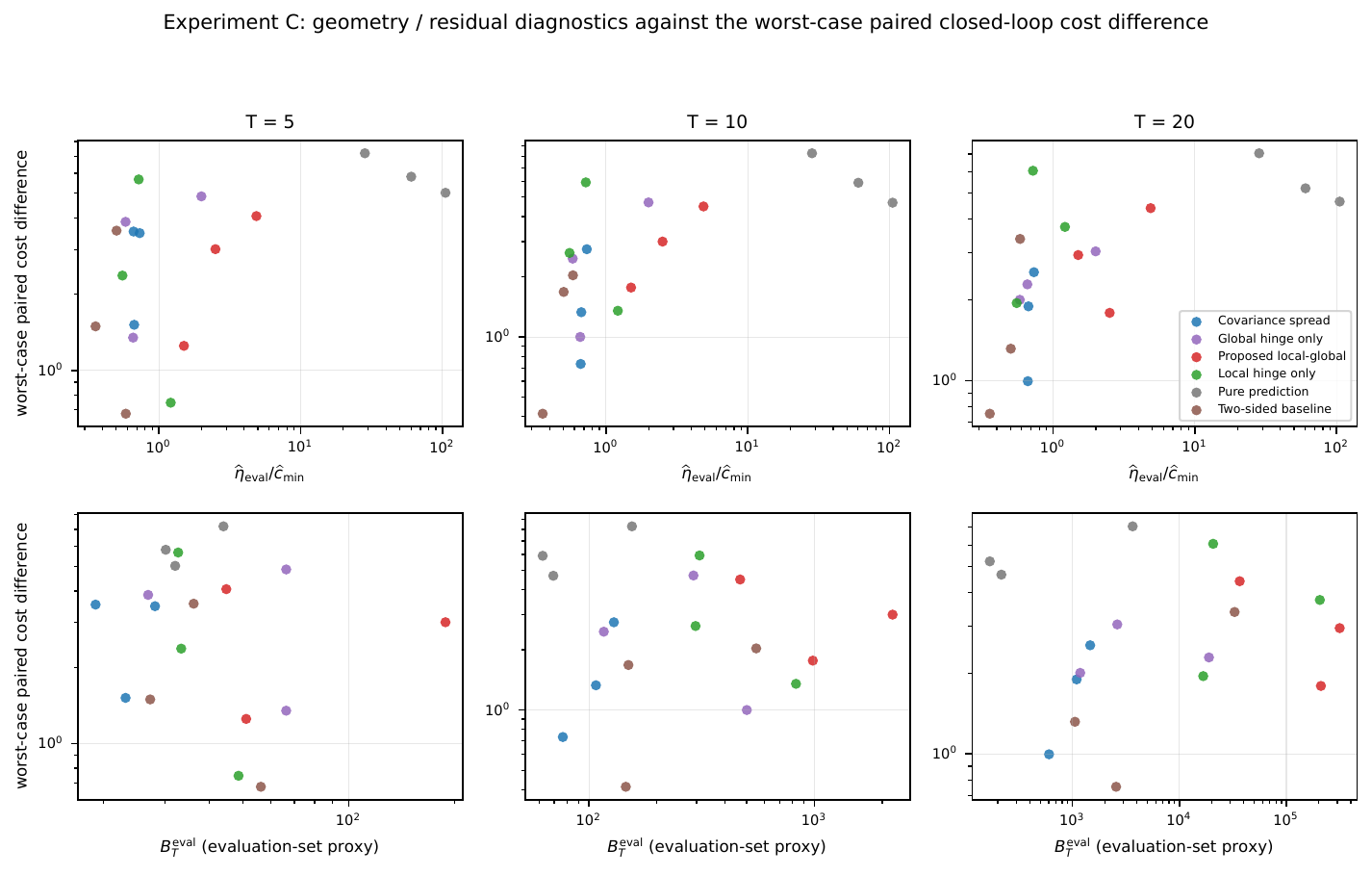}
\caption{Worst-case paired closed-loop cost difference against the geometry/residual ratio (top) and the evaluation-set theorem proxy (bottom).  The collapsed pure-prediction arm is clearly separated, while the regularized arms occupy a common control-relevant regime.}
\label{fig:geomplan}
\end{figure}
\FloatBarrier

\section{Discussion: scope, certification, and outlook}
\label{sec:discussion-certification}
The scope is deliberately concrete: observable-state distances and tangent directions are available during training from simulator state, proprioception, or certified state-estimator outputs whose distance and tangent errors leave positive effective margins, while deployment remains observation- and action-conditioned. Independently validated pixel-based geometry estimators enter the same framework only after their errors have been absorbed into those margins. The spline construction shows that the regularity and approximation hypotheses are non-vacuous, and the coverage and optimization assumptions translate into actionable sampling and margin-monitoring requirements.

The analytic threshold is a class-uniform sufficient condition, not a claim about the smallest empirically useful regularization weight. Its reserve is visible in Experiment~B and motivates sharper class-specific or finite-net certification. Experiment~C supports the broader geometric mechanism: several regularized objectives restore non-collapse and achieve competitive control, while the proposed local--global hinge is distinguished by exact alignment with the pointwise theorem. Extending the same objective-to-certificate principle to partially observed stochastic systems and independently certified visual metrics is a natural next direction.

\section{Conclusion}
\label{sec:conclusion}
For geometrically supervised latent control, the local--global hinge provides the missing pointwise stability mechanism between empirical prediction and deterministic planning. The principal theorem proves that every approximate empirical minimizer in an explicit certified regime is co-Lipschitz and uniformly approximately semiconjugate to the controlled dynamics; the modular planning result transfers those certificates to trajectory, cost, learned-cost-head, and optimizer bounds.

The constructive approximation theory, sharp interpolation result, finite-net certificate, and controlled experiments complete the learning-to-control chain. Together they show how an empirically optimized objective can select metrically faithful dynamics models and how the resulting constants support reliable finite-horizon planning without weakening the distinction between analytic certification and finite-model evidence.

\section*{Declaration on the use of AI tools}
AI tools were used for language editing, code review, and manuscript checking.  The authors assume responsibility for all content.

\bibliographystyle{unsrt}
\bibliography{bibfile}
\end{document}


\title[Supplementary Material: Metric Non-Collapse]{Supplementary Material for\\[3pt]
Finite-Sample Metric Non-Collapse for Geometrically Supervised Latent World Models in Control}
\author[Alain Bensoussan]{Alain Bensoussan}
\address[Alain Bensoussan]{Naveen Jindal School of Management\\
University of Texas at Dallas, Richardson, TX 75080, USA}
\email{\href{mailto:alain.bensoussan@utdallas.edu}{alain.bensoussan@utdallas.edu}}
\author[Minh-Nhat Phung]{Minh-Nhat Phung}
\address[Minh-Nhat Phung]{Department of Mathematics\\
Texas A\&M University, College Station, TX 77843, USA}
\email{\href{mailto:pmnt1114@tamu.edu}{pmnt1114@tamu.edu}}
\author[Minh-Binh Tran]{Minh-Binh Tran}
\address[Minh-Binh Tran]{Department of Mathematics\\
Texas A\&M University, College Station, TX 77843, USA}
\email{\href{mailto:minhbinh@tamu.edu}{minhbinh@tamu.edu}}
\date{}
\maketitle
\allowdisplaybreaks
\raggedbottom
\phantomsection\label{ass:A1}
\phantomsection\label{thm:representation}
\setcounter{tocdepth}{2}
\tableofcontents
\clearpage

\section{Exact realizers and detailed finite-capacity spline construction}
\label{supp:spline-approximation}
This section contains the complete construction material used by the main theorem.  It first constructs an exact non-collapsed latent realization, quantitative inverse charts, and global \(C^{1,1}\) representatives.  It then verifies Assumption~A4 for norm-constrained tensor-product splines, including range enforcement, compactness, coefficient counts, margin clearing, and a nonempty geometric parameter regime.  All arguments relocated from the main article are retained here in full.

\subsection{Exact and smooth non-collapsed realizers}
\begin{lemma}[Exact non-collapsed Lipschitz realization]
\label{lem:exact-Lip}
Assume Assumption~\hyperref[ass:A1]{A1}, $d_z\ge d_s$, and let $K_Z\subset\mathbb R^{d_z}$ be compact and convex with nonempty interior.  Then there exist a smooth affine bi-Lipschitz embedding $\psi_0:S\to\operatorname{int}K_Z$ with constants $0<c_{\psi_0}\le L_{\psi_0}$ and a $C^{1,1}$ map $\Pi_Z:Z\to K_Z$ that equals the identity on an open neighborhood of $\psi_0(S)$.  Moreover, there exist Lipschitz maps
\[
\Phi_0:X\to K_Z,
\qquad
F_0:K_Z\times A\to K_Z,
\]
such that
\[
\Phi_0(H(s))=\psi_0(s),
\qquad
F_0(\psi_0(s),a)=\psi_0(G(s,a)),
\qquad
R[\Phi_0,F_0]\equiv0.
\]
They may be chosen with bounds
\[
\Lip(\Phi_0)
\le
\sqrt{d_z}\,\frac{L_{\psi_0}}{c_H},
\qquad
\Lip(F_0)
\le
\sqrt{d_z}\,L_{\psi_0}L_G\max\{c_{\psi_0}^{-1},1\}.
\]
Consequently, any admissible hypothesis classes that contain these maps admit an exact non-collapsed realization. In the finite-capacity theorem, richness is supplied by the explicit approximation property, while the Lipschitz budget provides the uniform regularity needed for empirical-to-continuum control.
\end{lemma}
\begin{proof}
	Choose $z_c\in\operatorname{int}K_Z$ and $r>0$ such that
\[
\overline B_{2r}(z_c)\subset\operatorname{int}K_Z.
\]
Let $\iota:\mathbb R^{d_s}\to\mathbb R^{d_z}$ be the coordinate inclusion, fix $s_c\in S$, and set $D_S:=\max_{s\in S}|s-s_c|>0$.  With $\gamma:=r/(2D_S)$, define
\[
\psi_0(s):=z_c+\gamma\,\iota(s-s_c).
\]
Then $\psi_0(S)\subset B_r(z_c)$ and
\[
\|\psi_0(s)-\psi_0(s')\|_Z=\gamma|s-s'|,
\]
so $\psi_0$ is a smooth affine bi-Lipschitz embedding with $c_{\psi_0}=L_{\psi_0}=\gamma$.  Choose a smooth scalar saturation $\vartheta:[0,\infty)\to[0,\infty)$ such that $\vartheta(t)=1$ for $t\le r^2$ and $\vartheta(t)\sqrt t\le2r$ for all $t\ge0$, and define
\[
\Pi_Z(z):=z_c+\vartheta(|z-z_c|^2)(z-z_c).
\]
This map is $C^{1,1}$ (indeed smooth), takes values in $\overline B_{2r}(z_c)\subset K_Z$, and is the identity on $B_r(z_c)$, an open neighborhood of $\psi_0(S)$.

	Since \(H\) is bi-Lipschitz, \(H\) is injective. We define the map $\widehat\Phi_0:H(S)\to Z$ by \(\widehat\Phi_0(x):=\psi_0(H^{-1}(x))\).
	Then, we get that \(\widehat\Phi_0\) is Lipschitz with \(\Lip(\widehat\Phi_0)\le L_{\psi_0}c_H^{-1}\), where \(c_H\) is the lower co-Lipschitz constant of \(H\).
We extend each scalar coordinate of $\widehat\Phi_0$ to $X$ by the McShane extension theorem~\cite{McShane1934}.
Combining these coordinatewise extensions, we obtain the vector-valued extension $\widetilde\Phi_0$ with \[\Lip(\widetilde{\Phi}_0)\le \sqrt{d_z}\,\frac{L_{\psi_0}}{c_H}.\]
Since $K_Z$ is closed and convex, its metric projection $\Pi_{K_Z}$ is $1$-Lipschitz. As the values of \(\widehat\Phi_0\) lie in \(K_Z\) on \(H(S)\), by setting $\Phi_0:=\Pi_{K_Z}\circ\widetilde\Phi_0$, we obtain the encoder satisfying 
\[\Lip(\Phi_0)\le \sqrt{d_z}\,\frac{L_{\psi_0}}{c_H}\qquad\text{and}\qquad \Phi_0(H(s))=\psi_0(s).\]

Next, on $\psi_0(S)\times A$, we define \(\widehat F(\zeta,a) := \psi_0\bigl(G(\psi_0^{-1}(\zeta),a)\bigr)\).
With respect to the product metric
\[
d_F((\zeta,a),(\zeta',a'))
:=
\|\zeta-\zeta'\|_Z+|a-a'|,
\]
the map \(\widehat{F}\) is Lipschitz with constant at most \(L_{\psi_0}L_G\max\{c_{\psi_0}^{-1},1\}\).
As in the construction of \(\Phi_0\), applying coordinatewise McShane extension and then composing with $\Pi_{K_Z}$, we obtain $F_0$ such that 
\[\Lip(F_0)\le\sqrt{d_z}\, L_{\psi_0}L_G\max\{c_{\psi_0}^{-1},1\}\qquad \text{and}\qquad F_0(\psi_0(s),a)=\psi_0(G(s,a)).\]
Therefore, we have the identity
\[
F_0(\Phi_0(H(s)),a)
=
\psi_0(G(s,a))
=
\Phi_0(H(G(s,a))).
\]
Thus, the residual vanishes identically. The lower co-Lipschitz estimate is the assumed inequality
$\|\psi_0(s)-\psi_0(s')\|_Z\ge c_{\psi_0}d_S(s,s')$.
\end{proof}
\begin{lemma}[Quantitative inverse regularity for Euclidean observable embeddings]
\label{lem:inverse-C11}
Let $\Omega\subset\mathbb R^d$ be a bounded convex open set, let
$U\subset\mathbb R^d$ be an open neighborhood of $\overline\Omega$, and let
\[
f:U\to\mathbb R^m,
\qquad m\ge d,
\]
be of class $C^{1,1}$. Suppose that $f|_{\overline\Omega}$ is quantitatively co-Lipschitz: there exists $c_f>0$ such that
\[
\|f(s)-f(s')\|_{\mathbb R^m}\ge c_f|s-s'|
\qquad\text{for all }s,s'\in\overline\Omega.
\]
Then $f|_{\overline\Omega}$ is injective and
\[
\sigma_{\min}(Df(s))
:=\inf_{|v|=1}\|Df(s)v\|_{\mathbb R^m}
\ge c_f
\qquad\text{for every }s\in\overline\Omega.
\]

Choose $\delta>0$ such that the closed neighborhood
\[
\overline\Omega_{\delta}
:=\{x\in\mathbb R^d:\operatorname{dist}(x,\overline\Omega)\le\delta\}
\]
is contained in $U$, and set
\[
L_f:=\Lip(Df;\overline\Omega_{\delta}),
\qquad
r_*:=\min\left\{\delta,\frac{c_f}{2L_f}\right\},
\]
with the convention $c_f/(2L_f)=+\infty$ when $L_f=0$. For every
$s_0\in\overline\Omega$, define
\[
A_0:=Df(s_0),
\qquad
P_{s_0}:=(A_0^\top A_0)^{-1}A_0^\top,
\qquad
\mathcal F_{s_0}:=P_{s_0}\circ f,
\qquad
V_{s_0}:=B_{r_*}(s_0).
\]
Then
\[
P_{s_0}A_0=I_{\mathbb R^d},
\qquad
\|P_{s_0}\|_{\mathrm{op}}\le c_f^{-1},
\]
and, for all $s,s'\in V_{s_0}$,
\[
\frac12|s-s'|
\le
|\mathcal F_{s_0}(s)-\mathcal F_{s_0}(s')|
\le
\frac32|s-s'|.
\]
Consequently, $O_{s_0}:=\mathcal F_{s_0}(V_{s_0})$ is open,
$\mathcal F_{s_0}:V_{s_0}\to O_{s_0}$ is a $C^{1,1}$ diffeomorphism, and its inverse
$\Gamma_{s_0}:O_{s_0}\to V_{s_0}$ satisfies the quantitative estimates
\[
\Lip(\Gamma_{s_0};O_{s_0})\le2,
\qquad
\sup_{y\in O_{s_0}}\|D\Gamma_{s_0}(y)\|_{\mathrm{op}}\le2,
\qquad
\Lip(D\Gamma_{s_0};O_{s_0})\le\frac{8L_f}{c_f}.
\]
Moreover,
\[
(f|_{\overline\Omega})^{-1}(y)
=
\Gamma_{s_0}(P_{s_0}y)
\qquad
\text{for every }y\in f(\overline\Omega\cap V_{s_0}).
\]
Thus the local inverse charts may be chosen with a uniform radius and with constants depending only on $c_f$, the fixed-neighborhood $C^{1,1}$ budget of $f$, and the finite covering of $\overline\Omega$ by these charts.
\end{lemma}
\begin{proof}
The co-Lipschitz estimate immediately implies injectivity on $\overline\Omega$. Let $s\in\Omega$ and $v\in\mathbb S^{d-1}$. For all sufficiently small $t\ne0$, convexity gives $s+tv\in\Omega$, and hence we have
\[
\frac{\|f(s+tv)-f(s)\|_{\mathbb R^m}}{|t|}\ge c_f.
\]
Letting $t\to0$, we get $\|Df(s)v\|\ge c_f$. The same conclusion holds on $\partial\Omega$ by approximation from the interior and continuity of $Df$. Therefore, we arrive at
$\sigma_{\min}(Df(s))\ge c_f$ on $\overline\Omega$.

Fix $s_0\in\overline\Omega$. Since $A_0=Df(s_0)$ has full column rank, $A_0^\top A_0$ is invertible and $P_{s_0}$ is the Moore--Penrose left inverse of $A_0$. Thus, we have
\[
P_{s_0}A_0=I,
\qquad
\|P_{s_0}\|_{\mathrm{op}}
=\sigma_{\min}(A_0)^{-1}
\le c_f^{-1}.
\]
For $s\in V_{s_0}$, by the definitions of $r_*$ and \(V_{s_0}\), we estimate that
\[
\|D\mathcal F_{s_0}(s)-I\|_{\mathrm{op}}
=
\|P_{s_0}(Df(s)-Df(s_0))\|_{\mathrm{op}}
\le
\frac{L_f}{c_f}|s-s_0|
<\frac12.
\]
In particular, $D\mathcal F_{s_0}(s)$ is invertible and
\[
\|D\mathcal F_{s_0}(s)^{-1}\|_{\mathrm{op}}\le2.
\]
Because $V_{s_0}$ is convex, for $s,s'\in V_{s_0}$ we may integrate along the segment joining them:
\[
\mathcal F_{s_0}(s)-\mathcal F_{s_0}(s')-(s-s')
=
\int_0^1
\big(D\mathcal F_{s_0}(s'+t(s-s'))-I\big)(s-s')\,\dd t.
\]
The preceding derivative estimate therefore gives
\[
\left|
\mathcal F_{s_0}(s)-\mathcal F_{s_0}(s')-(s-s')
\right|
\le\frac12|s-s'|,
\]
which proves the two-sided estimate
\[
\frac12|s-s'|
\le
|\mathcal F_{s_0}(s)-\mathcal F_{s_0}(s')|
\le
\frac32|s-s'|.
\]
Hence $\mathcal F_{s_0}$ is injective on $V_{s_0}$. Since its derivative is invertible everywhere, the inverse function theorem shows that it is a local diffeomorphism and an open map. It follows that
$O_{s_0}=\mathcal F_{s_0}(V_{s_0})$ is open and that
$\mathcal F_{s_0}:V_{s_0}\to O_{s_0}$ is a global $C^1$ diffeomorphism. The lower estimate above implies
\[
\Lip(\Gamma_{s_0};O_{s_0})\le2,
\]
and the inverse derivative formula gives
\[
D\Gamma_{s_0}(y)
=D\mathcal F_{s_0}(\Gamma_{s_0}(y))^{-1},
\qquad
\sup_{y\in O_{s_0}}\|D\Gamma_{s_0}(y)\|_{\mathrm{op}}\le2.
\]
For $y_i=\mathcal F_{s_0}(s_i)\in O_{s_0}$, the matrix identity
$A^{-1}-B^{-1}=A^{-1}(B-A)B^{-1}$ yields
\[
\begin{aligned}
\|D\Gamma_{s_0}(y_1)-D\Gamma_{s_0}(y_2)\|_{\mathrm{op}}
&\le
4\|D\mathcal F_{s_0}(s_1)-D\mathcal F_{s_0}(s_2)\|_{\mathrm{op}}\\
&\le
4\|P_{s_0}\|_{\mathrm{op}}L_f|s_1-s_2|\\
&\le
\frac{8L_f}{c_f}|y_1-y_2|.
\end{aligned}
\]
Thus $\Gamma_{s_0}\in C^{1,1}(O_{s_0})$ with the claimed quantitative bounds. Finally, if $y=f(s)$ with $s\in\overline\Omega\cap V_{s_0}$, then
$P_{s_0}y=\mathcal F_{s_0}(s)$, and therefore
$\Gamma_{s_0}(P_{s_0}y)=s=(f|_{\overline\Omega})^{-1}(y)$.
Compactness of $\overline\Omega$ supplies a finite subcover by the neighborhoods $V_{s_0}$, while the radius and all displayed constants are uniform.
\end{proof}
\begin{lemma}[\(C^{1,1}\) exact representatives]
\label{lem:smooth-realizer}
Suppose that \hyperref[ass:A1]{A1} holds and let \(X=\R^{d_x}\).  Let \(\psi_0:S\to\operatorname{int}K_Z\) and \(\Pi_Z:Z\to K_Z\) be the affine bi-Lipschitz embedding and the smooth range-enforcing map constructed in Lemma~\ref{lem:exact-Lip}; in particular, $\Pi_Z$ is the identity on an open neighborhood of $\psi_0(S)$.  Then the exact realization admits representatives
\[
\Phi_0\in C^{1,1}(X;K_Z),
\qquad
F_0\in C^{1,1}(K_Z\times A;K_Z),
\]
with
\[
\Phi_0(H(s))=\psi_0(s),
\qquad
F_0(\psi_0(s),a)=\psi_0(G(s,a)),
\qquad
R[\Phi_0,F_0]\equiv0.
\]
\end{lemma}
\begin{proof}
We write
\[
\psi_0:S\to Z=\R^{d_z}
\]
for the affine \(C^{1,1}\) bi-Lipschitz embedding. By Assumption~\hyperref[ass:A1]{A1} and the construction in Lemma~\ref{lem:exact-Lip}, \(H\), \(G\), and \(\psi_0\) are restrictions of \(C^{1,1}\) maps defined on open Euclidean neighborhoods of their domains. After possibly shrinking these neighborhoods, we use the same notation for the corresponding extensions.

Let
\[
\Pi_Z:Z\to K_Z
\]
be the constructed $C^{1,1}$ range-enforcing map, and let
\[
\mathcal O_\psi\subset Z
\]
be an open neighborhood of \(\psi_0(S)\) on which
\[
\Pi_Z(\zeta)=\zeta.
\]

We first construct the encoder \(\Phi_0\). By Lemma \(\ref{lem:inverse-C11}\), for each \(s_i\in S\), there exist an open neighborhood \(V_i^H\subset \mathbb R^{d_s}\) of \(s_i\), a linear map
\[
        P_i^H:\mathbb R^{d_x}\to\mathbb R^{d_s},
\]
an open set \(O_i^H\subset\mathbb R^{d_s}\), and a \(C^{1,1}\) map
\[
        \Gamma_i^H:O_i^H\to V_i^H
\]
such that
\[
        P_i^H\circ H:V_i^H\to O_i^H
\]
is a \(C^{1,1}\) diffeomorphism with inverse \(\Gamma_i^H\). In particular, for \(s\in S\cap V_i^H\), we have
\[
        \Gamma_i^H(P_i^H H(s))=s .
\]
Since \(S\) is compact, we may pass to a finite subcover and assume that
\[
        S\subset \bigcup_{i=1}^{N_H} V_i^H .
\]

For each \(i\), we choose an open ambient neighborhood
\(\mathcal N_i^H\subset \mathbb R^{d_x}\) of
\(H(S\cap V_i^H)\) such that
\[
	P_i^H(\mathcal N_i^H)\subset O_i^H \quad \text{and}\quad\mathcal N_i^H\cap H(S)\subset H(S\cap V_i^H).
\]
This is possible after shrinking \(\mathcal N_i^H\), because \(H(S\cap V_i^H)\) is relatively open in \(H(S)\) and \(H(S)\) is compact.
On \(\mathcal N_i^H\), we define
\[
        \Phi_i^H(x):=\psi_0\bigl(\Gamma_i^H(P_i^H x)\bigr).
\]
Then \(\Phi_i^H\in C^{1,1}(\mathcal N_i^H;Z)\), and whenever
\(x=H(s)\in \mathcal N_i^H\cap H(S)\), we have
\[
        \Phi_i^H(H(s))
        =
        \psi_0\bigl(\Gamma_i^H(P_i^H H(s))\bigr)
        =
        \psi_0(s).
\]

We now choose smooth functions
\[
        \chi_i^H\in C_c^\infty(\mathcal N_i^H),
        \qquad i=1,\dots,N_H,
\]
such that
\[
        \sum_{i=1}^{N_H}\chi_i^H(x)=1
        \qquad \text{for all }x\text{ in a neighborhood of }H(S).
\]
Such a finite smooth partition of unity exists because the sets \(\mathcal N_i^H\) cover the compact set \(H(S)\).

Let \(z_\ast\in Z\) be fixed. We define on all of \(X=\mathbb R^{d_x}\)
\[
        \widetilde\Phi_0(x)
        :=
        \sum_{i=1}^{N_H}\chi_i^H(x)\Phi_i^H(x)
        +
        \left(1-\sum_{i=1}^{N_H}\chi_i^H(x)\right)z_\ast,
\]
where each product \(\chi_i^H\Phi_i^H\) is extended by zero outside
\(\mathcal N_i^H\). Since this is a finite sum of products of smooth functions and \(C^{1,1}\) maps,
\[
        \widetilde\Phi_0\in C^{1,1}(X;Z).
\]

Using the range-enforcing map, we set
\[
	\Phi_0:= \Pi_Z \circ \widetilde{\Phi}_0.
\]
This encoder satisfies \(\Phi_0(H(s))=\psi_0(s)\) and \(\Phi_0\in C^{1,1}(X;K_Z)\).

We next construct \(F_0\). Apply Lemma~\ref{lem:inverse-C11} to the embedding \(\psi_0\). For finitely many points \(s_j\in S\), there exist open neighborhoods \(V_j^\psi\subset\mathbb R^{d_s}\), linear maps
\[
P_j^\psi:Z\to\mathbb R^{d_s},
\]
open sets \(O_j^\psi\subset\mathbb R^{d_s}\), and local inverse maps
\[
\Gamma_j^\psi:O_j^\psi\to V_j^\psi
\]
such that \(P_j^\psi\circ\psi_0:V_j^\psi\to O_j^\psi\) is a \(C^{1,1}\) diffeomorphism and the sets \(V_j^\psi\) cover \(S\). There also exist ambient neighborhoods \(\mathcal N_j^\psi\subset Z\) of \(\psi_0(S\cap V_j^\psi)\) satisfying
\[
P_j^\psi(\mathcal N_j^\psi)\subset O_j^\psi,
\qquad
\mathcal N_j^\psi\cap\psi_0(S)
\subset
\psi_0(S\cap V_j^\psi).
\]
We choose \(\chi_j^\psi\in C_c^\infty(\mathcal N_j^\psi)\) such that \(\sum_j\chi_j^\psi=1\) on a neighborhood of \(\psi_0(S)\). Then, we define
\[
\widetilde F_0(\zeta,a)
        :=
        \sum_{j}
        \chi_j^\psi(\zeta)\psi_0\bigl(G(\Gamma_j^\psi(P_j^\psi\zeta),a)\bigr)
        +
        \left(1-\sum_{j}\chi_j^\psi(\zeta)\right)z_\ast.
\]
Applying the range-enforcing map to \(\widetilde F_0\), that is \(F_0:=\Pi_Z\circ \widetilde{F}_0\), we obtain \(F_0\in C^{1,1}(K_Z\times A;K_Z)\) with
\[F_0(\psi_0(s),a)=\psi_0(G(s,a)).\]
It immediately follows that \( \Phi_0,F_0 \) are the desired representatives.
\end{proof}

The next assumption makes explicit the finite-capacity and capacity-uniform regularity requirement.  It provides a concrete architecture-level verification target, and the spline construction below demonstrates that this target is attained by nontrivial finite-capacity classes with quantitative approximation rates.
\begin{assumption}{A4}{Finite-capacity \(C^1\)-approximation with uniform \(C^{1,1}\) budgets}
	Let \((\Phi_0,F_0)\) be the \(C^{1,1}\) exact representatives from Lemma~\ref{lem:smooth-realizer}.
There exist classes \(\mathcal A_\Phi(W),\mathcal A_F(W)\), indexed by a capacity parameter \(W\) and closed in the \(C^1\) topology, whose elements map into \(K_Z\).
Their Lipschitz and \(C^{1,1}\) budgets depend only on \(S\), \(A\), \(Z\), \(K_Z\), \(\Pi_Z\), and \((\Phi_0,F_0)\); in particular, they are independent of \(W\).
For each sufficiently large \(W\), there exist \((\Phi_W,F_W)\in\mathcal A_\Phi(W)\times\mathcal A_F(W)\) with
\[
\|\Phi_W-\Phi_0\|_{C^1(H(S))}
+
\|F_W-F_0\|_{C^1(\psi_0(S)\times A)}
\le \delta(W),
\qquad
\delta(W)\to0 \quad\text{as } W\to\infty .
\]
\end{assumption}

\subsection{Finite-capacity spline approximation}
\begin{lemma}[Range-constrained spline approximation on a box]
\label{lem:spline-A4}
Let \(\Omega\subset\R^d\) be a compact rectangular box, let \(f_0:\Omega\to\R^m\) be the restriction of a \(C^{1,1}\) map defined on an open neighborhood of \(\Omega\), and let \(K\subset\R^m\) be a compact convex set with \(f_0(\Omega)\subset\operatorname{int}K\).
Assume there is a \(C^{1,1}\) map \(\Pi_K:\R^m\to K\) that is the identity on a fixed open neighborhood of \(f_0(\Omega)\).

For a mesh width \(h>0\), let \(\mathcal S_h^m(\Omega)\) be the vector-valued tensor-product cubic B-spline space on a uniform grid of mesh size \(h\) over \(\Omega\); its dimension is \(mN_h\) with \(N_h\asymp h^{-d}\) as \(h\downarrow0\).
Writing \(W_h:=mN_h\asymp h^{-d}\) for the number of scalar coefficients, we have \(h\asymp W_h^{-1/d}\). For \(B>0\), we define the range-constrained, norm-constrained class
\[
\mathcal A_h^K(B;\Omega,m)
:=\Bigl\{\Pi_K\circ u:\ u\in\mathcal S_h^m(\Omega),\
\|u\|_{C^0(\Omega)}+\|Du\|_{C^0(\Omega)}+\Lip(Du;\Omega)\le B\Bigr\}.
\]
Every element of \(\mathcal A_h^K(B;\Omega,m)\) is \(C^{1,1}\), with a \(C^{1,1}\) bound depending only on \(B\) and \(\Pi_K\), uniformly in \(h\).

Fix a rectangular neighborhood \(\Omega^+\) with \(\Omega\Subset\Omega^+\) on which a \(C^{1,1}\) extension of \(f_0\) is defined. Then there exist constants \(B,C>0\) and \(h_0>0\), depending only on \(\Omega\), \(\Omega^+\), \(m\), \(K\), \(\Pi_K\), and the \(C^{1,1}(\Omega^+)\)-norm of this fixed extension, and functions \(\widetilde u_h\in\mathcal A_h^K(B;\Omega,m)\) for \(0<h<h_0\), such that
\[
\|\widetilde u_h-f_0\|_{C^1(\Omega)}\le C\,h\,\Lip(Df_0;\Omega^+),
\qquad
\sup_{0<h<h_0}\|\widetilde u_h\|_{C^{1,1}(\Omega)}\le C\|f_0\|_{C^{1,1}(\Omega^+)}.
\]
Equivalently, in terms of the number \(W_h\) of scalar coefficients,
\[
\|\widetilde u_h-f_0\|_{C^1(\Omega)}\le C\,W_h^{-1/d}.
\]
\end{lemma}

\begin{proof}
We use a standard local tensor-product cubic B-spline quasi-interpolant \(Q_h\) on the fixed box \(\Omega^+\), restricted afterward to \(\Omega\). The uniform grid is extended from \(\Omega\) to \(\Omega^+\), so this restriction lies in the same mesh-
\(h\) spline space \(\mathcal S_h^m(\Omega)\). We choose the usual local construction that reproduces affine polynomials and whose coefficient functionals are uniformly bounded under scaling; see \cite[Chapter~12]{schumaker2007spline} for the construction of local spline quasi-interpolants. For completeness, we derive the derivative estimates required below directly from locality, scaling, affine reproduction, and the \(C^{1,1}\) Taylor remainder.

Let \(K_h\) be a mesh cell meeting \(\Omega\), and let \(\omega_{K_h}\subset\Omega^+\) be the union of the fixed finite number of neighboring cells on which the local coefficient functionals contributing to \(Q_h\) over \(K_h\) depend. Because only a uniformly bounded number of fixed-degree tensor-product B-splines are active on a cell, the local coefficient functionals are uniformly bounded, and a derivative of order \(j\) scales as \(h^{-j}\), there is a constant independent of \(h\) such that, for \(j=0,1,2\),
\[
\|D^j Q_h v\|_{L^\infty(K_h)}
\le C h^{-j}\|v\|_{L^\infty(\omega_{K_h})}.
\]
We choose \(x_{K_h}\in K_h\) and let \(p_{K_h}(y)=f_0(x_{K_h})+Df_0(x_{K_h})(y-x_{K_h})\). Since \(Df_0\) is Lipschitz on \(\Omega^+\) and \(\omega_{K_h}\) has diameter \(O(h)\), Taylor's formula with Lipschitz derivative gives
\[
\|f_0-p_{K_h}\|_{L^\infty(\omega_{K_h})}
\le C h^2\Lip(Df_0;\Omega^+),
\qquad
\|D(f_0-p_{K_h})\|_{L^\infty(\omega_{K_h})}
\le C h\Lip(Df_0;\Omega^+).
\]
Affine reproduction gives \(Q_h p_{K_h}=p_{K_h}\). Applying the local stability estimate to \(f_0-p_{K_h}\), and combining it with the Taylor remainder, yields on every cell meeting \(\Omega\)
\[
\|Q_hf_0-f_0\|_{L^\infty(K_h)}
\le C h^2\Lip(Df_0;\Omega^+),
\qquad
\|D(Q_hf_0-f_0)\|_{L^\infty(K_h)}
\le C h\Lip(Df_0;\Omega^+),
\]
and, because \(D^2p_{K_h}=0\), we find
\[
\|D^2Q_hf_0\|_{L^\infty(K_h)}
=
\|D^2Q_h(f_0-p_{K_h})\|_{L^\infty(K_h)}
\le C\Lip(Df_0;\Omega^+).
\]
Taking the maximum over the finitely overlapping cells, we obtain the global estimates
\[
\|Q_hf_0-f_0\|_{C^0(\Omega)}\le C h^2\Lip(Df_0;\Omega^+),
\qquad
\|D(Q_hf_0-f_0)\|_{C^0(\Omega)}\le C h\Lip(Df_0;\Omega^+),
\]
together with a mesh-independent bound on \(\Lip(DQ_hf_0;\Omega)\). In particular, after increasing \(C\) and taking \(h\le1\), we have
\[
\|Q_h f_0\|_{C^0(\Omega)}+
\|D(Q_h f_0)\|_{C^0(\Omega)}+
\Lip(D(Q_h f_0);\Omega)
\le C\|f_0\|_{C^{1,1}(\Omega^+)}.
\]

Setting \(u_h:=Q_h f_0\) and choosing \(B\ge 2C\|f_0\|_{C^{1,1}(\Omega^+)}\), we obtain \(\widetilde u_h:=\Pi_K\circ u_h\in\mathcal A_h^K(B;\Omega,m)\).

By the hypothesis, \(\Pi_K\) coincides with the identity on an open neighborhood \(\mathcal O_K\) of \(f_0(\Omega)\). For sufficiently small \(h_0>0\), we have \(u_h(\Omega)\subset \mathcal O_K\) for every \(h\) with \(0<h<h_0\). Consequently, \(\widetilde{u}_h\) agrees with \(u_h\) on \(\Omega\) for \(0<h<h_0\). Therefore, we have
\begin{align*}
	\|\widetilde{u}_h-f_0\|_{C^1(\Omega)}\le Ch\Lip(Df_0;\Omega^+),
	\qquad
	\|\widetilde{u}_h\|_{C^{1,1}(\Omega)}\le C\|f_0\|_{C^{1,1}(\Omega^+)}.
\end{align*}

To complete the proof, we establish a \(C^{1,1}\) budget for \(\mathcal A_h^K(B;\Omega,m)\).

Let \(w= \Pi_K \circ u \in\mathcal A_h^K(B;\Omega,m)\), where
\begin{align*}
	u\in \mathcal S_h^m(\Omega),\quad \|u\|_{C^0(\Omega)}+\|Du\|_{C^0(\Omega)}+\Lip(Du;\Omega)\le B.
\end{align*}

Since \(\|u\|_{C^0(\Omega)}\le B\), we have \(u(\Omega)\subset\overline B_B(0)\). We define the constants
\[
M_0:=\|\Pi_K\|_{C^0(\overline B_B(0))},\quad M_1:=\|D\Pi_K\|_{C^0(\overline B_B(0))},\quad M_2:=\Lip(D\Pi_K;\overline B_B(0)).
\]
Then, it follows that
\[
\|w\|_{C^0(\Omega)}
\le M_0.
\]
By the chain rule, we get
\[
Dw(x)
=
D\Pi_K(u(x))Du(x).
\]
Hence, we further get
\[
\|Dw\|_{C^0(\Omega)}
\le
M_1\|Du\|_{C^0(\Omega)}
\le
M_1B.
\]
For \(x,y\in\Omega\), we estimate
\[
\begin{aligned}
\|Dw(x)-Dw(y)\|
&\le
\|D\Pi_K(u(x))\|
\|Du(x)-Du(y)\|
\\
&\quad+
\|D\Pi_K(u(x))-D\Pi_K(u(y))\|
\|Du(y)\|
\\
&\le
M_1\Lip(Du;\Omega)|x-y|
+
M_2\|u(x)-u(y)\|\|Du\|_{C^0(\Omega)}
\\
&\le
M_1B|x-y|+M_2B^2|x-y|.
\end{aligned}
\]
Therefore, we arrive at
\[
\Lip(Dw;\Omega)
\le
M_1B+M_2B^2.
\]
Thus every element of \(\mathcal A_h^K(B;\Omega,m)\) has a \(C^{1,1}\) norm bounded by a constant depending only on \(B\) and \(\Pi_K\).
\end{proof}

\begin{lemma}[Localized range-constrained spline approximation]
\label{lem:localized-spline-A4}
Let \(\Omega\subset\mathbb R^d\) be a compact rectangular box, let
\(E\subset\Omega\) be compact, and let
\(f_0:\Omega\to\mathbb R^m\) be the restriction of a \(C^{1,1}\) map defined on an open neighborhood of \(\Omega\).
Let \(K\subset\mathbb R^m\) be compact and convex, and suppose that
\[
\Pi_K:\mathbb R^m\to K
\]
is of class \(C^{1,1}\). Assume that there exists an open set
\(\mathcal O\subset\mathbb R^m\) such that
\[
f_0(E)\Subset\mathcal O,
\qquad
\Pi_K(y)=y
\quad\text{for every }y\in\mathcal O.
\]

Fix a rectangular neighborhood \(\Omega^+\) with \(\Omega\Subset\Omega^+\) and a fixed \(C^{1,1}\) extension of \(f_0\) to \(\Omega^+\). Let \(Q_hf_0\) denote the restriction to \(\Omega\) of the same local tensor-product cubic B-spline quasi-interpolant construction used in Lemma~\ref{lem:spline-A4}, on the extended uniform mesh, and set
\[
u_h:=Q_hf_0,
\qquad
\widetilde u_h:=\Pi_K\circ u_h.
\]
Then there exist constants \(C>0\) and \(h_0>0\), depending only on \(\Omega\), \(\Omega^+\), \(m\), \(\Pi_K\), \(\mathcal O\), and the \(C^{1,1}(\Omega^+)\)-norm of the fixed extension, such that, for \(0<h<h_0\),
\[
\|\widetilde u_h-f_0\|_{C^1(E)}
\le
Ch\,\Lip(Df_0;\Omega^+).
\]
Here
\[
\|w\|_{C^1(E)}
:=
\sup_{x\in E}\|w(x)\|
+
\sup_{x\in E}\|Dw(x)\|.
\]
Moreover,
\[
\|u_h\|_{C^0(\Omega)}
+
\|Du_h\|_{C^0(\Omega)}
+
\Lip(Du_h;\Omega)
\le
C\|f_0\|_{C^{1,1}(\Omega^+)},
\]
uniformly in \(h\), and
\[
\sup_{0<h<h_0}
\|\widetilde u_h\|_{C^{1,1}(\Omega)}
\le C.
\]
\end{lemma}

\begin{proof}
The cellwise affine-reproduction argument in Lemma~\ref{lem:spline-A4}, applied on the fixed outer box, gives constants independent of \(h\) such that
\[
\|Q_hf_0-f_0\|_{C^0(\Omega)}
\le C_0h^2\Lip(Df_0;\Omega^+),
\qquad
\|D(Q_hf_0-f_0)\|_{C^0(\Omega)}
\le C_0h\Lip(Df_0;\Omega^+),
\]
and
\[
\|Q_hf_0\|_{C^0(\Omega)}
+
\|DQ_hf_0\|_{C^0(\Omega)}
+
\Lip(DQ_hf_0;\Omega)
\le C_1\|f_0\|_{C^{1,1}(\Omega^+)}.
\]
These are precisely the mesh-independent \(C^{1,1}\) stability and \(C^1\) approximation estimates needed in the localized argument.

Because \(f_0(E)\Subset\mathcal O\), we observe that
\[
d_{\mathcal O}
:=
\operatorname{dist}
\bigl(f_0(E),\mathbb R^m\setminus\mathcal O\bigr)
>0.
\]
Thus, we can choose \(h_0>0\) such that
\[
C_0h^2\Lip(Df_0;\Omega^+)<\frac{d_{\mathcal O}}2
\qquad\text{for all }0<h<h_0.
\]
Then \(u_h(E)\subset\mathcal O\). Since \(\Pi_K\) equals the identity on the open set \(\mathcal O\), we get
\[
\Pi_K(u_h(x))=u_h(x),
\qquad
D\Pi_K(u_h(x))=I_{\mathbb R^m}
\qquad\text{for every }x\in E.
\]
Thus, we further get
\[
\widetilde u_h(x)=u_h(x),
\qquad
D\widetilde u_h(x)=Du_h(x)
\qquad\text{for every }x\in E.
\]
Hence, after increasing the constant and taking \(h\le1\), we obtain
\[
\begin{aligned}
\|\widetilde u_h-f_0\|_{C^1(E)}
&=\|u_h-f_0\|_{C^1(E)}\\
&\le
\|u_h-f_0\|_{C^0(\Omega)}
+
\|D(u_h-f_0)\|_{C^0(\Omega)}\\
&\le
Ch\,\Lip(Df_0;\Omega^+).
\end{aligned}
\]

It remains to establish the uniform \(C^{1,1}\) bound after composition
with \(\Pi_K\). The stability estimate implies that all sets \(u_h(\Omega)\) are contained in a fixed compact ball. On this ball, let
\[
M_1:=\|D\Pi_K\|_{C^0},
\qquad
M_2:=\Lip(D\Pi_K).
\]
The chain rule gives
\[
D\widetilde u_h(x)=D\Pi_K(u_h(x))Du_h(x),
\]
so we have
\[
\|D\widetilde u_h\|_{C^0(\Omega)}
\le
M_1\|Du_h\|_{C^0(\Omega)}.
\]
For \(x,y\in\Omega\), we estimate
\[
\begin{aligned}
\|D\widetilde u_h(x)-D\widetilde u_h(y)\|
&\le
M_1\|Du_h(x)-Du_h(y)\|\\
&\quad+
M_2\|u_h(x)-u_h(y)\|\,\|Du_h(y)\|.
\end{aligned}
\]
Therefore, we obtain the bound
\[
\Lip(D\widetilde u_h;\Omega)
\le
M_1\Lip(Du_h;\Omega)
+
M_2\|Du_h\|_{C^0(\Omega)}^2.
\]
Together with the compactness of \(K\), which controls the \(C^0\)-norm
of \(\widetilde u_h\), the uniform spline stability estimate proves the
claimed \(C^{1,1}\) bound.
\end{proof}

The next proposition records one concrete way in which Assumption~A4 of the main article can be satisfied.
Thus the assumption is not merely conditional on an abstract approximation axiom: there are finite-dimensional, norm-constrained architectures for which the required uniform \(C^{1,1}\) budget, the output constraint into \(K_Z\), and the \(C^1\)-approximation property hold simultaneously.

\begin{proposition}[A norm-constrained spline class satisfying Assumption~A4]
\label{prop:spline-A4}
Let \((\Phi_0,F_0)\) be the \(C^{1,1}\) exact representatives from
Lemma~\ref{lem:smooth-realizer}, and let
\[
\Pi_Z:Z\to K_Z,
\qquad
\mathcal O_Z\subset Z
\]
be the range-enforcing map and identity neighborhood from that lemma:
\[
\psi_0(S)\Subset\mathcal O_Z,
\qquad
\Pi_Z(z)=z
\quad\text{for every }z\in\mathcal O_Z.
\]

Choose compact rectangular boxes
\[
\Omega_\Phi\subset\mathbb R^{d_x},
\qquad
\Omega_F\subset\mathbb R^{d_z+d_a},
\]
such that
\[
H(S)\Subset\operatorname{int}\Omega_\Phi,
\qquad
K_Z\times A\Subset\operatorname{int}\Omega_F.
\]
Choose a fixed \(C^{1,1}\) map
\[
\rho_\Phi:\mathbb R^{d_x}\to\Omega_\Phi
\]
that equals the identity on an open neighborhood of \(H(S)\). Such a map
may be obtained by applying a fixed smooth coordinatewise saturation to
the rectangular box.

For mesh widths \(h_\Phi,h_F>0\), define
\[
\begin{aligned}
\mathcal A_\Phi(h_\Phi)
:=
\Bigl\{
\Phi(x)=\Pi_Z(u(\rho_\Phi(x))):\;&
 u\in\mathcal S_{h_\Phi}^{d_z}(\Omega_\Phi),\\
&
\|u\|_{C^0(\Omega_\Phi)}
+
\|Du\|_{C^0(\Omega_\Phi)}
+
\Lip(Du;\Omega_\Phi)
\le B_\Phi
\Bigr\},
\end{aligned}
\]
and
\[
\begin{aligned}
\mathcal A_F(h_F)
:=
\Bigl\{
F(z,a)=\Pi_Z(v(z,a)):\;&
 v\in\mathcal S_{h_F}^{d_z}(\Omega_F),\\
&
\|v\|_{C^0(\Omega_F)}
+
\|Dv\|_{C^0(\Omega_F)}
+
\Lip(Dv;\Omega_F)
\le B_F
\Bigr\},
\end{aligned}
\]
where \(F\) is restricted to \(K_Z\times A\).

For each sufficiently large integer capacity level \(W\), choose admissible uniform mesh widths \(h_\Phi(W)\) and \(h_F(W)\) such that
\[
h_\Phi(W)\asymp W^{-1/d_x},
\qquad
h_F(W)\asymp W^{-1/(d_z+d_a)},
\]
with comparison constants independent of \(W\). For example, one may choose the nearest admissible tensor-grid resolutions to the displayed powers.
Then, for sufficiently large fixed constants \(B_\Phi,B_F\), the classes
\[
\mathcal A_\Phi(W):=\mathcal A_\Phi(h_\Phi(W)),
\qquad
\mathcal A_F(W):=\mathcal A_F(h_F(W))
\]
satisfy Assumption~A4 of the main article. More precisely, there exist
\[
(\Phi_W,F_W)\in
\mathcal A_\Phi(W)\times\mathcal A_F(W)
\]
such that
\[
\|\Phi_W-\Phi_0\|_{C^1(H(S))}
+
\|F_W-F_0\|_{C^1(\psi_0(S)\times A)}
\le
\delta(W),
\]
where
\[
\delta(W)
\le
C_\Phi W^{-1/d_x}
+
C_F W^{-1/(d_z+d_a)}
\longrightarrow0.
\]
The Lipschitz and \(C^{1,1}\) budgets of the two classes are independent
of \(W\), and each class is compact, hence closed, in the \(C^1\)
topology. Each vector-valued class has \(O(d_zW)\) scalar coefficients.
\end{proposition}

\begin{proof}
We divide the argument into four steps.

\medskip\noindent
\textbf{Step 1: \(C^{1,1}\) extensions on the coordinate boxes.}
The encoder \(\Phi_0\) is defined on all of \(X=\mathbb R^{d_x}\), so we set
\[
\overline\Phi_0:=\Phi_0|_{\Omega_\Phi}.
\]

By the definition of \(C^{1,1}(K_Z\times A;K_Z)\), the transition map
\(F_0\) admits a \(C^{1,1}\) extension \(F_0^{\rm ext}\) to an open
neighborhood \(U_F\) of \(K_Z\times A\). We choose a smooth cut-off
\[
\chi_F\in C_c^\infty(U_F),
\qquad
0\le\chi_F\le1,
\]
satisfying \(\chi_F=1\) on an open neighborhood of \(K_Z\times A\). Fix
\(z_\ast\in Z\), we define
\[
\overline F_0
:=
z_\ast+\chi_F(F_0^{\rm ext}-z_\ast)
\]
on \(U_F\), extending the compactly supported second term by zero to
\(\mathbb R^{d_z+d_a}\). Then
\[
\overline F_0\in C^{1,1}(\mathbb R^{d_z+d_a};Z)
\]
and \(\overline F_0=F_0^{\rm ext}\) on an open neighborhood of \(K_Z\times A\); in particular, \(\overline F_0=F_0\) on \(K_Z\times A\).

We now define
\[
E_\Phi:=H(S),
\qquad
E_F:=\psi_0(S)\times A.
\]
The exact-realization identities give
\[
\overline\Phi_0(E_\Phi)=\psi_0(S)
\]
and
\[
\overline F_0(E_F)
=
\{\psi_0(G(s,a)):(s,a)\in S\times A\}
\subset\psi_0(S).
\]
Thus both target images are compactly contained in \(\mathcal O_Z\).

\medskip\noindent
\textbf{Step 2: localized spline approximation.}
We apply Lemma~\ref{lem:localized-spline-A4} to
\[
(\Omega,E,f_0)
=
(\Omega_\Phi,E_\Phi,\overline\Phi_0)
\]
and to
\[
(\Omega,E,f_0)
=
(\Omega_F,E_F,\overline F_0),
\]
with \(K=K_Z\), \(\Pi_K=\Pi_Z\), and \(\mathcal O=\mathcal O_Z\).
We obtain raw spline quasi-interpolants
\[
u_{h_\Phi}\in\mathcal S_{h_\Phi}^{d_z}(\Omega_\Phi),
\qquad
v_{h_F}\in\mathcal S_{h_F}^{d_z}(\Omega_F),
\]
whose \(C^{1,1}\) norms are bounded independently of the mesh widths and
such that
\[
\|\Pi_Z\circ u_{h_\Phi}-\Phi_0\|_{C^1(H(S))}
\le C_\Phi h_\Phi,
\]
\[
\|\Pi_Z\circ v_{h_F}-F_0\|_{C^1(\psi_0(S)\times A)}
\le C_Fh_F.
\]
By Lemma \ref{lem:spline-A4}, we can choose \(B_\Phi,B_F\) such that these constants are larger than the corresponding uniform raw-spline bounds and are independent of \(h_\Phi,h_F\), and hence of
\(W\).

We define
\[
\Phi_W(x)
:=
\Pi_Z\bigl(u_{h_\Phi(W)}(\rho_\Phi(x))\bigr),
\]
\[
F_W(z,a)
:=
\Pi_Z\bigl(v_{h_F(W)}(z,a)\bigr).
\]
Then, we get
\[
\Phi_W\in\mathcal A_\Phi(W),
\qquad
F_W\in\mathcal A_F(W).
\]
Because \(\rho_\Phi\) equals the identity on a neighborhood of \(H(S)\), we have
\[
\rho_\Phi(x)=x,
\qquad
D\rho_\Phi(x)=I
\quad\text{for }x\in H(S).
\]
Therefore, we arrive at
\[
\|\Phi_W-\Phi_0\|_{C^1(H(S))}
\le C_\Phi h_\Phi(W),
\]
\[
\|F_W-F_0\|_{C^1(\psi_0(S)\times A)}
\le C_Fh_F(W).
\]
The stated estimate for \(\delta(W)\) follows.

\medskip\noindent
\textbf{Step 3: uniform \(C^{1,1}\) budgets.}
We consider the encoder
\[
\Phi=\Pi_Z\circ u\circ\rho_\Phi
\in\mathcal A_\Phi(W),
\]
and set
\[
R_1:=\|D\rho_\Phi\|_{C^0},
\qquad
R_2:=\Lip(D\rho_\Phi).
\]
The raw-spline constraint gives
\[
\|D(u\circ\rho_\Phi)\|_{C^0}
\le B_\Phi R_1,
\]
\[
\Lip(D(u\circ\rho_\Phi))
\le B_\Phi(R_1^2+R_2).
\]
All values of \(u\circ\rho_\Phi\) lie in the fixed ball
\(\overline B_{B_\Phi}(0)\). On this ball, we define
\[
P_1:=\|D\Pi_Z\|_{C^0},
\qquad
P_2:=\Lip(D\Pi_Z).
\]
The chain rule yields
\[
\|D\Phi\|_{C^0}
\le
P_1B_\Phi R_1
\]
and
\[
\Lip(D\Phi)
\le
P_1B_\Phi(R_1^2+R_2)
+
P_2B_\Phi^2R_1^2.
\]
Since \(\Phi\) takes values in the compact set \(K_Z\), this gives a
\(C^{1,1}\) budget independent of \(W\).

For \(F=\Pi_Z\circ v\in\mathcal A_F(W)\), the same argument gives
\[
\|DF\|_{C^0}
\le
\widetilde P_1B_F,
\qquad
\Lip(DF)
\le
\widetilde P_1B_F+
\widetilde P_2B_F^2,
\]
where \(\widetilde P_1,\widetilde P_2\) are bounds for \(D\Pi_Z\) and
\(\Lip(D\Pi_Z)\) on the fixed ball containing the raw spline values.
Thus the transition classes also have uniform \(C^{1,1}\) budgets, and
both classes map into \(K_Z\) by construction.

\medskip\noindent
\textbf{Step 4: finite capacity and \(C^1\)-closedness.}
At fixed mesh width, each spline space is finite-dimensional. The raw
coefficient sets determined by
\[
\|u\|_{C^0}+\|Du\|_{C^0}+\Lip(Du)\le B_\Phi
\]
and
\[
\|v\|_{C^0}+\|Dv\|_{C^0}+\Lip(Dv)\le B_F
\]
are closed and bounded in finite-dimensional spaces. By equivalence of
norms, they are compact.

The maps
\[
u\longmapsto\Pi_Z\circ u\circ\rho_\Phi
\]
and
\[
v\longmapsto(\Pi_Z\circ v)|_{K_Z\times A}
\]
are continuous into the \(C^1\) topology. Hence
\(\mathcal A_\Phi(W)\) and \(\mathcal A_F(W)\) are continuous images of
compact sets and are therefore compact, in particular closed, in
\(C^1\).

Finally, a tensor-product spline space in input dimension \(d\) has
\(O(h^{-d})\) scalar basis functions per output coordinate. The chosen
mesh widths satisfy
\[
h_\Phi(W)^{-d_x}\asymp W,
\qquad
h_F(W)^{-(d_z+d_a)}\asymp W.
\]
Thus each vector-valued class has \(O(d_zW)\) scalar coefficients. All
requirements of Assumption~A4 of the main article are satisfied.
\end{proof}

\subsection{Margin-clearing competitors and a nonempty parameter regime}
The margin-clearing comparator is the bridge between realizability and the empirical objective.
\begin{proposition}[Margin-clearing competitors]
\label{prop:margin-clearing}
Adopt the hypotheses of Lemma~\ref{lem:smooth-realizer} and Assumption~\hyperref[ass:A4]{A4}.
We define
\[
\kappa_0:=\inf_{(s,v)\in\mathcal U}\|D\psi_0(s)[v]\|_Z,
\qquad
\alpha_0(\rho):=\inf_{(s,s')\in P_\rho}\|\psi_0(s)-\psi_0(s')\|_Z.
\]
By the bi-Lipschitz property of \(\psi_0\), \(\kappa_0>0\) and \(\alpha_0(\rho)>0\), we choose margins
\[0<\kappa<\kappa_0,\qquad 0<\alpha<\alpha_0(\rho).\]

Then, for every \(W\) sufficiently large, the pair \((\Phi_W,F_W)\in\mathcal A_\Phi(W)\times\mathcal A_F(W)\) from Assumption~\hyperref[ass:A4]{A4} is a margin-clearing competitor.
More precisely, writing \(\psi_W:=\Phi_W\circ H\), we get
\[
\|D\psi_W(s)[v]\|_Z\ge\kappa \quad\text{for all }(s,v)\in\mathcal U,
\qquad
\|\psi_W(s)-\psi_W(s')\|_Z\ge\alpha \quad\text{for all }(s,s')\in P_\rho .
\]
Consequently,
\[
\mathcal N_{\rm loc}(\Phi_W)=\mathcal N_{\rm glob}(\Phi_W)=\mathcal N_{\rm met}(\Phi_W)=0,
\]
and the same identities hold for the empirical penalties, for every sample:
\[
\widehat{\mathcal N_{\rm loc}}_{n}(\Phi_W)
=\widehat{\mathcal N_{\rm glob}}_{n}(\Phi_W)
=\widehat{\mathcal N_{\rm met}}_{n}(\Phi_W)=0 .
\]
Moreover,
\[
\mathcal E_{\rm pred}(\Phi_W,F_W)\le\beta(W),\qquad \beta(W):=C_\sharp\,\delta(W)^2\to0,
\]
where \(C_\sharp\) depends only on \(H\), \(G\), and the uniform Lipschitz and \(C^{1,1}\) budgets.
\end{proposition}
\begin{proof}
	We have the identities
	\begin{align*}
		\psi_W:=\Phi_W\circ H,\quad \psi_0:=\Phi_0\circ H,
	\end{align*}
and by the chain rule,
\[
D\psi_W(s)-D\psi_0(s)
=
\bigl(D\Phi_W(H(s))-D\Phi_0(H(s))\bigr)DH(s).
\]
Consequently, we estimate
\[
\|\psi_W-\psi_0\|_{C^1(S)}
\le
\|\Phi_W-\Phi_0\|_{C^0(H(S))}
+\|D\Phi_W-D\Phi_0\|_{C^0(H(S))}\|DH\|_{C^0(S)}
\le
C^\sharp_H\delta(W),
\]
where \(C^\sharp_H:=1+\|DH\|_{C^0(S)}\).

For every \((s,v)\in\mathcal U=S\times\mathbb S^{d_s-1}\), we have
\[
\begin{aligned}
\|D\psi_W(s)[v]\|_Z
&\ge
\|D\psi_0(s)[v]\|_Z
-
\|(D\psi_W(s)-D\psi_0(s))[v]\|_Z \\
&\ge
\kappa_0
-
\|D\psi_W-D\psi_0\|_{C^0(S)} \\
&\ge
\kappa_0
-
C_H^\sharp\delta(W).
\end{aligned}
\]
Since \(0<\kappa<\kappa_0\) and \(\delta(W)\to0\), we may choose \(W\) sufficiently large so that
\[
C_H^\sharp\delta(W)\le \kappa_0-\kappa.
\]
Then, it follows that
\[
\|D\psi_W(s)[v]\|_Z\ge \kappa
\qquad \text{for all }(s,v)\in\mathcal U.
\]
Therefore, we have proved
\[
	\mathcal N_{\rm loc}(\Phi_W)=0
	\quad\text{and}\quad
	\widehat{\mathcal N_{\rm loc}}_n(\Phi_W)=0.
\]

We next show the separated-point margin.
For every \((s,s')\in P_\rho\), we estimate
\[
\begin{aligned}
\|\psi_W(s)-\psi_W(s')\|_Z
&\ge
\|\psi_0(s)-\psi_0(s')\|_Z \\
&\quad
-
\|\psi_W(s)-\psi_0(s)\|_Z
-
\|\psi_W(s')-\psi_0(s')\|_Z \\
&\ge
\alpha_0(\rho)
-
2\|\psi_W-\psi_0\|_{C^0(S)} \\
&\ge
\alpha_0(\rho)-2C_H^\sharp\delta(W).
\end{aligned}
\]
Since \(0<\alpha<\alpha_0(\rho)\), we may choose \(W\) sufficiently large so that
\[
2C_H^\sharp\delta(W)\le \alpha_0(\rho)-\alpha.
\]
Then, it follows that
\[
\|\psi_W(s)-\psi_W(s')\|_Z\ge \alpha
\qquad \text{for all }(s,s')\in P_\rho.
\]
Therefore, we get
\[
	\mathcal N_{\rm glob}(\Phi_W)=0
	\quad \text{and}\quad
	\widehat{\mathcal N_{\rm glob}}_n(\Phi_W)=0.
\]
Consequently, we obtain
\[
	\mathcal N_{\rm met}(\Phi_W)=0
	\quad\text{and}\quad
	\widehat{\mathcal N_{\rm met}}_n(\Phi_W)=0.
\]

It remains to estimate the prediction loss. Since \((\Phi_0,F_0)\) is the smooth exact realizer given by Lemma~\ref{lem:smooth-realizer}, we have
\[
\psi_0(G(s,a))
=
F_0(\psi_0(s),a)
\qquad \text{for all }(s,a)\in S\times A.
\]
Using the exact identity for \((\psi_0,F_0)\), we obtain, for \((s,a)\in M\),
\[
\begin{aligned}
\|R[\Phi_W,F_W](s,a)\|_Z
&=
\|
\psi_W(G(s,a))-F_W(\psi_W(s),a)
\|_Z \\
&\le
\|\psi_W(G(s,a))-\psi_0(G(s,a))\|_Z \\
&\quad+
\|F_0(\psi_0(s),a)-F_W(\psi_W(s),a)\|_Z.
\end{aligned}
\]
Since \(G(s,a)\in S\), the first term is bounded by
\[
\|\psi_W(G(s,a))-\psi_0(G(s,a))\|_Z
\le
\|\psi_W-\psi_0\|_{C^0(S)}
\le
C_H^\sharp\delta(W).
\]

For the second term, we add and subtract \(F_W(\psi_0(s),a)\), which gives
\[
\begin{aligned}
\|F_0(\psi_0(s),a)-F_W(\psi_W(s),a)\|_Z
&\le
\|F_0(\psi_0(s),a)-F_W(\psi_0(s),a)\|_Z \\
&\quad+
\|F_W(\psi_0(s),a)-F_W(\psi_W(s),a)\|_Z.
\end{aligned}
\]
By Assumption~\hyperref[ass:A4]{A4}, we estimate
\[
\|F_0(\psi_0(s),a)-F_W(\psi_0(s),a)\|_Z
\le
\|F_W-F_0\|_{C^0(\psi_0(S)\times A)}
\le
\delta(W).
\]

For the remaining term, we use the uniform Lipschitz budget of the admissible transition class.
Namely, there exists a constant \(L_F^{\mathrm{bud}}\), independent of \(W\), such that
\[
\|F_W(z,a)-F_W(z',a)\|_Z
\le
L_F^{\mathrm{bud}}\|z-z'\|_Z
\qquad
\text{for all }z,z'\in K_Z,\ \text{and } a\in A.
\]
Therefore, we estimate
\[
\begin{aligned}
\|F_W(\psi_0(s),a)-F_W(\psi_W(s),a)\|_Z
&\le
L_F^{\mathrm{bud}}
\|\psi_0(s)-\psi_W(s)\|_Z \\
&\le
L_F^{\mathrm{bud}}C_H^\sharp\delta(W).
\end{aligned}
\]
Combining the preceding estimates yields
\[
\|R[\Phi_W,F_W](s,a)\|_Z
\le
C_R\delta(W)
\qquad \text{for all }(s,a)\in M,
\]
where
\[
C_R
:=
1+(1+L_F^{\mathrm{bud}})C_H^\sharp.
\]

Since \(\mu\) is a probability measure on \(M\), we obtain
\[
\begin{aligned}
	\mathcal E_{\rm pred}(\Phi_W,F_W)
&=
\int_M
\|R[\Phi_W,F_W](s,a)\|_Z^2
\dd\mu(s,a) \\
&\le
\left(
\sup_{(s,a)\in M}
\|R[\Phi_W,F_W](s,a)\|_Z
\right)^2 \\
&\le
C_R^2\delta(W)^2.
\end{aligned}
\]
Thus the desired estimate holds with
\[
C_\sharp:=C_R^2,
\qquad
\beta(W):=C_\sharp\delta(W)^2.
\]
Since \(\delta(W)\to0\), we have \(\beta(W)\to0.\)
\end{proof}
\begin{lemma}[Nonempty parameter regime]
\label{lem:nonempty-parameters}
Let \(S\subset\mathbb R^{d_s}\) be compact with \(\operatorname{diam}(S)>0\), and let \(\psi_0:S\to Z\) be the \(C^{1,1}\) bi-Lipschitz embedding from Lemma~\ref{lem:smooth-realizer}, with co-Lipschitz constant \(c_{\psi_0}>0\).
Assume
\[
\kappa_0:=\inf_{(s,v)\in\mathcal U}\|D\psi_0(s)[v]\|_Z>0,
\]
and fix \(\kappa\) with \(0<\kappa<\kappa_0\). Let
\[
B_{\mathrm{loc}}:=\sup_{W}\ \sup_{\Phi\in\mathcal A_\Phi(W)}\Lip\bigl(D(\Phi\circ H)\bigr),
\]
which is finite by Assumption~\hyperref[ass:A4]{A4} together with the
\(C^{1,1}\) regularity of \(H\). We use the convention
\[
\frac{\kappa}{4B_{\mathrm{loc}}}:=+\infty
\qquad\text{when }B_{\mathrm{loc}}=0,
\]
and choose
\[
0<\rho<\min\Bigl\{\operatorname{diam}(S),\ \frac{\kappa}{4B_{\mathrm{loc}}}\Bigr\},
\qquad\text{so that}\qquad
B_{\mathrm{loc}}\,\rho\le\frac{\kappa}{4}.
\]
Since \(0<\rho<\operatorname{diam}(S)\), the compact set
\[
P_\rho:=\{(s,s')\in S\times S:\ |s-s'|\ge\rho\}
\]
is nonempty and disjoint from the diagonal. By the co-Lipschitz constant  \(c_{\psi_0}\), we have
\[
\alpha_0(\rho):=\inf_{(s,s')\in P_\rho}\|\psi_0(s)-\psi_0(s')\|_Z
\ge c_{\psi_0}\,\rho>0,
\]
so any \(\alpha\in(0,\alpha_0(\rho))\) is a valid margin. In particular, the set of parameter tuples \((\kappa,\rho,\alpha)\) satisfying
\[
0<\kappa<\kappa_0,\qquad B_{\mathrm{loc}}\,\rho\le\frac{\kappa}{4},\qquad
0<\alpha<\alpha_0(\rho)
\]
is nonempty.
\end{lemma}
\begin{proof}
	By the uniform budgets in Assumption~\hyperref[ass:A4]{A4}, the \(C^{1,1}\) regularity of \(H\), and the chain rule, the quantity
\[
B_{\mathrm{loc}}
=
\sup_{W}\sup_{\Phi\in\mathcal A_\Phi(W)}
\Lip\bigl(D(\Phi\circ H)\bigr)
<\infty.
\]

By the convention \(\kappa/(4B_{\mathrm{loc}})=+\infty\) when \(B_{\mathrm{loc}}=0\), and since \(\kappa>0\) and \(\operatorname{diam}(S)>0\), the set \(\bigl(0,\min\{\operatorname{diam}(S),\kappa/(4B_{\mathrm{loc}})\}\bigr)\) is nonempty. We may therefore choose \(\rho\) in it.
Then,
\[0<\rho<\operatorname{diam}(S)\quad\text{and}\quad B_{\mathrm{loc}}\rho\le\frac{\kappa}{4}.\]

We next verify the properties of \(P_\rho\).
We observe that the map \((s,s')\mapsto|s-s'|\) is continuous on the compact set \(S\times S\), so \[P_\rho=\{(s,s')\in S\times S:\ |s-s'|\ge\rho\}\] is compact, and the diameter is attained: there exist \(s_1,s_2\in S\) with \(|s_1-s_2|=\operatorname{diam}(S)\).

Since \(\rho<\operatorname{diam}(S)\), we have \((s_1,s_2)\in P_\rho\), so \(P_\rho\) is nonempty. Moreover, \((s,s')\in P_\rho\) implies \(|s-s'|\ge\rho>0\), hence \(s\ne s'\). Thus, \(P_\rho\) is disjoint from the diagonal.

Finally, we show the positivity of \(\alpha_0(\rho)\). By the co-Lipschitz bound for \(\psi_0\),
\[
	\|\psi_0(s)-\psi_0(s')\|_Z\ge c_{\psi_0}|s-s'|\qquad\text{for all }s,s'\in S.
\]
This implies, for every \((s,s')\in P_\rho\),
\[
\|\psi_0(s)-\psi_0(s')\|_Z\ge c_{\psi_0}|s-s'|\ge c_{\psi_0}\rho.
\]
Taking the infimum, we obtain
\[\alpha_0(\rho)\ge c_{\psi_0}\rho>0.\]
We may therefore choose
\(
0<\alpha<\alpha_0(\rho).
\)
Together with the already fixed choice
\(
0<\kappa<\kappa_0
\)
and the choice of \(\rho\) satisfying
\(
B_{\mathrm{loc}}\rho\le \kappa/4,
\)
we obtain parameters satisfying
\[
B_{\mathrm{loc}}\rho\le\frac{\kappa}{4},
\qquad
0<\kappa<\kappa_0,
\qquad
0<\alpha<\alpha_0(\rho).
\]
Therefore the admissible parameter regime is nonempty.
\end{proof}

\section{Architecture-specific covering estimates}
\label{supp:spline-covering}
The main article proves the abstract uniform deviation bound and the finite-parametric estimate.  Here we verify the required coefficient stability and parameter-Lipschitz bounds for the concrete spline dictionaries of Supplementary Section~\ref{supp:spline-approximation}.

Let us fix the parameters \(\kappa,\alpha>0\).
Define the three function classes
\[
\mathcal G_M(W)
:=
\left\{
(s,a)\mapsto \|R[\Phi,F](s,a)\|_Z^2:
(\Phi,F)\in\mathcal A(W)
\right\},
\]
\[
\mathcal G_U(W)
:=
\left\{
(s,v)\mapsto
\bigl[\kappa-\|D(\Phi\circ H)(s)[v]\|_Z\bigr]_+^2:
\Phi\in\mathcal A_\Phi(W)
\right\},
\]
\[
\mathcal G_P(W)
:=
\left\{
(s,s')\mapsto
\bigl[\alpha-\|(\Phi\circ H)(s)-(\Phi\circ H)(s')\|_Z\bigr]_+^2:
\Phi\in\mathcal A_\Phi(W)
\right\}.
\]

Let \(\mu_M,\mu_U,\mu_P\) be the population measures on \(M,\mathcal U,P_\rho\), respectively, and let \(\mu_{M,n},\mu_{U,n},\mu_{P,n}\) be the corresponding empirical measures from independent samples of size \(n\).

We denote by \(N_M(r;W),N_U(r;W),N_P(r;W)\) the \(r\)-covering numbers of \(\mathcal G_M(W),\mathcal G_U(W),\mathcal G_P(W)\) in the \(C^0\) norm.

\begin{lemma}[Uniform coefficient stability for the spline dictionaries]
\label{lem:spline-coefficient-stability}
Fix the spline degree and a compact rectangular box \(\Omega\subset\mathbb R^d\). For the normalized tensor-product B-spline bases \(\{B_j^h\}_j\) on the admissible uniform meshes used above, there exists a constant \(C_{\rm stab}\ge1\), independent of \(h\), such that every scalar spline
\[
s_c=\sum_j c_jB_j^h
\]
satisfies
\[
C_{\rm stab}^{-1}\|c\|_{\ell^\infty}
\le
\|s_c\|_{L^\infty(\Omega)}
\le
C_{\rm stab}\|c\|_{\ell^\infty}.
\]
The same estimate holds componentwise for vector-valued splines.
\end{lemma}

\begin{proof}
The upper bound is immediate from nonnegativity and the partition-of-unity property of the normalized B-spline basis. Indeed, for every \(x\in\Omega\), we have
\[
|s_c(x)|
=
\left|\sum_j c_j B_j^h(x)\right|
\le
\|c\|_{\ell^\infty}\sum_j B_j^h(x)
=
\|c\|_{\ell^\infty}.
\]
Thus, we get
\[
\|s_c\|_{L^\infty(\Omega)}
\le
\|c\|_{\ell^\infty}.
\]

It remains to prove the inverse estimate. Let \(\mathcal T_h\) denote the underlying uniform tensor-product mesh, and for each cell \(K\in\mathcal T_h\) let
\[
I(K):=\{j:B_j^h|_K\not\equiv 0\}
\]
be the set of indices of B-splines active on \(K\). Since the spline degree is fixed, the cardinality of \(I(K)\) is uniformly bounded, independently of
\(h\).

Let \(A_K:\widehat K\to K\) be the affine map from the reference cube \(\widehat K=[0,1]^d\) to \(K\). For \(j\in I(K)\), we define the rescaled restriction
\[
\widehat B_{j,K}:=B_j^h\circ A_K .
\]
By uniformity of the mesh and by the fixed degree assumption, the collection of possible families
\[
\{\widehat B_{j,K}:j\in I(K)\}
\]
belongs to a finite set of reference patterns. This finite set includes the interior pattern and the finitely many boundary knot patterns.

For each such reference pattern \(P=\{\beta_1,\dots,\beta_m\}\), we consider the linear map
\[
T_P:\mathbb R^m\to C(\widehat K),
\qquad
T_P b:=\sum_{r=1}^m b_r\beta_r .
\]
The functions \(\beta_1,\dots,\beta_m\) are locally linearly independent. Hence \(T_P\) is injective. Since \(\mathbb R^m\) is finite-dimensional, the quantity
\[
C_P
:=
\inf_{\|b\|_{\ell^\infty}=1}
\|T_P b\|_{L^\infty(\widehat K)}>0.
\]
Because only finitely many patterns \(P\) can occur, we may define
\[
C_{\min}:=\min_P C_P>0.
\]
Therefore, for every cell \(K\) and every coefficient vector \(b=(b_j)_{j\in I(K)}\), we have
\[
\|b\|_{\ell^\infty}
\le
C_{\min}^{-1}
\left\|
\sum_{j\in I(K)} b_j B_j^h
\right\|_{L^\infty(K)} .
\]

Now, we fix an arbitrary coefficient index \(j\) and choose a mesh cell \(K_j\) on which \(B_j^h\) is active. Applying the preceding local estimate to the coefficient vector \((c_i)_{i\in I(K_j)}\), we obtain
\[
|c_j|
\le
\|(c_i)_{i\in I(K_j)}\|_{\ell^\infty}
\le
C_{\min}^{-1}
\|s_c\|_{L^\infty(K_j)}
\le
C_{\min}^{-1}
\|s_c\|_{L^\infty(\Omega)} .
\]
Taking the supremum over \(j\) gives
\[
\|c\|_{\ell^\infty}
\le
C_{\min}^{-1}\|s_c\|_{L^\infty(\Omega)}.
\]
Equivalently,
\[
C_{\min}\|c\|_{\ell^\infty}
\le
\|s_c\|_{L^\infty(\Omega)}.
\]

Combining this lower bound with the upper bound, the desired estimate holds with, for example,
\[
C_{\rm stab}:=\max\{1,C_{\min}^{-1}\}.
\]
The constant depends only on the spline degree, the dimension, and the
admissible boundary knot patterns, and is independent of \(h\).

For vector-valued splines, the same argument applies to each component separately, which gives the stated componentwise estimate.
\end{proof}

The following proposition shows that the spline architecture satisfies the finite-parametric hypothesis of the finite-parametric deviation corollary in the main article.

\begin{proposition}[Spline parameter-Lipschitz bound]
\label{prop:spline-parametric}
Consider the range-constrained tensor-product cubic B-spline classes of Proposition~\ref{prop:spline-A4}, parametrized by the vector \(\theta\) of spline coefficients of \((\Phi,F)\), constrained to a box \([-R_W,R_W]^{p_W}\).
Then the hypothesis of the finite-parametric deviation corollary in the main article holds for each \(Y\in\{M,U,P\}\): each integrand class \(\mathcal G_Y(W)\) satisfies
\[
\|g_\theta-g_{\theta'}\|_{C^0(Y)}\le L_{\theta,Y,W}\|\theta-\theta'\|_{\ell^\infty},
\]
and there are constants \(C_Y,q_Y\), independent of \(n\), with
\[
R_W L_{\theta,Y,W}\le C_Y W^{q_Y}.
\]
Consequently, the finite-parametric estimate of the finite-parametric deviation corollary in the main article applies to the norm-constrained spline architecture of Proposition~\ref{prop:spline-A4}.
\end{proposition}

\begin{proof}
	We write the proof for one fixed capacity level \(W\).

	Let \(\Omega_\Phi\subset X\) and \(\Omega_F\subset Z\times\R^{d_a}\) be the compact coordinate boxes containing \(H(S)\) and \(K_Z\times A\), respectively, as in Proposition~\ref{prop:spline-A4}.

	We recall the construction of the encoder and transition spline classes \(\mathcal A_\Phi(W)\) and \(\mathcal A_F(W)\) in Proposition~\ref{prop:spline-A4}. They have different mesh widths, which we denote by \(h_\Phi\) and \(h_F\).

The encoder dictionary is a tensor-product cubic B-spline dictionary on \(\Omega_\Phi\), and the transition dictionary is a tensor-product cubic B-spline dictionary on \(\Omega_F\).

We first record the basic coefficient-stability estimates. A vector-valued spline map on a \(d\)-dimensional coordinate box has the form
\[
u_c(x)=\sum_{j=1}^{p}c_j B_j^h(x),
\]
where \(B_j^h\) are scalar tensor-product cubic B-splines, \(c_j\in\R^m\) are the output coefficients, and \(p\asymp h^{-d}\) is the number of basis functions.
Applied to the two classes, the encoder splines on \(\Omega_\Phi\subset\R^{d_x}\) have \(p_\Phi\asymp h_\Phi^{-d_x}\) basis functions and the transition splines on \(\Omega_F\subset\R^{d_z+d_a}\) have \(p_F\asymp h_F^{-(d_z+d_a)}\).
Since both maps are \(\R^{d_z}\)-valued, the combined parameter vector \(\theta=(c_\Phi,c_F)\) has dimension
\[
p_W = d_z\,(p_\Phi+p_F)\asymp d_z\bigl(h_\Phi^{-d_x}+h_F^{-(d_z+d_a)}\bigr),
\]
which is the ambient dimension of the box \([-R_W,R_W]^{p_W}\) in
the finite-parametric deviation corollary in the main article.

Since the spline degree is fixed and the grid is uniform, the basis has uniformly bounded overlap. Therefore there exist constants \(C_0,C_1>0\), independent of \(h\), such that for every two coefficient vectors \(c,c'\),
\[
\|u_c-u_{c'}\|_{C^0}
\le
C_0\|c-c'\|_{\ell^\infty},
\]
and
\[
\|D(u_c-u_{c'})\|_{C^0}
\le
C_1h^{-1}\|c-c'\|_{\ell^\infty}.
\]
Here the constants depend only on the coordinate box, the dimension, and the fixed spline degree, but not on \(h\).

We now incorporate the output range-enforcing map. For the encoder and transition maps, we use the range-constrained form \(\widetilde u_c:=\Pi_Z\circ u_c,\)
where \(\Pi_Z:Z\to K_Z\) is the fixed $C^{1,1}$ range-enforcing map.
Since the norm-constrained spline class satisfies
\[
\|u_c\|_{C^0}+\|Du_c\|_{C^0}+\Lip(Du_c)\le B,
\]
all spline maps \(u_c\) take values in a fixed compact subset of \(Z\).
Hence, \(\|D\Pi_{K_Z}\|_{C^0}\) and \(\Lip(D\Pi_{K_Z})\) are bounded on that compact set by constants independent of \(h\) and \(W\).

For the values, the mean value theorem gives
\[
\begin{aligned}
\|\widetilde u_c-\widetilde u_{c'}\|_{C^0}
&=
\|\Pi_{K_Z}(u_c)-\Pi_{K_Z}(u_{c'})\|_{C^0}  \\
&\le
\|D\Pi_{K_Z}\|_{C^0}
\|u_c-u_{c'}\|_{C^0} \\
&\le
C\|c-c'\|_{\ell^\infty}.
\end{aligned}
\]
For the derivatives, by the chain rule,
\[
D\widetilde u_c(x)
=
D\Pi_{K_Z}(u_c(x))Du_c(x).
\]
Thus, we have the identity
\[
\begin{aligned}
&D\widetilde u_c(x)-D\widetilde u_{c'}(x) \\
&\quad=
\bigl[D\Pi_{K_Z}(u_c(x))-D\Pi_{K_Z}(u_{c'}(x))\bigr]Du_c(x)
+
D\Pi_{K_Z}(u_{c'}(x))\bigl[Du_c(x)-Du_{c'}(x)\bigr].
\end{aligned}
\]

Using the uniform \(C^1\)-bound on \(u_c\), the Lipschitz bound on
\(D\Pi_{K_Z}\), and the spline stability estimates above, we obtain
\[
\begin{aligned}
\|D\widetilde u_c-D\widetilde u_{c'}\|_{C^0}
&\le
C\|u_c-u_{c'}\|_{C^0}
+
C\|D(u_c-u_{c'})\|_{C^0} \\
&\le
C\left(1+h^{-1}\right)\|c-c'\|_{\ell^\infty}.
\end{aligned}
\]
For \(0<h\le1\), this gives
\[
\|D\widetilde u_c-D\widetilde u_{c'}\|_{C^0}
\le
Ch^{-1}\|c-c'\|_{\ell^\infty}.
\]

By Lemma~\ref{lem:spline-coefficient-stability} and the raw norm constraints,
\[
\|c_\Phi\|_{\ell^\infty}
\le
C_{\rm stab}B_\Phi,
\qquad
\|c_F\|_{\ell^\infty}
\le
C_{\rm stab}B_F.
\]
Thus the combined coefficient set is contained in a cube \([-R,R]^{p_W}\) with a radius \(R\) independent of \(W\); in the finite-parametric deviation corollary in the main article one may therefore take \(R_W=R\).

We apply these estimates separately to the transition maps and to the encoder maps
\[
\Phi_{c_\Phi}=\Pi_Z\circ u_{c_\Phi}\circ\rho_\Phi.
\]
Composition with the fixed map \(\rho_\Phi\) changes the preceding \(C^0\)- and \(C^1\)-parameter bounds only by constants depending on \(\|D\rho_\Phi\|_{C^0}\) and \(\Lip(D\rho_\Phi)\), not on \(W\). Moreover, \(\rho_\Phi=\operatorname{Id}\) on a neighborhood of \(H(S)\), so the value and derivative estimates entering the three integrand classes retain the displayed \(h_\Phi^{-1}\) scaling. Here \(c_\Phi\) and \(c_F\) denote the encoder and transition coefficient vectors.

We now verify the parameter-Lipschitz estimates for the three integrand classes.

Let \(\theta:=(c_\Phi,c_F)\) be the combined parameter vector. We define
\[
g_\theta^M(s,a)
:=
\|
\psi_{c_\Phi}(G(s,a))
-
F_{c_F}(\psi_{c_\Phi}(s),a)
\|_Z^2
\]
and
\[
R_\theta(s,a)
:=
\psi_{c_\Phi}(G(s,a))
-
F_{c_F}(\psi_{c_\Phi}(s),a),
\]
where \(\psi_{c_\Phi}:= \Phi_{c_\Phi}\circ H\).

Since the two terms in \(R_\theta\) take values in \(K_Z\), the residuals are uniformly bounded:
\[
	\|R_\theta(s,a)\|_Z\le D_Z:=\operatorname{diam}(K_Z).
\]

With \(\theta'=(c'_\Phi,c'_F)\), we estimate
\[
\left|g_\theta^M(s,a)-g_{\theta'}^M(s,a)\right|
\le
2D_Z\|R_\theta(s,a)-R_{\theta'}(s,a)\|_Z.
\]
We also estimate
\[
\begin{aligned}
\|R_\theta(s,a)-R_{\theta'}(s,a)\|_Z
&\le
\|
\psi_{c_\Phi}(G(s,a))
-
\psi_{c_\Phi'}(G(s,a))
\|_Z \\
&\quad+
\|
F_{c_F}(\psi_{c_\Phi}(s),a)
-
F_{c_F'}(\psi_{c_\Phi'}(s),a)
\|_Z.
\end{aligned}
\]
For the second term, we add and subtract \(F_{c_F}(\psi_{c_\Phi'}(s),a)\):
\[
\begin{aligned}
\|
F_{c_F}(\psi_{c_\Phi}(s),a)
-
F_{c_F'}(\psi_{c_\Phi'}(s),a)
\|_Z
&\le
\|
F_{c_F}(\psi_{c_\Phi}(s),a)
-
F_{c_F}(\psi_{c_\Phi'}(s),a)
\|_Z \\
&\quad+
\|
F_{c_F}(\psi_{c_\Phi'}(s),a)
-
F_{c_F'}(\psi_{c_\Phi'}(s),a)
\|_Z.
\end{aligned}
\]
The transition class has a uniform Lipschitz budget in \(z\), say
\(L_F^{\mathrm{bud}}\). Therefore
\[
\|
F_{c_F}(\psi_{c_\Phi}(s),a)
-
F_{c_F}(\psi_{c_\Phi'}(s),a)
\|_Z
\le
L_F^{\mathrm{bud}}
\|\psi_{c_\Phi}-\psi_{c_\Phi'}\|_{C^0(S)}.
\]
The parameter-value estimate for \(F\) gives
\[
\|
F_{c_F}(\psi_{c_\Phi'}(s),a)
-
F_{c_F'}(\psi_{c_\Phi'}(s),a)
\|_Z
\le
C\|c_F-c_F'\|_{\ell^\infty}.
\]
Combining these estimates yields
\[
\|R_\theta-R_{\theta'}\|_{C^0(M)}
\le
C\|\theta-\theta'\|_{\ell^\infty}.
\]
Consequently,
\[
\|g_\theta^M-g_{\theta'}^M\|_{C^0(M)}
\le
C\|\theta-\theta'\|_{\ell^\infty}.
\]
Thus one may take
\[
L_{\theta,M,W}\le C.
\]

A similar computation for \(\mathcal G_U(W)\) and \(\mathcal G_P(W)\) yields
\[L_{\theta,U,W}\le Ch_\Phi^{-1}\quad\text{and}\quad L_{\theta,P,W}\le C.\]

We recall from the construction in Proposition~\ref{prop:spline-A4} that
\[h_\Phi^{-1}\le C_\Phi W^{1/d_x}.\]
Since the coefficient radius \(R_W=R\) is independent of \(W\), it follows that
\[
R_WL_{\theta,M,W},
\qquad
R_WL_{\theta,U,W},
\qquad
R_WL_{\theta,P,W}
\]
all grow at most polynomially in \(W\).
\end{proof}

\section{Secondary deterministic-control diagnostics}
\label{supp:control-diagnostics}
The main article retains the deterministic simulation lemma and the modular planning corollary.  This section records the horizon-amplification table, a concentrated-residual construction, the adversarial-planner diagnostic, and secondary consequences of metric non-collapse.  Together these results show constructively why uniform residual control is the correct companion to metric faithfulness and how the constants produced by the representation theorem propagate to planning, Lyapunov, and reachability statements.

The rollout factor $C(T,L_F)$ is finite for every finite $T$ and $L_F$. It remains uniformly bounded in $T$ when $L_F<1$, grows linearly when $L_F=1$, and grows geometrically when $L_F>1$. The accumulated stage-cost factor $D(T,L_F)$ has one additional order in the neutral case: $D(T,1)=T(T-1)/2$. Thus a terminal-state estimate can be $O(\eta T)$ at $L_F=1$, while an undiscounted running-cost estimate is generally $O(\eta T^2)$. Table~\ref{tab:sim-lemma} reports $\eta\,C(T,L_F)$ for $\eta=10^{-2}$ and several values of $L_F$ and $T$.

\begin{table}[ht]
\centering
\small
\setlength{\tabcolsep}{8pt}
\begin{tabular}{rcccc}
\toprule
$T$ & $L_F=0.9$ & $L_F=1.0$ & $L_F=1.1$ & $L_F=1.5$ \\
\midrule
$1$    & $1.00\!\times\!10^{-2}$ & $1.00\!\times\!10^{-2}$ & $1.00\!\times\!10^{-2}$ & $1.00\!\times\!10^{-2}$ \\
$5$    & $4.10\!\times\!10^{-2}$ & $5.00\!\times\!10^{-2}$ & $6.10\!\times\!10^{-2}$ & $1.32\!\times\!10^{-1}$ \\
$10$   & $6.51\!\times\!10^{-2}$ & $1.00\!\times\!10^{-1}$ & $1.59\!\times\!10^{-1}$ & $1.13\!\times\!10^{0}$ \\
$25$   & $9.28\!\times\!10^{-2}$ & $2.50\!\times\!10^{-1}$ & $9.83\!\times\!10^{-1}$ & $5.05\!\times\!10^{2}$ \\
$50$   & $9.95\!\times\!10^{-2}$ & $5.00\!\times\!10^{-1}$ & $1.16\!\times\!10^{1}$ & $1.28\!\times\!10^{7}$ \\
\bottomrule
\end{tabular}
\caption{Cumulative rollout factor $\eta\,C(T,L_F)$ for $\eta=10^{-2}$. For $L_F<1$ it is bounded by $\eta/(1-L_F)$ uniformly in the horizon, for $L_F=1$ it equals $\eta T$, and for $L_F>1$ it grows geometrically.}
\label{tab:sim-lemma}
\end{table}

\subsection{Sharpness of the interpolation exponent}
\label{sec:exponent}
The exponent $1/(q+2)$ supplied by Lipschitz--Ahlfors interpolation cannot be improved under those hypotheses alone.
\begin{proposition}[Sharpness of the \(1/(q+2)\) exponent]
\label{prop:sharp-exponent}
Let \(q\in\mathbb N\), let
\[
Y=[0,1]^q
\]
with Lebesgue measure, and fix \(L>0\). The exponent \(1/(q+2)\) in the estimate
\[
\|u\|_{L^\infty(Y)}
\lesssim
\left(\int_Y u^2\,\dd y\right)^{1/(q+2)}
\]
cannot be improved uniformly over all nonnegative \(L\)-Lipschitz functions
\(u:Y\to\R\).

More precisely, for arbitrarily small \(h>0\), there exists a nonnegative \(L\)-Lipschitz function \(u_h:Y\to\R\) such that
\[
\|u_h\|_{L^\infty(Y)}=h,
\qquad
\int_Y u_h^2\,\dd y
\le
C_q L^{-q}h^{q+2},
\]
where \(C_q>0\) depends only on the dimension \(q\). Consequently, for every
\[
\alpha>\frac1{q+2},
\]
there is no constant \(C_\alpha<\infty\) such that
\[
\|u\|_{L^\infty(Y)}
\le
C_\alpha
\left(\int_Y u^2\,\dd y\right)^\alpha
\]
holds for all nonnegative \(L\)-Lipschitz functions \(u:Y\to\R\).
\end{proposition}
\begin{proof}

We set \(y_0:=\left(\tfrac12,\ldots,\tfrac12\right)\in (0,1)^q\) and choose \(h\in (0,L/2)\).
Then, we observe that
\[
\frac hL<\frac12=\operatorname{dist}(y_0,\partial Y),
\]
and hence \(B_{h/L}(y_0)\subset Y\).
We define the map
\[
u_h(y)
:=
\bigl[h-L|y-y_0|\bigr]_+,
\qquad y\in Y.
\]
This is the truncated cone of height \(h\), slope \(L\), and base radius \(h/L\).
We first verify the basic properties of \(u_h\). The map \(u_h\) is the composition of the \(1\)-Lipschitz map \(y\mapsto |y-y_0|\), the affine map \(r\mapsto h- L r\) on \(\mathbb R\), and the \(1\)-Lipschitz map \(r\mapsto[r]_+\) on \(\mathbb R\).
Therefore, \(u_h\) is \(L\)-Lipschitz.
It is also nonnegative by definition. Moreover, \(u_h(y_0)=h\), and since \(u_h\le h\) everywhere, we have
\[
\|u_h\|_{L^\infty(Y)}=h.
\]
We compute
\[
\int_Y u_h(y)^2\,\dd y
=
\int_{B_{h/L}(y_0)}
\bigl(h-L|y-y_0|\bigr)^2\,\dd y.
\]
Using the change of variables
\[
y=y_0+\frac hL \xi,
\qquad
\dd y=\left(\frac hL\right)^q\dd \xi,
\]
we obtain
\[
\begin{aligned}
\int_Y u_h(y)^2\,\dd y
&=
\int_{B_1(0)}
\bigl(h-L(h/L)|\xi|\bigr)^2
\left(\frac hL\right)^q
\dd \xi  \\
&=
h^2\left(\frac hL\right)^q
\int_{B_1(0)}
(1-|\xi|)^2\,\dd \xi  \\
&=
C_q\, L^{-q}\,h^{q+2},
\end{aligned}
\]
where
\(
C_q
:=
\int_{B_1(0)}(1-|\xi|)^2\,\dd \xi
<\infty.
\)

We now establish sharpness by contradiction. Suppose that there exist
\(\alpha>1/(q+2)\)
and a constant \(C_\alpha<\infty\) such that
\[
\|u\|_{L^\infty(Y)}
\le
C_\alpha
\left(\int_Y u^2\,\dd y\right)^\alpha
\]
for every nonnegative \(L\)-Lipschitz function \(u:Y\to\R\). Applying this estimate to
\(u_h\), we obtain
\[
h
\le
C_\alpha
\left(C_q L^{-q}h^{q+2}\right)^\alpha,
\]
or equivalently,
\[
1
\le
C_\alpha
\left(C_q L^{-q}\right)^\alpha
h^{\alpha(q+2)-1}.
\]
Since \(\alpha(q+2)-1>0\), the right-hand side tends to \(0\) as \(h\downarrow0\), which is a contradiction.
Therefore, no uniform estimate with exponent \(\alpha>1/(q+2)\) can hold.

The family \(\{u_h\}\) shows that an \(L^2\)-error of order \(h^{q+2}\) can conceal a pointwise error of order \(h\).
Hence the exponent \(1/(q+2)\) in the Lipschitz--Ahlfors \(L^2\)-to-\(L^\infty\) interpolation estimate is optimal under a Lipschitz-only regularity assumption.
\end{proof}

\subsection{Why an \texorpdfstring{$L^2$}{L2}-mean residual is not enough}
\label{sec:L2-not-enough}

The natural training objective minimizes an empirical squared residual,
while its population counterpart \(\mathcal E_{\rm pred}\) is an
\(L^2(\mu)\)-average. One might hope that a small average is sufficient for
downstream planning; in general, it is not. The obstruction is geometric:
in the deterministic simulation lemma in the main article, \(\eta\) is a supremum over
\((s,a)\in M\), whereas an \(L^2(\mu)\) loss can be small even when the
residual is large on a small region. An optimizing planner may then prefer
trajectories that pass through precisely those regions where the surrogate
is overly optimistic.

\paragraph{\bf Spike example.}
Let \(M=[0,1]^{d_m}\) carry normalized Lebesgue measure, and choose an
interior center \(x_0\) such that \(B_r(x_0)\subset[0,1]^{d_m}\). Define
\[
R_{h,r}(x)
:=
h\left[1-\frac{|x-x_0|}{r}\right]_+.
\]
Then \(\|R_{h,r}\|_{L^\infty}=h\), independently of \(r\). A radial
integration gives
\[
\|R_{h,r}\|_{L^2}
=
h r^{d_m/2}
\left[
\frac{2\omega_{d_m}}{(d_m+1)(d_m+2)}
\right]^{1/2},
\]
where
\[
\omega_{d_m}
=
\frac{\pi^{d_m/2}}{\Gamma(d_m/2+1)}
\]
is the volume of the \(d_m\)-dimensional unit ball.
Table~\ref{tab:spike} reports representative values.

\begin{table}[ht]
\centering
\small
\setlength{\tabcolsep}{8pt}
\begin{tabular}{ccccc}
\toprule
$d_m$ & $r$ & $\|R_{h,r}\|_{L^\infty}$ & $\|R_{h,r}\|_{L^2}$ & $\|R_{h,r}\|_{L^2}/\|R_{h,r}\|_{L^\infty}$ \\
\midrule
$2$ & $0.50$ & $1.00\!\times\!10^{-1}$ & $3.618\!\times\!10^{-2}$ & $3.618\!\times\!10^{-1}$ \\
$2$ & $0.10$ & $1.00\!\times\!10^{-1}$ & $7.236\!\times\!10^{-3}$ & $7.236\!\times\!10^{-2}$ \\
$2$ & $0.01$ & $1.00\!\times\!10^{-1}$ & $7.236\!\times\!10^{-4}$ & $7.236\!\times\!10^{-3}$ \\
$4$ & $0.10$ & $1.00\!\times\!10^{-1}$ & $5.736\!\times\!10^{-4}$ & $5.736\!\times\!10^{-3}$ \\
$8$ & $0.10$ & $1.00\!\times\!10^{-1}$ & $3.003\!\times\!10^{-6}$ & $3.003\!\times\!10^{-5}$ \\
\bottomrule
\end{tabular}
\caption{Supremum norm and \(L^2\)-mean of the spike residual for
\(h=10^{-1}\), with the support ball contained in the cube. The
\(L^2\)-mean can be made arbitrarily small at fixed supremum norm. The
factor \(r^{d_m/2}\) makes the discrepancy especially pronounced in high
dimension.}
\label{tab:spike}
\end{table}

\paragraph{\bf Adversarial-planner diagnostic.}
The spike construction proves the average-versus-supremum separation, but
one may still ask whether a concentrated model error can materially affect
planning. We illustrate this on a fully specified one-dimensional
receding-horizon problem. This diagnostic isolates the average-versus-uniform planning mechanism in a transparent finite problem; the principal representation theorem supplies the corresponding smooth, finite-sample certification framework.

Let
\[
S=[0,1],
\qquad
A=\{-1,-0.5,0,0.5,1\},
\]
and define
\[
G(s,a)
:=
\operatorname{clip}_{[0,1]}(s+0.05a).
\]
The running cost is
\[
\ell(s,a)
:=
(s-\tau(s))^2,
\qquad
\tau(s)
:=
\begin{cases}
0.7,&s<0.5,\\
0.3,&s\ge0.5,
\end{cases}
\]
and the terminal cost is zero. Thus the preferred target switches across
the decision boundary \(s=0.5\).

At each closed-loop step, an MPC controller enumerates all
\(5^4\) action sequences \(\mathbf a=(a_0,\ldots,a_3)\in A^4\), rolls
them out for horizon \(T=4\) under its planning model \(\widetilde G\),
and minimizes
\[
\sum_{t=0}^{3}\ell(z_t,a_t),
\qquad
z_{t+1}=\widetilde G(z_t,a_t).
\]
It applies the first action of a minimizing sequence; ties are broken by
the first sequence in the lexicographic order induced by
\(-1<-0.5<0<0.5<1\).

We compare the two surrogate models
\[
\widetilde G_A(s,a)
:=
\operatorname{clip}_{[0,1]}\bigl(s+0.05a+r_A(s)\bigr),
\]
\[
r_A(s)
:=
0.15\exp\!\left(-\frac{(s-0.48)^2}{2\sigma^2}\right),
\qquad
\sigma=0.04,
\]
and
\[
\widetilde G_B(s,a)
:=
\operatorname{clip}_{[0,1]}\bigl(s+0.05a+r_B(s)\bigr),
\qquad
r_B(s)\equiv0.04.
\]
For \(i\in\{A,B\}\), define the actual one-step model error
\[
e_i(s,a)
:=
|\widetilde G_i(s,a)-G(s,a)|.
\]
We measure \(e_i\) in the product of normalized Lebesgue measure on
\(S\) and the uniform counting measure on \(A\). Composite trapezoidal
quadrature on \(2{,}000{,}001\) equally spaced state points gives
\[
\|e_A\|_{L^\infty}=0.15,
\qquad
\|e_A\|_{L^2}\approx0.03994,
\]
\[
\|e_B\|_{L^\infty}=0.04,
\qquad
\|e_B\|_{L^2}\approx0.03916.
\]
Thus the two models have nearly equal \(L^2\) one-step errors, while
Model~A has a much larger worst-case error concentrated near the decision
boundary.

For each initial state in the uniform grid of \(11\) points in
\([0.4,0.6]\), we run the receding-horizon controller for \(10\) steps.
At every step, the action selected using the relevant planning model is
applied to the true dynamics \(G\), and the true running cost is
accumulated. Let \(J_{10}^A(s_0)\), \(J_{10}^B(s_0)\), and
\(J_{10}^{\rm exact}(s_0)\) denote the resulting true closed-loop costs
when planning with \(\widetilde G_A\), \(\widetilde G_B\), and the exact
model \(G\), respectively. We report the excess costs
\[
\Delta J_i(s_0)
:=
J_{10}^i(s_0)-J_{10}^{\rm exact}(s_0),
\qquad i\in\{A,B\}.
\]

\begin{table}[ht]
\centering
\small
\setlength{\tabcolsep}{8pt}
\begin{tabular}{lcccc}
\toprule
Model & $\|e\|_{L^\infty}$ & $\|e\|_{L^2}$
& mean excess cost & max. excess cost \\
\midrule
A (spike near boundary)
& $\mathbf{0.150}$ & $0.0399$
& $\mathbf{2.10\!\times\!10^{-1}}$ & $2.82\!\times\!10^{-1}$ \\
B (uniform offset)
& $0.040$ & $0.0392$
& $1.31\!\times\!10^{-1}$ & $2.03\!\times\!10^{-1}$ \\
\bottomrule
\end{tabular}
\caption{Adversarial-planner diagnostic. The two models have nearly equal
\(L^2\) one-step errors, but Model~A has a much larger supremum error and
incurs approximately \(60\%\) more mean excess closed-loop cost relative
to exact-model MPC. The observed ordering displays the additional information supplied by uniform-error control beyond the \(L^2\)-error ordering.  The theorem places this mechanism in a complete quantitative envelope that also incorporates transition and cost Lipschitz constants and horizon factors.}
\label{tab:adversarial}
\end{table}

For a planning state \(z\), the MPC optimizer selects a sequence
\[
\mathbf a^\star(z)
\in
\operatorname*{argmin}_{\mathbf a\in A^4}
\sum_{t=0}^{3}\ell(z_t,a_t),
\qquad
z_{t+1}=\widetilde G(z_t,a_t),
\]
and applies its first component. Optimization can therefore preferentially
select trajectories passing through regions in which the surrogate is
optimistic, even if those regions have small training-measure mass. A uniform semiconjugacy modulus supplies the natural guarantee over all initial states and action sequences, strengthening the average residual into exactly the control-relevant form required by the simulation argument.

\subsection{Why metric non-collapse is the additional ingredient}
\label{sec:why-c-star}

The discussion above explains why the \emph{sup-norm} matters.
The lower co-Lipschitz constant $c_\ast$ of the principal representation theorem in the main article plays a separate but complementary role: even when the semiconjugacy modulus $\eta$ vanishes, a folded encoder produces nonsensical plans. We explain this in two ways.

\paragraph{\bf Cost and optimizer transfer.}
The lower co-Lipschitz estimate makes the original costs well-defined and Lipschitz on the encoded state set, with a quantitative loss proportional to $c_\ast^{-1}$; the latent-cost extension lemma in the main article extends these costs to all of $K_Z$. The deterministic planning-transfer corollary in the main article then compares the true and latent objectives uniformly over the common action space $A^T$. It applies to arbitrary optimizer selections, including nonunique optima, and gives
\[
J_{\rm true}(s_0;\widehat{\mathbf a}(s_0))
-J^\star(s_0)
\le
2\Delta_{T,L_F}(\eta)+\xi_{\rm plan}
\]
for every $\xi_{\rm plan}$-optimal latent action sequence. Thus $c_\ast$ controls the regularity of the cost representation, while $\eta$ controls the dynamical error accumulated over the planning horizon.

\paragraph{\bf Practical Lyapunov transfer.}
Suppose $z^\ast=\psi(s^\dagger)$ for some target state $s^\dagger\in S$, and let $V_F:Z\to\mathbb R_+$ be $L_{V_F}$-Lipschitz on the relevant compact set. Assume
\[
\beta_1\|z-z^\ast\|_Z^2\le V_F(z),
\qquad
V_F(F^\star(z,\pi_F(z)))\le \beta_2V_F(z),
\]
with $\beta_1>0$ and $0<\beta_2<1$. Then the lower co-Lipschitz estimate and the one-step semiconjugacy error imply
\[
\beta_1c_\ast^2d_S(s,s^\dagger)^2
\le
V_F(\psi(s))
\]
and
\[
V_F(\psi(G(s,\pi_F(\psi(s)))))
\le
\beta_2V_F(\psi(s))+L_{V_F}\eta_{W,n,\delta}(\varepsilon_{\rm train}).
\]
Thus $V_F\circ\psi$ is a practical Lyapunov function with an additive disturbance proportional to the semiconjugacy modulus. In particular, iteration yields
\[
V_F(\psi(s_t))
\le
\beta_2^tV_F(\psi(s_0))
+
\frac{L_{V_F}\eta_{W,n,\delta}(\varepsilon_{\rm train})}{1-\beta_2}.
\]
The role of $c_\ast$ is to convert latent closeness into state-space closeness; without a positive lower metric constant, small latent Lyapunov values need not localize the true state.

\paragraph{\bf Reachability in latent space.}
Let $\mathrm{Reach}_t(S_0)\subset S$ be the set of states reachable in $t$ steps from $S_0$, and let $\mathrm{Reach}_t^F(\psi(S_0))\subset Z$ be the corresponding set generated by the learned latent transition under the same admissible action sequences. The deterministic simulation lemma in the main article gives the valid latent-space Hausdorff estimate
\[
d_H\bigl(\psi(\mathrm{Reach}_t(S_0)),\,\mathrm{Reach}_t^F(\psi(S_0))\bigr)
\le
\eta C(t,L_F).
\]
The theorem already yields the displayed latent-space Hausdorff estimate.  When a decoder or a quantified projection onto $\psi(S)$ is available, the same argument extends this certificate to state-space reachable sets.

\subsection{Comparison with stochastic bisimulation bounds.}
\label{sec:vs-bisimulation}
The stochastic bisimulation-metric program of Ferns--Panangaden--Precup \cite{ferns2004metrics,ferns2011bisimulation}, DeepMDP \cite{gelada2019deepmdp}, and DBC \cite{zhang2021learning} provides expected value-error or representation-error bounds under specified distributions, which need not coincide with the visitation distribution induced by a downstream planner.
For a $\Delta$-accurate stochastic transition model, Asadi--Misra--Littman \cite{asadi2018lipschitz} prove the multi-step Wasserstein estimate
\[
W\!\left(\widehat T^{t}(\cdot\mid\mu),T^t(\cdot\mid\mu)\right)
\le
\Delta\sum_{j=0}^{t-1}\overline K^{j},
\qquad
\overline K:=\min\{K_{\widehat T},K_T\}.
\]
This is a transition-distribution bound; their separate discounted value-error result additionally uses Lipschitz reward assumptions.  The geometric factor has the same recurrence form as $C(t,L_F)$ in the deterministic simulation lemma, but the driving discrepancy is Wasserstein-1 rather than a pointwise supremum. In the deterministic setting considered here, the transition kernel reduces to a point map and the uniform residual gives a worst-case guarantee over initial states and action sequences. The results are therefore complementary rather than ordered by a universal notion of sharpness.

\subsection{Summary}
The constant $c_\ast$ and modulus $\eta_{W,n,\delta}(\varepsilon_{\rm train})$ of the principal representation theorem are precisely the quantities that yield
(i)~a uniform finite-horizon cost-transfer estimate and a deterministic latent-optimizer suboptimality bound for compatible Lipschitz costs,
(ii)~the resolution of state recovery from latent codes (\S\ref{sec:why-c-star}),
(iii)~Lyapunov and practical-stability estimates after transfer, and
(iv)~latent-space Hausdorff approximations of forward-reachable sets; a state-space pullback would require an additional decoder or projection assumption.
Under the stated assumptions, the local--global certificate strengthens average prediction and global-spread information into these uniform control conclusions.
The constants $c_\ast$ and $\eta_{W,n,\delta}(\varepsilon_{\rm train})$ are explicit, and the planning guarantee keeps the empirical-training tolerance $\varepsilon_{\rm train}$ separate from the latent-solver tolerance $\xi_{\rm plan}$. The finite-sample planning corollary in the main article further quantifies learned-cost errors through $\varepsilon_\ell$ and $\varepsilon_g$. The guarantee holds whenever the regularization strength belongs to the admissible range of the principal representation theorem in the main article.

\IfFileExists{numbers_macros.tex}{
\newcommand{\AthreeSeeds}{12}
\newcommand{\AthreeSteps}{2000}
\newcommand{\AthreeRZ}{2.0}
\newcommand{\AthreeGrid}{4001}
\newcommand{\AEvenPureNoninj}{12}
\newcommand{\AEvenPureCmin}{0}
\newcommand{\AEvenPureCmax}{0.12}
\newcommand{\AEvenPureKgeom}{\ensuremath{1.20\!\times\!10^{11}}}
\newcommand{\AEvenPureMloc}{\ensuremath{1.20\!\times\!10^{-5}}}
\newcommand{\AEvenPureEta}{\ensuremath{5.34\!\times\!10^{-4}}}
\newcommand{\AEvenCovNoninj}{10}
\newcommand{\AEvenCovCmin}{\ensuremath{3.75\!\times\!10^{-7}}}
\newcommand{\AEvenCovCmax}{1.57}
\newcommand{\AEvenCovKgeom}{\ensuremath{4.37\!\times\!10^{6}}}
\newcommand{\AEvenCovMloc}{\ensuremath{2.32\!\times\!10^{-4}}}
\newcommand{\AEvenCovEta}{\ensuremath{6.24\!\times\!10^{-3}}}
\newcommand{\AEvenLocalNoninj}{9}
\newcommand{\AEvenLocalCmin}{\ensuremath{1.50\!\times\!10^{-7}}}
\newcommand{\AEvenLocalCmax}{3.82}
\newcommand{\AEvenLocalKgeom}{\ensuremath{3.49\!\times\!10^{7}}}
\newcommand{\AEvenLocalMloc}{0.02}
\newcommand{\AEvenLocalEta}{\ensuremath{4.84\!\times\!10^{-3}}}
\newcommand{\AEvenGlobalNoninj}{1}
\newcommand{\AEvenGlobalCmin}{0.82}
\newcommand{\AEvenGlobalCmax}{1.33}
\newcommand{\AEvenGlobalKgeom}{1.59}
\newcommand{\AEvenGlobalMloc}{0.82}
\newcommand{\AEvenGlobalEta}{\ensuremath{7.83\!\times\!10^{-3}}}
\newcommand{\AEvenHingeNoninj}{9}
\newcommand{\AEvenHingeCmin}{\ensuremath{3.36\!\times\!10^{-6}}}
\newcommand{\AEvenHingeCmax}{16.45}
\newcommand{\AEvenHingeKgeom}{\ensuremath{4.55\!\times\!10^{6}}}
\newcommand{\AEvenHingeMloc}{0.01}
\newcommand{\AEvenHingeEta}{0.01}
\newcommand{\AEvenHingeupNoninj}{3}
\newcommand{\AEvenHingeupCmin}{0.70}
\newcommand{\AEvenHingeupCmax}{1.36}
\newcommand{\AEvenHingeupKgeom}{1.97}
\newcommand{\AEvenHingeupMloc}{0.70}
\newcommand{\AEvenHingeupEta}{\ensuremath{8.06\!\times\!10^{-3}}}
\newcommand{\AEvenTwosidedNoninj}{3}
\newcommand{\AEvenTwosidedCmin}{0.46}
\newcommand{\AEvenTwosidedCmax}{0.55}
\newcommand{\AEvenTwosidedKgeom}{1.19}
\newcommand{\AEvenTwosidedMloc}{0.46}
\newcommand{\AEvenTwosidedEta}{\ensuremath{4.64\!\times\!10^{-3}}}
\newcommand{\APlainPureNoninj}{10}
\newcommand{\APlainPureCmin}{0}
\newcommand{\APlainPureCmax}{0.16}
\newcommand{\APlainPureKgeom}{\ensuremath{9.04\!\times\!10^{10}}}
\newcommand{\APlainPureMloc}{\ensuremath{1.62\!\times\!10^{-5}}}
\newcommand{\APlainPureEta}{\ensuremath{5.72\!\times\!10^{-4}}}
\newcommand{\APlainCovNoninj}{1}
\newcommand{\APlainCovCmin}{0.25}
\newcommand{\APlainCovCmax}{1.17}
\newcommand{\APlainCovKgeom}{5.19}
\newcommand{\APlainCovMloc}{0.25}
\newcommand{\APlainCovEta}{\ensuremath{9.25\!\times\!10^{-3}}}
\newcommand{\APlainLocalNoninj}{1}
\newcommand{\APlainLocalCmin}{0.50}
\newcommand{\APlainLocalCmax}{0.54}
\newcommand{\APlainLocalKgeom}{1.10}
\newcommand{\APlainLocalMloc}{0.50}
\newcommand{\APlainLocalEta}{\ensuremath{4.62\!\times\!10^{-3}}}
\newcommand{\APlainGlobalNoninj}{0}
\newcommand{\APlainGlobalCmin}{0.83}
\newcommand{\APlainGlobalCmax}{1.29}
\newcommand{\APlainGlobalKgeom}{1.55}
\newcommand{\APlainGlobalMloc}{0.83}
\newcommand{\APlainGlobalEta}{\ensuremath{9.57\!\times\!10^{-3}}}
\newcommand{\APlainHingeNoninj}{0}
\newcommand{\APlainHingeCmin}{0.90}
\newcommand{\APlainHingeCmax}{1.31}
\newcommand{\APlainHingeKgeom}{1.48}
\newcommand{\APlainHingeMloc}{0.90}
\newcommand{\APlainHingeEta}{\ensuremath{9.12\!\times\!10^{-3}}}
\newcommand{\APlainHingeupNoninj}{0}
\newcommand{\APlainHingeupCmin}{0.90}
\newcommand{\APlainHingeupCmax}{1.30}
\newcommand{\APlainHingeupKgeom}{1.48}
\newcommand{\APlainHingeupMloc}{0.90}
\newcommand{\APlainHingeupEta}{\ensuremath{9.12\!\times\!10^{-3}}}
\newcommand{\APlainTwosidedNoninj}{0}
\newcommand{\APlainTwosidedCmin}{0.49}
\newcommand{\APlainTwosidedCmax}{0.53}
\newcommand{\APlainTwosidedKgeom}{1.12}
\newcommand{\APlainTwosidedMloc}{0.49}
\newcommand{\APlainTwosidedEta}{\ensuremath{4.21\!\times\!10^{-3}}}
\newcommand{\AEvenHingeOkN}{3}
\newcommand{\AEvenHingeOkCmin}{0.86}
\newcommand{\AEvenHingeOkCmax}{1.27}
\newcommand{\AEvenHingeOkKgeom}{1.47}
\newcommand{\AEvenHingeBadN}{9}
\newcommand{\AEvenHingeBadCmin}{\ensuremath{2.59\!\times\!10^{-6}}}
\newcommand{\AEvenHingeBadCmax}{20.78}
\newcommand{\AEvenHingeBadKgeom}{\ensuremath{7.87\!\times\!10^{6}}}
\newcommand{\AEvenHingeupOkN}{9}
\newcommand{\AEvenHingeupOkCmin}{0.76}
\newcommand{\AEvenHingeupOkCmax}{1.34}
\newcommand{\AEvenHingeupOkKgeom}{1.78}
\newcommand{\AEvenHingeupBadN}{3}
\newcommand{\AEvenHingeupBadCmin}{\ensuremath{2.23\!\times\!10^{-6}}}
\newcommand{\AEvenHingeupBadCmax}{2.55}
\newcommand{\AEvenHingeupBadKgeom}{\ensuremath{1.15\!\times\!10^{6}}}
\newcommand{\BClearFirstN}{2048}
\newcommand{\BClearLams}{0.3}
\newcommand{\BClearNs}{2048}
\newcommand{\CpureRatioMin}{4.03}
\newcommand{\CpureRatioMax}{25.79}
\newcommand{\BPstar}{\ensuremath{4.81\!\times\!10^{-6}}}
\newcommand{\BEpsMin}{5.19}
\newcommand{\BEpsMax}{18.46}
\newcommand{\BBeta}{\ensuremath{7.01\!\times\!10^{-4}}}
\newcommand{\BRatioMin}{\ensuremath{1.08\!\times\!10^{6}}}
\newcommand{\BBestLam}{0.0}
\newcommand{\BBestEta}{\ensuremath{5.22\!\times\!10^{-8}}}
\newcommand{\BBestCmin}{0}
\newcommand{\BBestCleared}{0}
\newcommand{\BBestSeeds}{2}
\newcommand{\BZeroCmin}{0}
\newcommand{\BZeroEta}{\ensuremath{5.22\!\times\!10^{-8}}}
\newcommand{\BZeroNOneTwoEightCmin}{0}
\newcommand{\BZeroNOneTwoEightEta}{\ensuremath{4.31\!\times\!10^{-8}}}
\newcommand{\BZeroNOneTwoEightNloc}{0.16}
\newcommand{\BZeroNOneTwoEightCleared}{0}
\newcommand{\BZeroNOneTwoEightSeeds}{2}
\newcommand{\BLoNOneTwoEightCmin}{0.67}
\newcommand{\BLoNOneTwoEightEta}{0.11}
\newcommand{\BLoNOneTwoEightNloc}{0}
\newcommand{\BLoNOneTwoEightCleared}{0}
\newcommand{\BLoNOneTwoEightSeeds}{2}
\newcommand{\BMidNOneTwoEightCmin}{0.36}
\newcommand{\BMidNOneTwoEightEta}{0.37}
\newcommand{\BMidNOneTwoEightNloc}{\ensuremath{4.77\!\times\!10^{-3}}}
\newcommand{\BMidNOneTwoEightCleared}{0}
\newcommand{\BMidNOneTwoEightSeeds}{2}
\newcommand{\BZeroNFiveOneTwoCmin}{0}
\newcommand{\BZeroNFiveOneTwoEta}{\ensuremath{4.58\!\times\!10^{-8}}}
\newcommand{\BZeroNFiveOneTwoNloc}{0.16}
\newcommand{\BZeroNFiveOneTwoCleared}{0}
\newcommand{\BZeroNFiveOneTwoSeeds}{2}
\newcommand{\BLoNFiveOneTwoCmin}{0.71}
\newcommand{\BLoNFiveOneTwoEta}{0.01}
\newcommand{\BLoNFiveOneTwoNloc}{0}
\newcommand{\BLoNFiveOneTwoCleared}{0}
\newcommand{\BLoNFiveOneTwoSeeds}{2}
\newcommand{\BMidNFiveOneTwoCmin}{0.38}
\newcommand{\BMidNFiveOneTwoEta}{0.25}
\newcommand{\BMidNFiveOneTwoNloc}{\ensuremath{5.15\!\times\!10^{-3}}}
\newcommand{\BMidNFiveOneTwoCleared}{0}
\newcommand{\BMidNFiveOneTwoSeeds}{2}
\newcommand{\BZeroNTwoZeroFourEightCmin}{0}
\newcommand{\BZeroNTwoZeroFourEightEta}{\ensuremath{5.22\!\times\!10^{-8}}}
\newcommand{\BZeroNTwoZeroFourEightNloc}{0.16}
\newcommand{\BZeroNTwoZeroFourEightCleared}{0}
\newcommand{\BZeroNTwoZeroFourEightSeeds}{2}
\newcommand{\BLoNTwoZeroFourEightCmin}{0.68}
\newcommand{\BLoNTwoZeroFourEightEta}{0.02}
\newcommand{\BLoNTwoZeroFourEightNloc}{0}
\newcommand{\BLoNTwoZeroFourEightCleared}{0}
\newcommand{\BLoNTwoZeroFourEightSeeds}{2}
\newcommand{\BMidNTwoZeroFourEightCmin}{0.88}
\newcommand{\BMidNTwoZeroFourEightEta}{0.03}
\newcommand{\BMidNTwoZeroFourEightNloc}{0}
\newcommand{\BMidNTwoZeroFourEightCleared}{1}
\newcommand{\BMidNTwoZeroFourEightSeeds}{2}
\newcommand{\BSeeds}{2}
\newcommand{\CrefFailTFive}{0.69}
\newcommand{\CrefSpreadTFive}{0.07}
\newcommand{\CoptimizerSpreadTFive}{\ensuremath{9.98\!\times\!10^{-3}}}
\newcommand{\CrefFailTOneZero}{0.66}
\newcommand{\CrefSpreadTOneZero}{0.11}
\newcommand{\CoptimizerSpreadTOneZero}{0.05}
\newcommand{\CrefFailTTwoZero}{0.72}
\newcommand{\CrefSpreadTTwoZero}{0.28}
\newcommand{\CoptimizerSpreadTTwoZero}{0.28}
\newcommand{\CrefRestarts}{6}
\newcommand{\CrhoPooledRatio}{0.62}
\newcommand{\CrhoPooledBT}{-0.06}
\newcommand{\CrhoPooledEta}{-0.34}
\newcommand{\CrhoPooledInvC}{0.55}
\newcommand{\CrhoPooledRatioLo}{0.30}
\newcommand{\CrhoPooledRatioHi}{0.80}
\newcommand{\CrhoTFiveRatio}{0.51}
\newcommand{\CrhoTFiveBT}{-0.23}
\newcommand{\CrhoTFiveEta}{-0.35}
\newcommand{\CrhoTFiveInvC}{0.44}
\newcommand{\CrhoTTwoZeroRatio}{0.69}
\newcommand{\CrhoTTwoZeroBT}{0.15}
\newcommand{\CrhoTTwoZeroEta}{-0.45}
\newcommand{\CrhoTTwoZeroInvC}{0.65}
\newcommand{\CrhoNoPurePooledRatio}{0.42}
\newcommand{\CrhoNoPurePooledBT}{0.06}
\newcommand{\CrhoNoPurePooledEta}{0.07}
\newcommand{\CrhoNoPurePooledInvC}{0.30}
\newcommand{\CrhoNoPurePooledRatioLo}{0.04}
\newcommand{\CrhoNoPurePooledRatioHi}{0.69}
\newcommand{\CrhoNoPureTFiveRatio}{0.20}
\newcommand{\CrhoNoPureTFiveBT}{-0.21}
\newcommand{\CrhoNoPureTFiveEta}{0.09}
\newcommand{\CrhoNoPureTFiveInvC}{0.09}
\newcommand{\CrhoNoPureTTwoZeroRatio}{0.54}
\newcommand{\CrhoNoPureTTwoZeroBT}{0.59}
\newcommand{\CrhoNoPureTTwoZeroEta}{-0.11}
\newcommand{\CrhoNoPureTTwoZeroInvC}{0.48}
\newcommand{\CproxyLooseMed}{81.95}
\newcommand{\CproxyLooseMax}{\ensuremath{1.18\!\times\!10^{5}}}
\newcommand{\CchordMin}{0.21}
\newcommand{\CchordMax}{1.18}
\newcommand{\CsigMin}{0.20}
\newcommand{\CsigMax}{1.18}
\newcommand{\CSeeds}{[0, 1, 2]}
\newcommand{\CInits}{32}
\newcommand{\CPop}{48}
\newcommand{\CIters}{4}
\newcommand{\CClSteps}{12}
\newcommand{\CpureEta}{0.01}
\newcommand{\CpureCmin}{\ensuremath{1.75\!\times\!10^{-4}}}
\newcommand{\CpureKgeom}{177.21}
\newcommand{\CpureEpsL}{1.82}
\newcommand{\CcovEta}{0.13}
\newcommand{\CcovCmin}{0.19}
\newcommand{\CcovKgeom}{5.69}
\newcommand{\CcovEpsL}{0.29}
\newcommand{\ClocalEta}{0.08}
\newcommand{\ClocalCmin}{0.12}
\newcommand{\ClocalKgeom}{5.62}
\newcommand{\ClocalEpsL}{0.28}
\newcommand{\CglobalEta}{0.04}
\newcommand{\CglobalCmin}{0.06}
\newcommand{\CglobalKgeom}{4.15}
\newcommand{\CglobalEpsL}{0.71}
\newcommand{\ChingeEta}{0.09}
\newcommand{\ChingeCmin}{0.04}
\newcommand{\ChingeKgeom}{19.72}
\newcommand{\ChingeEpsL}{0.29}
\newcommand{\CtwosidedEta}{0.06}
\newcommand{\CtwosidedCmin}{0.14}
\newcommand{\CtwosidedKgeom}{3.92}
\newcommand{\CtwosidedEpsL}{0.34}
\newcommand{\CpuredJmedTFive}{1.41}
\newcommand{\CpuredJmaxTFive}{5.81}
\newcommand{\CpureFracNegTFive}{0}
\newcommand{\CpuredJoracleTFive}{0.68}
\newcommand{\CpuredJmedTOneZero}{1.54}
\newcommand{\CpuredJmaxTOneZero}{5.90}
\newcommand{\CpureFracNegTOneZero}{0}
\newcommand{\CpuredJoracleTOneZero}{0.63}
\newcommand{\CpuredJmedTTwoZero}{1.55}
\newcommand{\CpuredJmaxTTwoZero}{5.22}
\newcommand{\CpureFracNegTTwoZero}{0}
\newcommand{\CpuredJoracleTTwoZero}{0.50}
\newcommand{\CcovdJmedTFive}{0.10}
\newcommand{\CcovdJmaxTFive}{3.48}
\newcommand{\CcovFracNegTFive}{0.03}
\newcommand{\CcovdJoracleTFive}{0.06}
\newcommand{\CcovdJmedTOneZero}{0.10}
\newcommand{\CcovdJmaxTOneZero}{1.33}
\newcommand{\CcovFracNegTOneZero}{0.12}
\newcommand{\CcovdJoracleTOneZero}{0.07}
\newcommand{\CcovdJmedTTwoZero}{0.23}
\newcommand{\CcovdJmaxTTwoZero}{1.89}
\newcommand{\CcovFracNegTTwoZero}{0.06}
\newcommand{\CcovdJoracleTTwoZero}{0.10}
\newcommand{\ClocaldJmedTFive}{0.08}
\newcommand{\ClocaldJmaxTFive}{2.37}
\newcommand{\ClocalFracNegTFive}{0.22}
\newcommand{\ClocaldJoracleTFive}{0.06}
\newcommand{\ClocaldJmedTOneZero}{0.07}
\newcommand{\ClocaldJmaxTOneZero}{2.63}
\newcommand{\ClocalFracNegTOneZero}{0.16}
\newcommand{\ClocaldJoracleTOneZero}{0.06}
\newcommand{\ClocaldJmedTTwoZero}{0.38}
\newcommand{\ClocaldJmaxTTwoZero}{3.74}
\newcommand{\ClocalFracNegTTwoZero}{0.19}
\newcommand{\ClocaldJoracleTTwoZero}{0.13}
\newcommand{\CglobaldJmedTFive}{0.08}
\newcommand{\CglobaldJmaxTFive}{3.86}
\newcommand{\CglobalFracNegTFive}{0.09}
\newcommand{\CglobaldJoracleTFive}{0.05}
\newcommand{\CglobaldJmedTOneZero}{0.16}
\newcommand{\CglobaldJmaxTOneZero}{2.46}
\newcommand{\CglobalFracNegTOneZero}{0.09}
\newcommand{\CglobaldJoracleTOneZero}{0.07}
\newcommand{\CglobaldJmedTTwoZero}{0.20}
\newcommand{\CglobaldJmaxTTwoZero}{2.29}
\newcommand{\CglobalFracNegTTwoZero}{0.25}
\newcommand{\CglobaldJoracleTTwoZero}{0.10}
\newcommand{\ChingedJmedTFive}{0.20}
\newcommand{\ChingedJmaxTFive}{3.01}
\newcommand{\ChingeFracNegTFive}{0.06}
\newcommand{\ChingedJoracleTFive}{0.05}
\newcommand{\ChingedJmedTOneZero}{0.18}
\newcommand{\ChingedJmaxTOneZero}{3.00}
\newcommand{\ChingeFracNegTOneZero}{0.16}
\newcommand{\ChingedJoracleTOneZero}{0.06}
\newcommand{\ChingedJmedTTwoZero}{0.32}
\newcommand{\ChingedJmaxTTwoZero}{2.94}
\newcommand{\ChingeFracNegTTwoZero}{0.09}
\newcommand{\ChingedJoracleTTwoZero}{0.11}
\newcommand{\CtwosideddJmedTFive}{0.09}
\newcommand{\CtwosideddJmaxTFive}{1.49}
\newcommand{\CtwosidedFracNegTFive}{0.16}
\newcommand{\CtwosideddJoracleTFive}{0.06}
\newcommand{\CtwosideddJmedTOneZero}{0.06}
\newcommand{\CtwosideddJmaxTOneZero}{1.68}
\newcommand{\CtwosidedFracNegTOneZero}{0.16}
\newcommand{\CtwosideddJoracleTOneZero}{0.06}
\newcommand{\CtwosideddJmedTTwoZero}{0.18}
\newcommand{\CtwosideddJmaxTTwoZero}{1.32}
\newcommand{\CtwosidedFracNegTTwoZero}{0.19}
\newcommand{\CtwosideddJoracleTTwoZero}{0.11}
}{}
\providecommand{\AthreeSeeds}{??}
\providecommand{\AthreeSteps}{??}
\providecommand{\AthreeRZ}{??}
\providecommand{\AthreeGrid}{??}
\providecommand{\CsigMin}{??}
\providecommand{\CsigMax}{??}
\providecommand{\CchordMin}{??}
\providecommand{\CchordMax}{??}
\section{Numerical protocols, summarized records, and supporting-material availability}\label{sec:numerics}\label{sec:experiments}\label{supp:experiments}

The experiments below provide three complementary layers of evidence: exact geometric obstruction results for prediction and covariance-spread objectives, sampled geometric diagnostics in finite trained models, and a nonlinear benchmark linking restored metric resolution to improved control.  The analytical layer establishes the obstruction exactly; the finite-dimensional layers quantify the mechanism, threshold ingredients, optimization behavior, and planning consequences on declared fixed evaluation sets.  Geometry-aware objectives consistently restore nontrivial metric resolution and strongly improve control relative to pure prediction, while the local--global hinge retains the distinctive advantage of matching the sufficient mechanism proved in the main theorem.
In our reports, the quantities in the principal representation theorem in the main article are computed using the learned representation and learned latent cost heads. For example, if \((\widehat{\Phi},\widehat{F})\) is a learned representation, we compute
\[
	\widehat{\eta}_{\rm eval}:= \max_{(s_i,a_i)\in D_M^{\rm eval}} \|\widehat{\Phi}(H(G(s_i,a_i)))-\widehat{F}(\widehat{\Phi}(H(s_i)),a_i)\|,
\]
Furthermore, we denote the pairs consisting of minimum and maximum of the pairwise chord ratios and the norm of \(D(\widehat{\Phi}\circ H)\) by \((\widehat{c}_{\min},\widehat{c}_{\max})\), \((\widehat{m}_{\rm loc},\widehat{M}_{\rm loc})\), respectively. We also denote by \(\widehat{m}_{\rm glob}\) the minimum distance of \(\widehat{\Phi}\circ H(s)\) and \(\widehat{\Phi}\circ H(s')\) for all samples \(s,s'\) satisfying \(|s-s'|\ge\rho\).

The analytic obstructions in Experiment A are exact, while the trained studies provide complementary finite-architecture and sampled-diagnostic evidence.
In particular, the coefficient-box spline class in Experiment B is deliberately chosen as a tractable finite-capacity proxy for restart analysis and threshold decomposition, whereas the uniform-capacity $C^{1,1}$ class constructed in Proposition~\ref{prop:spline-A4} supplies the theorem-level approximation architecture.  This division combines a rigorous constructive class with a computationally transparent calibration study.
Three design decisions distinguish the experiments below from a naive implementation.

\begin{enumerate}
\item \textbf{A common compact latent range and explicit upper control.}
The proposed hinge is one-sided. In an unconstrained network, multiplying the encoder output by a large constant improves every lower diagnostic without improving the representation at all. The theorem excludes this by assuming a compact latent range $K_Z$ and a uniform $C^{1,1}$ budget. Accordingly every trained encoder in every experiment has the form
\begin{equation}\label{eq:bounded-enc}
  \psi_\theta(s) \;=\; R_Z\,\tanh\!\big(h_\theta(H(s))\big)
  \qquad\text{coordinatewise,}
\end{equation}
with the \emph{same} fixed $R_Z$ for every objective and seed, and we report the upper diagnostics alongside the lower ones.
Resolution improvement is evaluated jointly with $\widehat c_{\max}$, $\widehat M_{\rm loc}$, and the latent norm, so the reported finite-set gain documents improved lower geometry rather than scale inflation.  Continuum certification is reserved for the validated finite-net conditions of the main article.

\item \textbf{A fixed empirical objective.} The principal representation theorem in the main article concerns three fixed independent samples $D_M$, $D_U$, $D_P$. In the theorem-facing experiments these arrays are drawn once per seed and frozen; minibatches are index draws into them, and no direction or separated pair is generated during optimization. A separate, clearly labeled ablation compares this with resampled training.

\item \textbf{Restart-aware optimization accounting.} The theorem keeps the optimization tolerance $\varepsilon_{\rm train}$ explicit, so every run reports final objective components, sampled total-objective curves, and terminal gradient norms.  In Experiment B, the observed restart gap is computed within the same fixed empirical objective, with seed, capacity, sample size, and regularization strength held fixed.  Experiment A reports cross-seed objective distributions separately.  This accounting makes the achieved optimization quality visible and connects the numerical study directly to the theorem's tolerance parameter.
\end{enumerate}

\paragraph{\bf Evidence hierarchy.} We state in advance how each analytical and numerical layer contributes to the paper's conclusions.
\begin{enumerate}
\item[(H1)] \emph{Analytically established.} Zero prediction residual together
with positive, indeed full-rank, latent covariance does not imply injectivity
(Experiment~A1--A2).
\item[(H2)] \emph{Empirically established in the designed stress test.}
Under matched bounded training, the objectives produce distinct folding and expansion patterns.  The local-only and global-only penalties isolate complementary geometric roles, while upper-controlled and two-sided alternatives reveal the importance of coupling lower resolution with upper-distortion diagnostics (Experiment~A3).
\item[(H3)] \emph{Empirically established in the nonlinear benchmark.}
The pure-prediction arm combines the smallest evaluated residual with collapse and markedly larger control costs, whereas every geometry-aware arm restores nontrivial metric resolution and strongly improves planning.  The comparable performance of several regularizers demonstrates the robustness of the geometric principle, with the local--global hinge distinguished by its theorem-level certificate (Experiment~C).
\item[(H4)] \emph{A modular certification pathway.} Neural optimization tolerance, continuum suprema, learned recovery of the training geometry, and finite-horizon envelopes are organized as separate layers.  Objectives, gradients, evaluation nets, Lipschitz diagnostics, and restart records provide the inputs needed to validate those layers and, after controlled finite-net remainders are supplied, to upgrade a selected trained model to a continuum certificate.
\end{enumerate}

Table~\ref{tab:protocol} records the protocol common to all experiments.
The supplement contains the mathematical protocols and summarized numerical records displayed below.  Complete code, configurations, frozen arrays, per-seed and per-initial-state raw outputs, objective and gradient histories, and supporting verification documents are retained by the authors and will be provided upon request.  They are not asserted to be part of this Supplementary Material.

\begin{table}[t]\centering\footnotesize
\resizebox{\textwidth}{!}{%
\begin{tabular}{lccccccc}\toprule
experiment & $R_Z$ & upper control & $n_M/n_U/n_P$ & runs per objective & architecture & params & preprocessing time (s)\\\midrule
A3 folding & 2.0 & bounded range; upper hinge in objective (vi) & 512/512/256 & 12 seeds $\times$ 2 regimes & MLP $1$-$128$-$128$-$1$ & -- & --\\
B spline (main grid) & 2.0 & coefficient box $\pm3.0$ & $n$/$n$/$n$, $n\le2048$ & 2 seeds $\times$ 3 restarts & cubic B-spline & -- & --\\
B spline (capacity sweep) & 2.0 & coefficient box & $512$/$512$/$512$ & 1 seed(s) & cubic B-spline & -- & --\\
B bounded MLP & 2.0 & approx. PI spectral target 1.6 & $n$/$n$/$n$ & 2 seeds & MLP width 64 & -- & --\\
C pendulum + planning & 2.0 & approx. PI spectral target 1.5 & 2048/1024/512 & 3 seeds & MLP width 64 & 13896 & 20.64 [14.87, 50.40]\\
\bottomrule\end{tabular}
}
\caption{Protocol alignment. Every trained encoder in a given experiment maps into the same compact latent box $K_Z=[-R_Z,R_Z]^{d_z}$, uses the same fixed $D_M,D_U,D_P$ arrays, and receives the same number of updates. Run counts are read from the authors' retained result records, which will be provided upon request. The timing column is the median [min, max] over individual representation preprocessing runs, including representation training, cost-head fitting, and diagnostics but excluding MPC planning; it is not a total experiment wall-clock.}\label{tab:protocol}
\end{table}

\subsection{Experiment A: folding, full-rank covariance, and the
local/global ablation}\label{sec:expA}

\subsubsection*{A1. The exact scalar obstruction}

On $S=[-1,1]$, $A=\{0\}$, $G(s,0)=s$, $H=\mathrm{id}$, the family
(with the singleton action again serving only as shorthand for an action-independent controlled system)
$\psi_C(s)=Cs^2$ with $F(z)=z$ has identically zero JEPA residual and latent variance $4C^2/45$, so for every variance threshold $\gamma$ there is a half-line $C\ge\sqrt{45\gamma/4}$ on which a variance penalty vanishes while $\psi_C(s)=\psi_C(-s)$.
The population local defect is
\[
\mathcal N_{\rm loc}(\psi_C)
=
\begin{cases}
\displaystyle \kappa^2-2\kappa C+\frac{4}{3}C^2,
&0<C\le \kappa/2,\\[2mm]
\displaystyle \frac{\kappa^3}{6C},
&C\ge \kappa/2,
\end{cases}
\]
and is therefore strictly positive for every finite $C>0$.
Table~\ref{tab:a1} reports the Monte-Carlo check together with the full sampling details.
A \emph{single} fixed separated-pair array is shared by every value of $C$, so that differences across $C$ cannot be an artifact of changing the sample.
This is a verification of algebra, not evidence from machine learning.

\subsubsection*{A2. A full-rank covariance obstruction}

A scalar variance example does not by itself answer a spectral penalty, which constrains $\lambda_{\min}(\Sigma_\psi)$. We therefore record the two-dimensional family
\begin{equation}\label{eq:fullrank}
  S=[-1,1]^2,\qquad
  \psi_{C,D}(s_1,s_2)=(C s_1^2,\; D s_2),\qquad F(z)=z ,
\end{equation}
for which, under the uniform law, $\operatorname{Cov}(\psi_{C,D})=\operatorname{diag}(4C^2/45,\;D^2/3)$ is positive definite for all $C,D>0$, while $\psi_{C,D}(s_1,s_2)=\psi_{C,D}(-s_1,s_2)$. Thus a minimum-eigenvalue penalty is cleared with room to spare by a non-injective, zero-residual model. Table~\ref{tab:a2} reports the empirical covariance against the closed form and the exactly vanishing gap on reflected pairs.

\begin{table}[t]\centering\footnotesize
\resizebox{\textwidth}{!}{%
\begin{tabular}{lcccccc}\toprule
$C$ & $\widehat{\rm Var}(\psi_C)$ & $4C^2/45$ & JEPA res. & $P_{\rm var}(0.2)$ & $N_{\rm loc}$ & $N_{\rm glob}$\\\midrule
0.5 & 0.022 & 0.022 & 0 & 0.03 & 0.04 & 0.10\\
1.0 & 0.089 & 0.089 & 0 & 0.01 & 0.02 & 0.05\\
2.0 & 0.356 & 0.356 & 0 & 0 & 0.01 & 0.03\\
5.0 & 2.226 & 2.222 & 0 & 0 & $4.17\!\times\!10^{-3}$ & 0.01\\
10.0 & 8.903 & 8.889 & 0 & 0 & $2.08\!\times\!10^{-3}$ & $5.33\!\times\!10^{-3}$\\
\bottomrule\end{tabular}
}
\caption{Experiment A1: exact scalar obstruction $\psi_C(s)=Cs^2$. One common fixed separated-pair array is used for every $C$, so differences across $C$ cannot come from changing the Monte Carlo sample.}\label{tab:a1}
\end{table}

\begin{table}[t]\centering\footnotesize
\resizebox{\textwidth}{!}{%
\begin{tabular}{lccccc}\toprule
$(C,D)$ & $\lambda_{\min}(\Sigma)$ emp. & closed form & $\lambda_{\max}(\Sigma)$ & JEPA res. & $\sup$ gap on reflected pairs\\\midrule
(1, 1) & 0.089 & 0.089 & 0.332 & 0 & 0\\
(2, 1) & 0.332 & 0.333 & 0.355 & 0 & 0\\
(5, 2) & 1.327 & 1.333 & 2.222 & 0 & 0\\
(10, 3) & 2.985 & 3.000 & 8.886 & 0 & 0\\
\bottomrule\end{tabular}
}
\caption{Experiment A2: full-rank covariance obstruction $\psi_{C,D}(s_1,s_2)=(Cs_1^2, Ds_2)$ with $\mathrm{Cov}=\mathrm{diag}(4C^2/45, D^2/3)\succ0$ and $\psi_{C,D}(s_1,s_2)=\psi_{C,D}(-s_1,s_2)$. A minimum-eigenvalue penalty is cleared while the encoder stays non-injective.}\label{tab:a2}
\end{table}

\subsubsection*{A3. Bounded trained comparison}

We retain the identity dynamics and the latent-only transition of the Experiment~A1 --- $G(s,0)=s$ and $F$ receiving only $z=\psi(s)$, so that the residual is $\psi(s)-F(\psi(s))$ and no physical state enters the transition --- and re-run it under the bounded output \eqref{eq:bounded-enc} and fixed $D_M,D_U,D_P$ arrays. Seven objectives are compared under identical fixed arrays, identical initial encoder and transition weights, and identical minibatch index sequences: pure prediction; covariance spread; local hinge only; separated-pair hinge only; the combined local--global hinge; the combined hinge plus an upper directional term $N_{\rm up}(\psi)=\mathbb{E}\big[(\|D\psi(s)v\|-K)_+^2\big]$; and a two-sided bi-Lipschitz-relaxation baseline that penalizes chord ratios from both sides under the same latent range and parameter budget.

Two initialization regimes are used: the symmetry-biased even pretraining of the original experiment and a second regime with no pretraining.  This separates recovery from a strong symmetry bias from ordinary training behavior.

Evaluation uses a common deterministic grid of \AthreeGrid{} states. Exact grid derivatives are obtained by automatic differentiation and all pairwise chord ratios are evaluated by chunking, with a prescribed minimum pair distance equal to the grid spacing. We report
\[
  \widehat{\operatorname{Var}},\quad \lambda_{\min}(\widehat\Sigma),\quad
  \widehat c_{\min},\quad \widehat c_{\max},\quad
  \widehat\kappa_{\rm geom}=\frac{\widehat c_{\max}}{\widehat c_{\min}},\quad
  \widehat m_{\rm loc},\quad \widehat M_{\rm loc},\quad
  \widehat m_{\rm glob},\quad \widehat\eta_{\rm eval},
\]
and classify injectivity both by monotonicity on the grid and by a robust derivative-sign reversal test. We report outcome counts alongside medians, so optimization robustness is visible directly.

Figure~\ref{fig:folding} shows the learned profiles and the $(\widehat c_{\min},\widehat c_{\max})$ scatter. Tables~\ref{tab:geometryA} and~\ref{tab:geometryAplain} report the per-method medians and ranges, while Table~\ref{tab:geometryAoutcome} conditions the folded-basin results on optimization outcome.

\begin{figure}[t]\centering
\includegraphics[width=\textwidth]{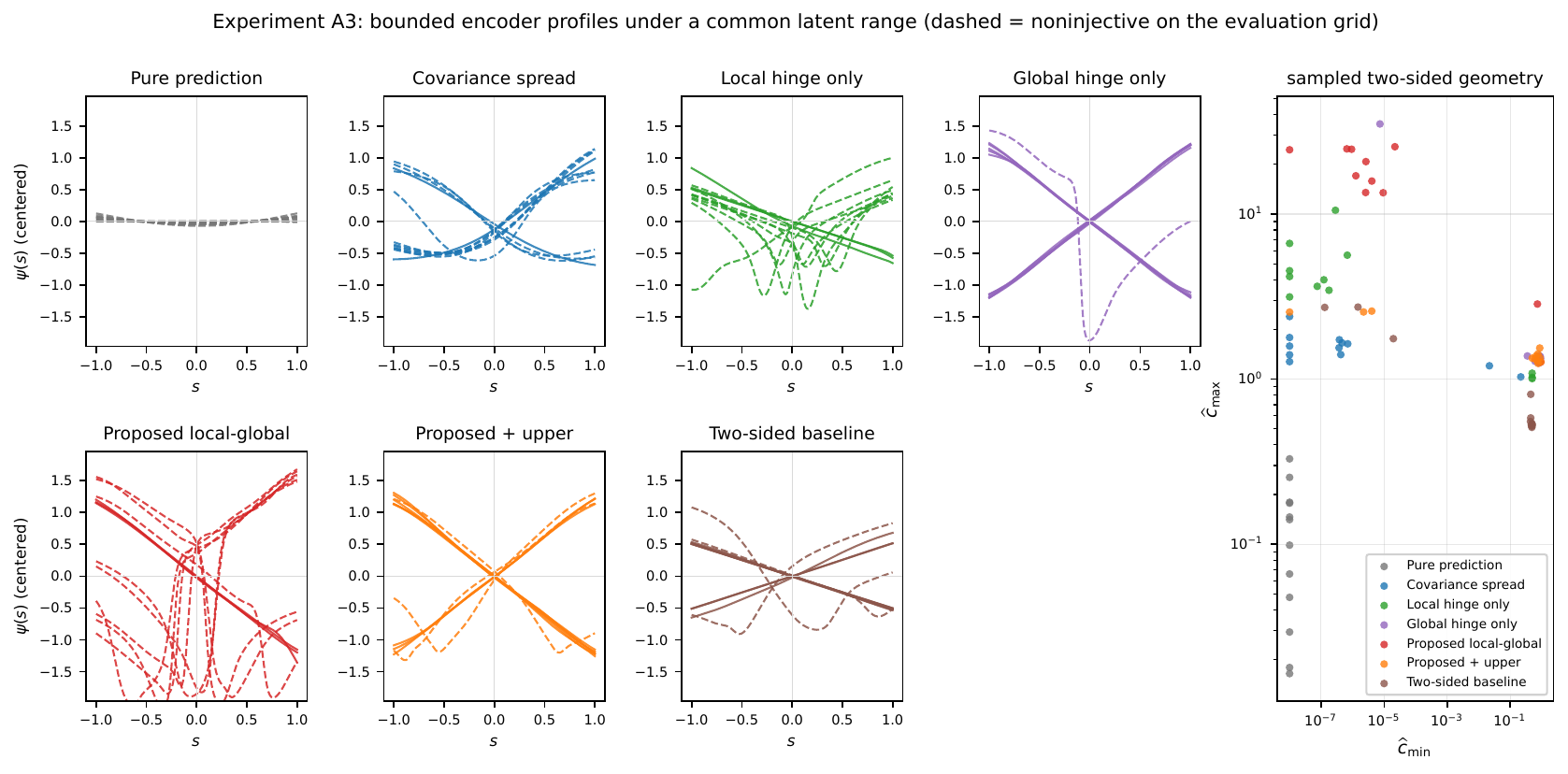}
\caption{Experiment A3. Learned encoder profiles $s\mapsto\psi(s)$ under a common compact latent range, for the seven objectives and \AthreeSeeds{} paired seeds in the folded-basin initialization regime. Dashed curves are noninjective on the evaluation grid. The right panel plots the sampled lower against the sampled upper chord ratio: points far to the right are metrically faithful, points high up are metrically inflated, and only the lower-right region represents a genuine improvement in resolution.}
\label{fig:folding}
\end{figure}

\begin{table}[p]\centering\small
\resizebox{0.98\textwidth}{!}{%
\begin{tabular}{llccc}\toprule status & method & $\widehat c_{\min}$ & $\widehat c_{\max}$ & $\widehat\kappa_{\rm geom}$\\
\midrule
12/12 fail & Pure prediction & 0 [0, 0] & 0.12 [0.02, 0.33] & $1.2\!\times\!10^{11}$ [$1.6\!\times\!10^{10}$, $3.3\!\times\!10^{11}$]\\
10/12 fail & Covariance spread & $3.75\!\times\!10^{-7}$ [0, 0.22] & 1.57 [1.03, 2.40] & $4.4\!\times\!10^{6}$ [4.7, $2.4\!\times\!10^{12}$]\\
9/12 fail & Local hinge only & $1.50\!\times\!10^{-7}$ [0, 0.50] & 3.82 [1.01, 10.54] & $3.5\!\times\!10^{7}$ [2.0, $6.6\!\times\!10^{12}$]\\
1/12 fail & Global hinge only & 0.82 [$7.37\!\times\!10^{-6}$, 0.92] & 1.33 [1.25, 35.14] & 1.6 [1.4, $4.8\!\times\!10^{6}$]\\
9/12 fail & Proposed local-global & $3.36\!\times\!10^{-6}$ [0, 0.94] & 16.45 [1.27, 25.58] & $4.5\!\times\!10^{6}$ [1.3, $2.5\!\times\!10^{13}$]\\
3/12 fail & Proposed + upper & 0.70 [0, 0.94] & 1.36 [1.25, 2.58] & 2.0 [1.3, $2.5\!\times\!10^{12}$]\\
3/12 fail & Two-sided baseline & 0.46 [$1.30\!\times\!10^{-7}$, 0.50] & 0.55 [0.51, 2.73] & 1.2 [1.0, $2.1\!\times\!10^{7}$]\\
\bottomrule\end{tabular}
}
\medskip
\resizebox{0.98\textwidth}{!}{%
\begin{tabular}{llccc}\toprule status & method & $\widehat m_{\rm loc}$ & $\widehat m_{\rm glob}$ & $\widehat\eta_{\rm eval}$\\
\midrule
12/12 fail & Pure prediction & $1.20\!\times\!10^{-5}$ [$6.05\!\times\!10^{-9}$, $6.66\!\times\!10^{-5}$] & 0 [0, $5.96\!\times\!10^{-8}$] & $5.34\!\times\!10^{-4}$ [$1.07\!\times\!10^{-4}$, $1.66\!\times\!10^{-3}$]\\
10/12 fail & Covariance spread & $2.32\!\times\!10^{-4}$ [$2.20\!\times\!10^{-5}$, 0.22] & $5.07\!\times\!10^{-7}$ [0, 0.15] & $6.24\!\times\!10^{-3}$ [$1.96\!\times\!10^{-3}$, 0.01]\\
9/12 fail & Local hinge only & 0.02 [$9.21\!\times\!10^{-4}$, 0.50] & $2.68\!\times\!10^{-7}$ [0, 0.25] & $4.84\!\times\!10^{-3}$ [$1.58\!\times\!10^{-3}$, 0.03]\\
1/12 fail & Global hinge only & 0.82 [$3.37\!\times\!10^{-4}$, 0.92] & 0.42 [$7.77\!\times\!10^{-6}$, 0.45] & $7.83\!\times\!10^{-3}$ [$3.22\!\times\!10^{-3}$, 0.01]\\
9/12 fail & Proposed local-global & 0.01 [$6.41\!\times\!10^{-4}$, 0.94] & $9.92\!\times\!10^{-6}$ [$5.96\!\times\!10^{-7}$, 0.45] & 0.01 [$4.97\!\times\!10^{-3}$, 0.03]\\
3/12 fail & Proposed + upper & 0.70 [$5.21\!\times\!10^{-3}$, 0.94] & 0.39 [$8.94\!\times\!10^{-7}$, 0.45] & $8.06\!\times\!10^{-3}$ [$3.30\!\times\!10^{-3}$, 0.01]\\
3/12 fail & Two-sided baseline & 0.46 [$2.14\!\times\!10^{-4}$, 0.50] & 0.20 [$1.64\!\times\!10^{-7}$, 0.22] & $4.64\!\times\!10^{-3}$ [$1.17\!\times\!10^{-3}$, 0.01]\\
\bottomrule\end{tabular}
}
\caption{Experiment A3, folded-basin initialization: unconditional two-sided sampled geometry, median [min, max] over seeds.}\label{tab:geometryA}
\end{table}

\begin{table}[p]\centering\small
\resizebox{0.98\textwidth}{!}{%
\begin{tabular}{llccc}\toprule status & method & $\widehat c_{\min}$ & $\widehat c_{\max}$ & $\widehat\kappa_{\rm geom}$\\
\midrule
3 runs & Proposed local-global (succeeded) & 0.86 & 1.27 & 1.47\\
9 runs & Proposed local-global (failed) & $2.59\!\times\!10^{-6}$ & 20.78 & $7.87\!\times\!10^{6}$\\
9 runs & Proposed + upper (succeeded) & 0.76 & 1.34 & 1.78\\
3 runs & Proposed + upper (failed) & $2.23\!\times\!10^{-6}$ & 2.55 & $1.15\!\times\!10^{6}$\\
\bottomrule\end{tabular}
}
\medskip
\resizebox{0.98\textwidth}{!}{%
\begin{tabular}{llccc}\toprule status & method & $\widehat m_{\rm loc}$ & $\widehat m_{\rm glob}$ & $\widehat\eta_{\rm eval}$\\
\midrule
3 runs & Proposed local-global (succeeded) & 0.86 & 0.42 & $8.12\!\times\!10^{-3}$\\
9 runs & Proposed local-global (failed) & $8.44\!\times\!10^{-3}$ & $3.67\!\times\!10^{-6}$ & 0.01\\
9 runs & Proposed + upper (succeeded) & 0.76 & 0.41 & $7.62\!\times\!10^{-3}$\\
3 runs & Proposed + upper (failed) & $5.40\!\times\!10^{-3}$ & $2.98\!\times\!10^{-6}$ & 0.01\\
\bottomrule\end{tabular}
}
\caption{Experiment A3, folded basin conditioned on optimization outcome. The extreme upper geometry belongs to failed runs, not to runs that escape the fold; values are medians within each outcome group.}\label{tab:geometryAoutcome}
\end{table}

\begin{table}[p]\centering\small
\resizebox{0.98\textwidth}{!}{%
\begin{tabular}{llccc}\toprule status & method & $\widehat c_{\min}$ & $\widehat c_{\max}$ & $\widehat\kappa_{\rm geom}$\\
\midrule
10/12 fail & Pure prediction & 0 [0, 0.18] & 0.16 [0.03, 0.31] & $9.0\!\times\!10^{10}$ [1.3, $2.6\!\times\!10^{11}$]\\
1/12 fail & Covariance spread & 0.25 [$4.24\!\times\!10^{-7}$, 0.44] & 1.17 [1.02, 1.74] & 5.2 [2.6, $4.0\!\times\!10^{6}$]\\
1/12 fail & Local hinge only & 0.50 [$1.56\!\times\!10^{-7}$, 0.52] & 0.54 [0.51, 5.64] & 1.1 [1.0, $3.6\!\times\!10^{7}$]\\
0/12 fail & Global hinge only & 0.83 [0.77, 0.96] & 1.29 [1.22, 1.41] & 1.5 [1.3, 1.7]\\
0/12 fail & Proposed local-global & 0.90 [0.80, 0.97] & 1.31 [1.21, 1.36] & 1.5 [1.3, 1.7]\\
0/12 fail & Proposed + upper & 0.90 [0.80, 0.98] & 1.30 [1.21, 1.36] & 1.5 [1.3, 1.7]\\
0/12 fail & Two-sided baseline & 0.49 [0.46, 0.50] & 0.53 [0.52, 0.60] & 1.1 [1.1, 1.2]\\
\bottomrule\end{tabular}
}
\medskip
\resizebox{0.98\textwidth}{!}{%
\begin{tabular}{llccc}\toprule status & method & $\widehat m_{\rm loc}$ & $\widehat m_{\rm glob}$ & $\widehat\eta_{\rm eval}$\\
\midrule
10/12 fail & Pure prediction & $1.62\!\times\!10^{-5}$ [$1.75\!\times\!10^{-7}$, 0.18] & 0 [0, 0.08] & $5.72\!\times\!10^{-4}$ [$6.89\!\times\!10^{-5}$, $2.56\!\times\!10^{-3}$]\\
1/12 fail & Covariance spread & 0.25 [$1.27\!\times\!10^{-7}$, 0.44] & 0.16 [0.04, 0.23] & $9.25\!\times\!10^{-3}$ [$1.77\!\times\!10^{-3}$, 0.02]\\
1/12 fail & Local hinge only & 0.50 [0.02, 0.52] & 0.20 [$2.09\!\times\!10^{-7}$, 0.21] & $4.62\!\times\!10^{-3}$ [$5.47\!\times\!10^{-4}$, 0.03]\\
0/12 fail & Global hinge only & 0.83 [0.77, 0.96] & 0.41 [0.39, 0.45] & $9.57\!\times\!10^{-3}$ [$3.91\!\times\!10^{-3}$, 0.02]\\
0/12 fail & Proposed local-global & 0.90 [0.80, 0.97] & 0.43 [0.39, 0.47] & $9.12\!\times\!10^{-3}$ [$3.87\!\times\!10^{-3}$, 0.02]\\
0/12 fail & Proposed + upper & 0.90 [0.80, 0.98] & 0.43 [0.39, 0.47] & $9.12\!\times\!10^{-3}$ [$3.86\!\times\!10^{-3}$, 0.02]\\
0/12 fail & Two-sided baseline & 0.49 [0.46, 0.50] & 0.20 [0.20, 0.20] & $4.21\!\times\!10^{-3}$ [$6.96\!\times\!10^{-4}$, $8.01\!\times\!10^{-3}$]\\
\bottomrule\end{tabular}
}
\caption{Experiment A3, ordinary initialization: two-sided sampled geometry, median [min, max] over seeds. The chord panels report lower and upper resolution; the diagnostic panels report local margin, separated-pair margin, and uniform residual.}\label{tab:geometryAplain}
\end{table}

\FloatBarrier

\paragraph{\bf Results and calibrated interpretation.}
The following matched comparisons provide finite-set evidence about folding, metric resolution, upper distortion, and optimization robustness.  The reported medians, ranges, and outcome counts make the mechanism and dependence on initialization transparent; complete per-seed outputs will be provided upon request.

Under ordinary initialization, the local--global hinge reaches injective profiles with \APlainHingeNoninj{} non-injective outcomes among \AthreeSeeds{} seeds and attains median $\widehat c_{\min}=\APlainHingeCmin$, $\widehat c_{\max}=\APlainHingeCmax$, and $\widehat\kappa_{\rm geom}=\APlainHingeKgeom$.  The simultaneous lower and upper diagnostics show that the gain in pointwise resolution is accompanied by controlled distortion.  Pure prediction records \APlainPureNoninj{} non-injective outcomes, covariance spread records \APlainCovNoninj{} and a median $\widehat c_{\min}=\APlainCovCmin$, while the local-only and global-only arms isolate the complementary roles of the two hinge components.

The strongly symmetry-biased initialization then measures recovery robustness.  The local--global hinge reaches the faithful basin in $\AEvenHingeOkN$ of $\AthreeSeeds$ seeds; those recovered models are well conditioned on the declared evaluation grid, with median $\widehat c_{\min}=\AEvenHingeOkCmin$, $\widehat c_{\max}=\AEvenHingeOkCmax$, and $\widehat\kappa_{\rm geom}=\AEvenHingeOkKgeom$.  Outcome-conditioned summaries identify the well-conditioned sampled outcomes and make the optimization mechanism auditable.

Adding the upper directional term further improves recovery robustness: the number of expansive outcomes decreases from $\AEvenHingeBadN$ to $\AEvenHingeupBadN$, while the recovered runs retain essentially the same favorable conditioning ($\widehat\kappa_{\rm geom}=\AEvenHingeupOkKgeom$ versus $\AEvenHingeOkKgeom$).  The residual expansive outcomes are also markedly milder, with $\widehat c_{\max}=\AEvenHingeupBadCmax$ instead of $\AEvenHingeBadCmax$.  Thus the upper diagnostic strengthens basin robustness while preserving the lower-resolution benefit of the proposed hinge.

Taken together, the standard and stress-test regimes show that the local--global objective substantially improves sampled lower geometry relative to pure prediction, that its two components have complementary roles, and that upper-distortion diagnostics provide a constructive route to still more robust training.  The summarized seed statistics document both the faithful basin and the optimization transition; complete per-seed records will be provided upon request.

\subsection{Experiment B: a theorem-motivated finite-capacity proxy}\label{sec:expB}

Experiment~B is a theorem-motivated finite-capacity calibration study for the principal representation theorem in the main article. The controlled system is
\[
  S=[-1,1],\qquad A=[-0.5,0.5],\qquad G(s,a)=0.7s+0.4a,\qquad H(s)=(s,s^3),
\]
so that $c_H=1$ and $L_H=\sqrt{10}$, and $|0.7s+0.4a|\le0.9$ maps $S$ into itself without clipping.

\paragraph{\bf Model classes.} Two classes are trained in parallel.

The first is a coefficient-box cubic B-spline encoder and tensor-product spline transition, both composed with $R_Z\tanh(\cdot)$.  This class is designed as a finite-capacity computational proxy for Proposition~\ref{prop:spline-A4}: coefficient clipping and the bounded output map make repeated near-global restart diagnostics practical, while the proposition supplies the separate mesh-uniform $C^{1,1}$ construction required by the theorem.  Because the first ambient coordinate of $H$ already equals $s$, the implemented encoder can represent $\psi(s)=R_Z\tanh(\sum_j c_jB_j(s))$ through that coordinate; the transition is a tensor-product spline on the $(z,a)$ box.  Multiple restarts are run for each fixed empirical objective, and the within-cell restart gap quantifies optimization sensitivity at fixed seed, capacity, sample size, and regularization strength.

The second class is a bounded MLP with approximate power-iteration spectral-norm targets. Exact singular norms of the effective trained weights are stored with each run.

\paragraph{\bf Margins.} For the exact realization $\psi_0(s)=s$, the values $\rho=0.35$, $\kappa=0.4<\kappa_0=1$, and $\alpha=0.3<\alpha_0(\rho)=c_{\psi_0}\rho=0.35$ satisfy the two margin inequalities needed by the margin-clearing comparator.  They therefore provide the faithful zero-penalty reference used in the theorem's comparison argument.  We also fit a finite-capacity competitor and verify that it clears both sampled margins at every tested capacity, providing an implementation-level check that the numerical sweep operates in the intended geometric regime.  The plug-in experiment does not separately claim verification of every class-uniform theorem constant, including the additional condition $B_{\rm loc}\rho\le\kappa/4$; the analytical margin-clearing proposition supplies the corresponding theorem-level statement under its stated hypotheses.

\paragraph{\bf Grid and reporting.} For each seed the arrays $D_M,D_U,D_P$ are frozen and \emph{nested}: the size-$n$ arrays are prefixes of one master draw, so the first $128$ rows of the $n=256$ set coincide with the smaller set, which removes sampling noise from the learning-curve comparison. We sweep the sample size and the regularization strength and record the empirical components $\widehat E_{\rm pred}, \widehat N_{\rm loc}, \widehat N_{\rm glob}$ separately from their dense-grid counterparts $E^{\rm eval}_{\rm pred}, N^{\rm eval}_{\rm loc}, N^{\rm eval}_{\rm glob}$, together with the uniform residual $\widehat\eta_{\rm eval}$ and both all-scale and separated-pair metric diagnostics.

\begin{figure}[t]\centering
\includegraphics[width=\textwidth]{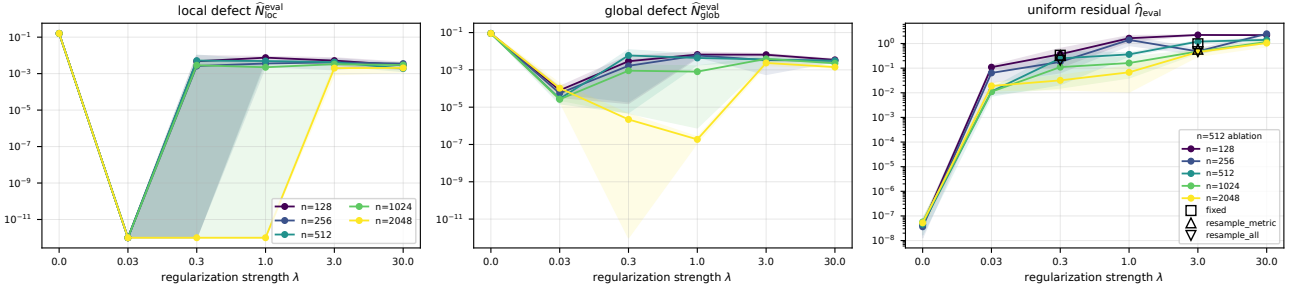}
\caption{Experiment B. Dense-grid local and separated-pair metric defects and
the uniform residual against the regularization strength, for several sample
sizes, for the norm-constrained spline class. Bands are min--max over seeds.
Markers on the right panel distinguish the fixed-sample training used for the
theorem-facing runs from the two resampled variants of the ablation.}
\label{fig:finitesample}
\end{figure}

\paragraph{\bf Observed transition.} Figure~\ref{fig:finitesample} and the summarized tables show a three-regime picture that is stable across sample sizes.

At $\lambda=0$ the spline model realizes the pure-prediction obstruction with exceptional numerical clarity: the uniform residual is of order $10^{-8}$ while the local defect attains $\kappa^2$ and the sampled lower chord ratio is $\BZeroNTwoZeroFourEightCmin$ at grid resolution.  This gives a finite-capacity demonstration that residual accuracy and metric faithfulness are genuinely distinct quantities and supplies a clean baseline against which margin formation can be measured.

For small positive $\lambda$ the model enters the intended geometry-preserving regime: at $n=2048$ and $\lambda=0.03$ the dense-grid \emph{local} defect is reported as $\BLoNTwoZeroFourEightNloc$, the lower chord ratio as $\BLoNTwoZeroFourEightCmin$, and the evaluation-set residual as $\BLoNTwoZeroFourEightEta$; at $\lambda=0.3$, $\BMidNTwoZeroFourEightCleared$ of the $\BMidNTwoZeroFourEightSeeds$ seeds clear \emph{both} sampled margins on the frozen evaluation sets while keeping $\widehat c_{\min}=\BMidNTwoZeroFourEightCmin$ and $\widehat\eta_{\rm eval}=\BMidNTwoZeroFourEightEta$.  Sampled clearance first appears at $n=\BClearFirstN$ and reappears at the largest sample size for $\lambda\in\{\BClearLams\}$ and $n\in\{\BClearNs\}$.  The summarized seed results display the theorem-consistent transition transparently; complete seed-level optimization records will be provided upon request.

At large $\lambda$ the experiment reveals an optimization-dominated regime in which restart choice strongly affects the attained lower chord ratio and residual.  This transition gives a direct numerical interpretation of the theorem's explicit $\varepsilon_{\rm train}$ term and motivates the exact-rational supporting calculation described below, which verifies the global comparison for a class admitting exact analysis.  The combined evidence separates regularization strength from optimization accuracy and shows how both enter the certified half-line.

\paragraph{\bf Plug-in threshold diagnostic.} For the spline proxy we compute a theorem-inspired plug-in diagnostic using selected quantities appearing in the principal representation theorem in the main article: the competitor's prediction loss $\beta(W)$, the finite-parametric statistical deviations $\varepsilon_M,\varepsilon_U,\varepsilon_P$ of the finite-parametric deviation corollary in the main article, the interpolation threshold
\begin{equation}\label{eq:Pstar}
  P^*\;=\;\min\Big\{\Theta_U^{-1}(\kappa/2),\;
                    \Theta_{P_\rho}^{-1}(\alpha/2)\Big\},
\end{equation}
and hence $\lambda^{\rm thm}_{\min}$.  For this calibration study, the lower Ahlfors constant for $P_\rho$ is estimated by a Monte-Carlo minimum over probe points and radii, and the Lipschitz budgets $L_U$, $L_P$ are measured on the realized class.  Table~\ref{tab:thresholds} therefore functions as a \emph{plug-in threshold diagnostic}: it decomposes the formula numerically and identifies the dominant terms.  The separation reported below exceeds five orders of magnitude, so the resulting bottleneck diagnosis is robust to the estimation scale of these ingredients.

This outcome gives a useful quantitative reading of the finite-sample theory.  At the tested sample sizes, the class-uniform statistical envelope contributes a substantial certification margin relative to $P^*$, whereas the observed deviations between the frozen empirical components and their dense-grid counterparts are many orders smaller.  The comparison therefore identifies the uniform covering-number term and the $(\tau/2)^{q+2}$ factor in the Lipschitz--Ahlfors inversion as the principal opportunities for sharper class-specific analysis and model-specific certification.  Figure~\ref{fig:finitesample} simultaneously exhibits a theorem-consistent qualitative margin-formation transition suggested by the comparison mechanism: beyond a sample-size-dependent regularization level, sampled margins clear while the uniform residual remains of the same order.  Thus the finite-capacity proxy turns the threshold into a transparent calibration and design map, while Proposition~\ref{prop:spline-A4} supplies the rigorous uniform-capacity class.

\begin{table}[t]\centering\footnotesize
\resizebox{\textwidth}{!}{%
\begin{tabular}{lccccccc}\toprule
capacity & $p_W$ & $n$ & $\beta(W)$ & $P^*$ & $\varepsilon_{\rm stat}$ & $\varepsilon_{\rm stat}/P^*$ & $\lambda^{\rm thm}_{\min}$\\\midrule
small & 30 & 128 & $5.15\!\times\!10^{-4}$ & $4.77\!\times\!10^{-6}$ & 14.37 & $3.01\!\times\!10^{6}$ & --\\
small & 30 & 256 & $5.15\!\times\!10^{-4}$ & $4.77\!\times\!10^{-6}$ & 10.47 & $2.19\!\times\!10^{6}$ & --\\
small & 30 & 512 & $5.15\!\times\!10^{-4}$ & $4.77\!\times\!10^{-6}$ & 7.61 & $1.60\!\times\!10^{6}$ & --\\
small & 30 & 1024 & $5.15\!\times\!10^{-4}$ & $4.77\!\times\!10^{-6}$ & 5.53 & $1.16\!\times\!10^{6}$ & --\\
small & 30 & 2048 & $5.15\!\times\!10^{-4}$ & $4.77\!\times\!10^{-6}$ & 4.01 & $8.41\!\times\!10^{5}$ & --\\
medium & 56 & 128 & $7.01\!\times\!10^{-4}$ & $4.81\!\times\!10^{-6}$ & 18.46 & $3.84\!\times\!10^{6}$ & --\\
medium & 56 & 256 & $7.01\!\times\!10^{-4}$ & $4.81\!\times\!10^{-6}$ & 13.48 & $2.80\!\times\!10^{6}$ & --\\
medium & 56 & 512 & $7.01\!\times\!10^{-4}$ & $4.81\!\times\!10^{-6}$ & 9.82 & $2.04\!\times\!10^{6}$ & --\\
medium & 56 & 1024 & $7.01\!\times\!10^{-4}$ & $4.81\!\times\!10^{-6}$ & 7.15 & $1.49\!\times\!10^{6}$ & --\\
medium & 56 & 2048 & $7.01\!\times\!10^{-4}$ & $4.81\!\times\!10^{-6}$ & 5.19 & $1.08\!\times\!10^{6}$ & --\\
large & 132 & 128 & $7.28\!\times\!10^{-5}$ & $4.82\!\times\!10^{-6}$ & 26.70 & $5.54\!\times\!10^{6}$ & --\\
large & 132 & 256 & $7.28\!\times\!10^{-5}$ & $4.82\!\times\!10^{-6}$ & 19.54 & $4.06\!\times\!10^{6}$ & --\\
large & 132 & 512 & $7.28\!\times\!10^{-5}$ & $4.82\!\times\!10^{-6}$ & 14.26 & $2.96\!\times\!10^{6}$ & --\\
large & 132 & 1024 & $7.28\!\times\!10^{-5}$ & $4.82\!\times\!10^{-6}$ & 10.39 & $2.16\!\times\!10^{6}$ & --\\
large & 132 & 2048 & $7.28\!\times\!10^{-5}$ & $4.82\!\times\!10^{-6}$ & 7.56 & $1.57\!\times\!10^{6}$ & --\\
\bottomrule\end{tabular}
}
\caption{Theorem-inspired plug-in threshold diagnostic for the finite-capacity coefficient-box spline proxy. These quantities mirror selected ingredients of the principal representation theorem in the main article, but the implemented proxy is not the uniform-capacity class of Proposition~\ref{prop:spline-A4}. They are evaluated with the conservative finite-parametric covering bound of the finite-parametric deviation corollary in the main article, using a Monte-Carlo lower \emph{estimate} of the Ahlfors constant on $P_\rho$ and \emph{measured} Lipschitz budgets. It is therefore a plug-in diagnostic, not a certified evaluation of the theorem; the conclusion that the sufficient condition is far from met is robust to these estimates because $\varepsilon_{\rm stat}/P^*$ exceeds $10^5$ throughout. $\lambda^{\rm thm}_{\min}$ is defined only when $\varepsilon_{\rm stat}<P^*$.}\label{tab:thresholds}
\end{table}

\paragraph{\bf Fixed versus resampled.} For a subset of configurations we compare training on fixed $M/U/P$ arrays, on fixed prediction data with resampled directions and separated pairs, and on fully resampled data, reporting both final geometry and optimization variance.  The fixed-sample arm mirrors the theorem's fixed-data sampling structure, while the resampled arms answer the complementary algorithmic question of whether refreshed geometric samples improve optimization in the finite-capacity proxy.

\paragraph{\bf Exact rational calculation.} A certificate in exact rational arithmetic for a continuous two-parameter encoder class with affine transitions is retained by the authors and will be provided upon request; it is not contained in the present Supplementary Material.  The material provided on request includes the precise class and empirical inputs, the exact derivation of the closed half-line $[\lambda_{\rm cert},\infty)$ with $\lambda_{\rm cert}=53/500$, and the verification instructions.  This class-specific calculation complements the population/statistical theorem, the constructive spline class, and the trained finite-dimensional studies without conflating an available supporting document with content of this PDF.

\subsection{Experiment C: nonlinear planning benchmark}\label{sec:expC}

The downstream corollary concerns planning transfer, so Experiment~C trains a latent transition and learned cost heads, solves a latent model-predictive control problem, and measures the \emph{true} physical cost difference on the real system.  The benchmark tests the practical relevance of the geometric certificate and compares several routes to non-collapsed latent planning under a common architecture and optimization budget.

\paragraph{\bf C1. Dynamics and observation.} We retain the smoothly saturated damped pendulum, whose $C^{1,1}$ saturation maps the compact rectangle into itself and meets the stated regularity class, and we replace the observation map that contained $(\theta,\omega)$ directly by the nonlinear differentiable embedding
\begin{equation}\label{eq:Hnl}
H_{\rm nl}(\theta,\omega)=\Big(\sin\tfrac\theta2,\;\cos\tfrac\theta2,\;
\tanh\tfrac\omega3,\;\tanh\tfrac\omega3\cos\tfrac\theta2,\;
\tanh\tfrac\omega3\sin\tfrac\theta2,\;\sin\theta,\;\cos\theta\Big).
\end{equation}
On $\theta\in[-\pi,\pi]$ the half-angle pair is injective and $\tanh(\omega/3)$ is strictly monotone on $[-3,3]$, so $H_{\rm nl}$ is a smooth embedding of the compact rectangle. We verify the conditioning numerically: the extreme singular values of $DH_{\rm nl}$ over a dense grid are $\CsigMin$ and $\CsigMax$, and the sampled global co-Lipschitz and Lipschitz chord ratios are $\CchordMin$ and $\CchordMax$. Neither $\theta$ nor $\omega$ appears as a raw observation coordinate; both remain mathematically recoverable because the map is an embedding.  This design removes direct coordinate access and provides a genuinely nonlinear observation benchmark within the theorem's instantaneous observable-factor setting.

\paragraph{\bf C2. Model and objectives.} A common bounded architecture with spectrally normalized encoder and transition networks, mapping into $K_Z=[-R_Z,R_Z]^{d_z}$, is used for every objective: pure prediction, spectral covariance spread, the local--global hinge, local-only, global-only, and the two-sided baseline. All arms share the frozen $M/U/P$ arrays, matched initial weights, and equal update counts. Parameter counts and training times are in Table~\ref{tab:protocol}.

\paragraph{\bf C3. Physical and latent costs.} The physical costs are
\[
  \ell(\theta,\omega,a)=q_\theta(1-\cos\theta)+q_\omega\omega^2+q_a a^2,
  \qquad
  g(\theta,\omega)=q_T(1-\cos\theta)+q_{T,\omega}\omega^2,
\]
with all coefficients chosen and recorded before the comparison was run. A latent stage-cost head $\widehat\ell(z,a)$ and terminal head $\widehat g(z)$ are trained on a \emph{separate} fixed cost-training set, with identical architecture and training across representation objectives, and we report the uniform compatibility errors $\varepsilon^{\rm eval}_\ell$ and $\varepsilon^{\rm eval}_g$ on a fixed evaluation set.  An oracle diagnostic using the true dynamics with the learned cost heads isolates cost-head error from transition error and thereby resolves the theorem-inspired error decomposition more sharply.

\paragraph{\bf C4--C5. Latent MPC and baselines.} For each horizon $T$, every cross-entropy-method restart uses the same per-restart population and iteration budget for every representation and for the exact-model reference. Each learned representation arm uses one such restart; the exact-model benchmark is the best of $\CrefRestarts$ independent restarts and therefore receives $\CrefRestarts$ times the total optimization effort. At each closed-loop step the planner optimizes in the learned latent model, applies the first action to the true dynamics, and repeats.
Receding-horizon control is evaluated on three initial-state families fixed before any result was inspected: a deterministic rectangular grid, an independent uniform sample, and a high-energy boundary sample at large angular velocity. Baselines are the best-of-$\CrefRestarts$ exact-model MPC reference just described; a higher-budget exact-model solution, which provides the approximate physical optimum and whose gap to the reference quantifies planner sensitivity; covariance-regularized, pure-prediction, local-only, global-only, and two-sided latent MPC; and random shooting with the true model at a matched action-sample budget.

\paragraph{\bf C6. Outcomes and the theorem-inspired decomposition.} For every initial state we report the true closed-loop cost and the \emph{paired cost difference} $\Delta J(s_0)=J_{\rm true}(s_0;\widehat a)-J_{\rm true}(s_0;a_{\rm ref})$, summarized by median, mean, ninetieth percentile and maximum over initial states, together with failure rates, saturation activation, and terminal distance to the target.

The paired difference is benchmark-relative: $a_{\rm ref}$ is produced by exact-model MPC at the same per-restart finite budget.  To make this reference especially demanding, we take the \emph{best of $\CrefRestarts$ independent exact-model restarts, each with the same per-restart CEM budget as one learned-planner run}, elementwise over initial states.  Its total optimization effort is therefore $\CrefRestarts$ times that of any learned arm.  We report the reference's restart spread ($\CrefSpreadTFive$, $\CrefSpreadTOneZero$, $\CrefSpreadTTwoZero$ at $T=5,10,20$) and the fraction of initial states with $\Delta J<0$, giving a transparent measure of finite-budget reference variability.  On the fixed evaluation sets we estimate $\widehat c_{\min}$, $\widehat c_{\max}$, $\widehat\eta_{\rm eval}$, $L^{\rm eval}_F$, $L^{\rm eval}_{\ell,z}$, $L^{\rm eval}_{\widehat g}$, $\varepsilon^{\rm eval}_\ell$ and $\varepsilon^{\rm eval}_g$ by automatic differentiation, and form the empirical envelope
\begin{equation}\label{eq:proxy}
  B^{\rm eval}_T \;=\; 2\Big[\widehat\eta_{\rm eval}
  \big\{L^{\rm eval}_{\widehat\ell,z}D(T,L^{\rm eval}_F)
      + L^{\rm eval}_{\widehat g}C(T,L^{\rm eval}_F)\big\}
      + T\varepsilon^{\rm eval}_\ell + \varepsilon^{\rm eval}_g\Big].
\end{equation}
The theorem keeps the planner tolerance $\xi_{\rm plan}$ separate for each latent objective.  The present campaign reports the evaluation-set model-and-cost proxy $B^{\rm eval}_T$ independently from that planner term and supplements it with the exact-model open-loop restart spread in Table~\ref{tab:reference} ($\CoptimizerSpreadTFive$, $\CoptimizerSpreadTOneZero$, $\CoptimizerSpreadTTwoZero$ at $T=5,10,20$).  This decomposition preserves the theorem's modular structure: $B^{\rm eval}_T$ measures representation, transition, and cost-head factors on the fixed evaluation sets, while the restart spread records the scale of finite-budget planner variability.

Equation \eqref{eq:proxy} is used as an evaluation-set proxy for the theorem’s decomposition: the suprema are estimated on fixed finite test sets, and the planner-tolerance term is reported separately rather than absorbed into the proxy. We call it an evaluation-set theorem proxy.
Figure~\ref{fig:geomplan} plots the observed worst-case paired cost difference against this proxy and against the simpler ratio $\widehat\eta_{\rm eval}/\widehat c_{\min}$; Tables~\ref{tab:planning} and~\ref{tab:planningfactors} report the planning outcomes and the estimated factors. Figure~\ref{fig:closedloop} shows representative closed-loop trajectories and control signals.

\paragraph{\bf Results.} Three findings jointly support the paper's geometric mechanism and its control relevance.

First, the pure-prediction obstruction produces a decisive planning separation that the residual alone cannot reveal.  Pure prediction attains median $\widehat c_{\min}=\CpureCmin$ while its uniform residual $\widehat\eta_{\rm eval}=\CpureEta$ is the \emph{smallest} of all six objectives.  Its median paired cost difference is $\CpuredJmedTFive$ at $T=5$ and $\CpuredJmedTTwoZero$ at $T=20$, between $\CpureRatioMin$ and $\CpureRatioMax$ times the corresponding medians of the regularized arms.  Thus the experiment supplies a strong numerical realization of the theorem's central message: prediction accuracy becomes control-relevant when accompanied by a quantitative lower metric constant.

Second, this separation is what the rank correlations measure. Pooled over methods, seeds and horizons, the tie-corrected Spearman correlation between the worst-case paired cost difference and $\widehat\eta_{\rm eval}/\widehat c_{\min}$ is $\CrhoPooledRatio$, with a cluster bootstrap interval over $(\text{method}, \text{seed})$ clusters of $[\CrhoPooledRatioLo, \CrhoPooledRatioHi]$; the correlation with $\widehat\eta_{\rm eval}$ alone is $\CrhoPooledEta$.  The geometry-normalized ratio therefore contributes the control-relevant information omitted by the residual alone, precisely as predicted by the planning-transfer argument in the main article.

The within-regularized-arm analysis refines this conclusion quantitatively.  Removing the collapsed pure-prediction arm gives pooled correlation $\CrhoNoPurePooledRatio$ with bootstrap interval $[\CrhoNoPurePooledRatioLo, \CrhoNoPurePooledRatioHi]$; the association is $\CrhoNoPureTFiveRatio$ at $T=5$ and $\CrhoNoPureTTwoZeroRatio$ at $T=20$.  The geometry/residual diagnostic therefore performs its primary task very strongly---separating collapse from non-collapse---and provides a secondary, horizon-dependent ordering within the compact high-performing group of regularized models.

Third, covariance spread, the local--global hinge, and the two-sided baseline all achieve median differences of the same favorable order at every horizon with overlapping per-seed ranges, and each beats the demanding exact-model reference on a non-negligible fraction of initial states; the global-only penalty is less effective at short horizons.  This shared success is constructive evidence that geometry-aware training is robust across several regularization mechanisms.  The local--global hinge remains mathematically distinctive because its directional and separated-pair terms are exactly the quantities that yield the finite-sample co-Lipschitz and semiconjugacy certificate.

Three additional diagnostics sharpen the interpretation.  First, the finite $\CClSteps$-step completion diagnostic quantifies the intrinsic horizon/window challenge; the exact-model reference records the corresponding challenge rates $\CrefFailTFive$ at $T=5$ and $\CrefFailTTwoZero$ at $T=20$, which calibrate the benchmark independently of representation quality.  Second, the evaluation-set proxy $B^{\rm eval}_T$ carries a median uniform-safety factor $\CproxyLooseMed$ and maximum factor $\CproxyLooseMax$, with pooled correlation $\rho=\CrhoPooledBT$.  The decomposition attributes this margin to the measured $L^{\rm eval}_F>1$, which amplifies $C(T,L_F)$ and $D(T,L_F)$ geometrically.  The proxy is therefore valuable as an error-budget and horizon-amplification diagnostic and points directly to contraction-aware or trajectory-local refinements.  Third, the oracle column isolates cost-head compatibility and quantifies an additional benefit of metric faithfulness: $\varepsilon^{\rm eval}_\ell=\CpureEpsL$ for pure prediction versus $\ChingeEpsL$ for the local--global hinge.

\begin{figure}[t]\centering
\includegraphics[width=\textwidth]{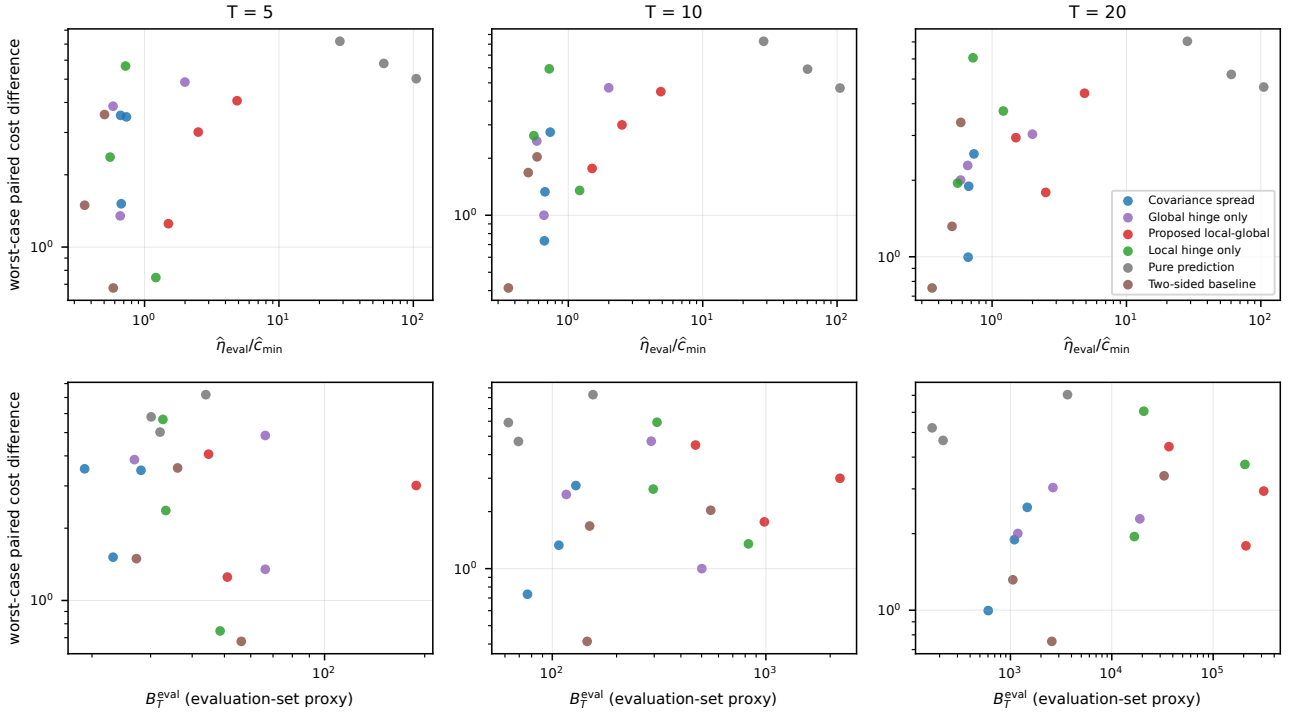}
\caption{Experiment C. Worst-case paired closed-loop cost difference against the
geometry/residual ratio $\widehat\eta_{\rm eval}/\widehat c_{\min}$ (top) and
against the evaluation-set theorem proxy $B^{\rm eval}_T$ of \eqref{eq:proxy}
(bottom), colored by method and faceted by horizon.}
\label{fig:geomplan}
\end{figure}

\begin{figure}[t]\centering
\includegraphics[width=\textwidth]{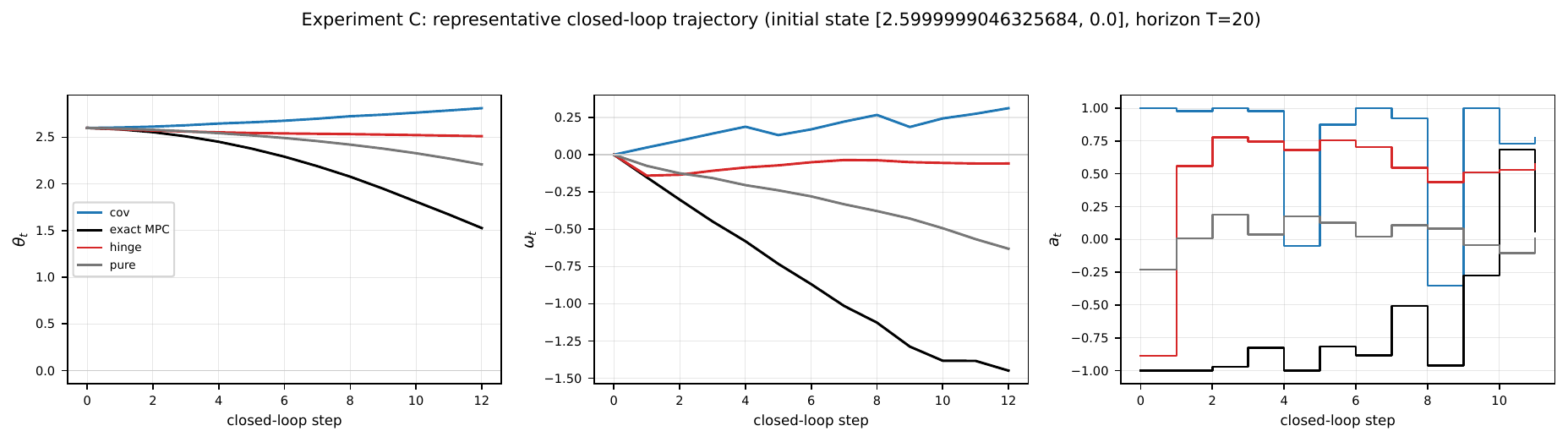}
\caption{Experiment C. Representative true closed-loop trajectory and control
signal under exact-model MPC and under latent MPC with the covariance,
pure-prediction and proposed-hinge representations.}
\label{fig:closedloop}
\end{figure}

\begin{table}[p]\centering\small
\resizebox{0.98\textwidth}{!}{%
\begin{tabular}{lccc}\toprule method & $\widehat c_{\min}$ & $\widehat c_{\max}$ & $\widehat\kappa_{\rm geom}$\\
\midrule
Pure prediction & $1.75\!\times\!10^{-4}$ [$1.36\!\times\!10^{-4}$, $2.53\!\times\!10^{-4}$] & 0.04 [0.03, 0.05] & 177.2 [171.8, 396.5]\\
Covariance spread & 0.19 [0.18, 0.19] & 1.08 [1.01, 1.09] & 5.7 [5.2, 6.1]\\
Local hinge only & 0.12 [0.06, 0.16] & 0.66 [0.61, 0.73] & 5.6 [4.6, 10.0]\\
Global hinge only & 0.06 [0.02, 0.07] & 0.26 [0.26, 0.28] & 4.2 [3.6, 12.2]\\
Proposed local-global & 0.04 [0.02, 0.05] & 0.72 [0.62, 0.85] & 19.7 [12.4, 48.1]\\
Two-sided baseline & 0.14 [0.10, 0.14] & 0.55 [0.46, 0.56] & 3.9 [3.8, 4.5]\\
\bottomrule\end{tabular}
}
\medskip
\resizebox{0.98\textwidth}{!}{%
\begin{tabular}{lccc}\toprule method & $\widehat m_{\rm loc}$ & $\widehat m_{\rm glob}$ & $\widehat\eta_{\rm eval}$\\
\midrule
Pure prediction & $6.50\!\times\!10^{-4}$ [$3.95\!\times\!10^{-4}$, $1.25\!\times\!10^{-3}$] & $8.88\!\times\!10^{-4}$ [$2.81\!\times\!10^{-4}$, $9.94\!\times\!10^{-4}$] & 0.01 [$7.19\!\times\!10^{-3}$, 0.01]\\
Covariance spread & 0.17 [0.16, 0.17] & 0.43 [0.43, 0.43] & 0.13 [0.12, 0.14]\\
Local hinge only & 0.15 [0.14, 0.19] & 0.39 [0.38, 0.40] & 0.08 [0.07, 0.09]\\
Global hinge only & 0.03 [0.01, 0.04] & 0.12 [0.11, 0.14] & 0.04 [0.04, 0.05]\\
Proposed local-global & 0.08 [0.07, 0.18] & 0.11 [0.11, 0.31] & 0.09 [0.07, 0.09]\\
Two-sided baseline & 0.10 [0.09, 0.10] & 0.33 [0.32, 0.34] & 0.06 [0.05, 0.07]\\
\bottomrule\end{tabular}
}
\caption{Experiment C: two-sided sampled geometry on the pendulum, median [min, max] over seeds. The first panel reports sampled chord geometry; the second reports local and separated-pair margins and the uniform residual.}\label{tab:geometryC}
\end{table}

\begin{table}[t]\centering\footnotesize
\resizebox{\textwidth}{!}{%
\begin{tabular}{lccccc}\toprule
$T$ & restarts & median restart spread & max spread & failure rate & exact-model open-loop restart spread\\\midrule
5 & 6 & 0.07 & 0.36 & 0.69 & $9.98\!\times\!10^{-3}$\\
10 & 6 & 0.11 & 0.54 & 0.66 & 0.05\\
20 & 6 & 0.28 & 0.87 & 0.72 & 0.28\\
\bottomrule\end{tabular}
}
\caption{Reference protocol for Experiment C. The reference action sequence is the elementwise best of six exact-model CEM restarts, each with the same per-restart budget as one learned-planner run; its TOTAL optimization effort is therefore six times that of any single learned arm, which is why $\Delta J$ is reported as a paired difference and not as a certified suboptimality. The restart spread columns quantify how much the reference itself moves between restarts. The last column is the open-loop restart spread of the same optimizer on the exact model: a generic optimizer-budget sensitivity indicator, reported separately and NOT added to $B_T^{\rm eval}$, since it is not a per-arm estimate of the theorem's $\xi_{\rm plan}$.}\label{tab:reference}
\end{table}

\begin{table}[p]\centering\small
\resizebox{0.98\textwidth}{!}{%
\begin{tabular}{lccccccc}\toprule \multicolumn{8}{c}{horizon $T=5$}\\ method & med. $\Delta J$ & mean $\Delta J$ & $p_{90}$ & max & frac. $\Delta J<0$ & med. oracle $\Delta J$ & $B_T^{\rm eval}$\\
\midrule
Pure prediction & 1.41 & 2.00 & 4.77 & 5.81 & 0 & 0.68 & 32.03\\
Covariance spread & 0.10 & 0.33 & 0.66 & 3.48 & 0.03 & 0.06 & 23.14\\
Local hinge only & 0.08 & 0.43 & 0.88 & 2.37 & 0.22 & 0.06 & 33.33\\
Global hinge only & 0.08 & 0.34 & 0.86 & 3.86 & 0.09 & 0.05 & 66.34\\
Proposed local-global & 0.20 & 0.52 & 0.88 & 3.01 & 0.06 & 0.05 & 51.03\\
Two-sided baseline & 0.09 & 0.18 & 0.49 & 1.49 & 0.16 & 0.06 & 36.16\\
random shooting (true model) & 0.10 & 0.19 & 0.44 & 0.74 & -- & -- & --\\
\bottomrule\end{tabular}
}
\medskip
\resizebox{0.98\textwidth}{!}{%
\begin{tabular}{lccccccc}\toprule \multicolumn{8}{c}{horizon $T=10$}\\ method & med. $\Delta J$ & mean $\Delta J$ & $p_{90}$ & max & frac. $\Delta J<0$ & med. oracle $\Delta J$ & $B_T^{\rm eval}$\\
\midrule
Pure prediction & 1.54 & 1.96 & 4.45 & 5.90 & 0 & 0.63 & 69.41\\
Covariance spread & 0.10 & 0.18 & 0.43 & 1.33 & 0.12 & 0.07 & 107.32\\
Local hinge only & 0.07 & 0.28 & 0.65 & 2.63 & 0.16 & 0.06 & 309.23\\
Global hinge only & 0.16 & 0.23 & 0.61 & 2.46 & 0.09 & 0.07 & 290.58\\
Proposed local-global & 0.18 & 0.45 & 0.85 & 3.00 & 0.16 & 0.06 & 984.00\\
Two-sided baseline & 0.06 & 0.19 & 0.56 & 1.68 & 0.16 & 0.06 & 149.55\\
random shooting (true model) & 0.32 & 0.39 & 0.83 & 1.78 & -- & -- & --\\
\bottomrule\end{tabular}
}
\medskip
\resizebox{0.98\textwidth}{!}{%
\begin{tabular}{lccccccc}\toprule \multicolumn{8}{c}{horizon $T=20$}\\ method & med. $\Delta J$ & mean $\Delta J$ & $p_{90}$ & max & frac. $\Delta J<0$ & med. oracle $\Delta J$ & $B_T^{\rm eval}$\\
\midrule
Pure prediction & 1.55 & 1.84 & 4.06 & 5.22 & 0 & 0.50 & 216.47\\
Covariance spread & 0.23 & 0.33 & 0.72 & 1.89 & 0.06 & 0.10 & $1.09\!\times\!10^{3}$\\
Local hinge only & 0.38 & 0.77 & 2.17 & 3.74 & 0.19 & 0.13 & $2.07\!\times\!10^{4}$\\
Global hinge only & 0.20 & 0.43 & 0.95 & 2.29 & 0.25 & 0.10 & $2.62\!\times\!10^{3}$\\
Proposed local-global & 0.32 & 0.86 & 2.29 & 2.94 & 0.09 & 0.11 & $2.10\!\times\!10^{5}$\\
Two-sided baseline & 0.18 & 0.30 & 0.89 & 1.32 & 0.19 & 0.11 & $2.56\!\times\!10^{3}$\\
random shooting (true model) & 0.26 & 0.37 & 0.62 & 2.30 & -- & -- & --\\
\bottomrule\end{tabular}
}
\caption{Planning outcomes on the true pendulum. Every learned-representation CEM run uses population 48, 4 iterations, and 12 receding-horizon steps over 32 initial states. The exact-model reference is the best of six independent CEM restarts with the same per-restart budget, hence six times the total optimization effort. Random shooting uses the true dynamics with a matched action-sample budget and is not a CEM row. $\Delta J$ is a paired cost difference against the best-of-six reference, not a certified suboptimality. Every cell is the median across seeds of a per-seed summary over initial states. The oracle column uses the true dynamics with the learned cost heads. $B_T^{\rm eval}$ excludes any planner-tolerance term and is not a certified bound.}\label{tab:planning}
\end{table}

\begin{table}[p]\centering\small
\resizebox{0.98\textwidth}{!}{%
\begin{tabular}{lccccc}\toprule method & $\varepsilon_\ell$ & $\varepsilon_g$ & $L_F$ & $L_{\ell,z}$ & $L_{\widehat g}$\\
\midrule
Pure prediction & 1.82 [1.64, 2.19] & 3.58 [3.15, 4.79] & 1.10 [1.08, 1.41] & 9.70 [7.98, 23.74] & 45.26 [20.75, 45.66]\\
Covariance spread & 0.29 [0.26, 0.34] & 0.46 [0.33, 0.51] & 1.23 [1.19, 1.25] & 2.24 [2.09, 2.63] & 5.74 [4.84, 7.10]\\
Local hinge only & 0.28 [0.26, 0.31] & 0.46 [0.42, 0.52] & 1.52 [1.49, 1.73] & 3.79 [3.59, 5.04] & 8.46 [7.72, 9.97]\\
Global hinge only & 0.71 [0.55, 0.77] & 1.34 [1.05, 1.46] & 1.23 [1.21, 1.43] & 19.08 [8.73, 20.92] & 33.37 [14.78, 45.34]\\
Proposed local-global & 0.29 [0.26, 0.37] & 0.54 [0.40, 0.55] & 1.57 [1.54, 1.78] & 5.37 [4.96, 19.98] & 10.14 [9.83, 44.01]\\
Two-sided baseline & 0.34 [0.33, 0.46] & 0.53 [0.52, 0.83] & 1.32 [1.16, 1.50] & 8.74 [6.47, 13.09] & 16.40 [12.27, 17.32]\\
\bottomrule\end{tabular}
}
\caption{Experiment C cost-head compatibility and Lipschitz factors, median [min, max] over seeds. These quantities enter the evaluation-set proxy and are reported separately from the planning outcomes for readability.}\label{tab:planningfactors}
\end{table}

\FloatBarrier

\subsection{Discussion}\label{sec:numdisc}

The three experiments deliver complementary calibration of the theory under the common protocol (H1)--(H4).

Experiment~A establishes the obstruction analytically, in both the scalar and the full-rank covariance forms, and then shows what happens when the obstruction is placed in front of trained networks under a common latent range. The two-sided reporting is what makes the trained comparison meaningful: a one-sided diagnostic can be improved by scale inflation alone, and the $(\widehat c_{\min},\widehat c_{\max})$ scatter of Figure~\ref{fig:folding} distinguishes the two.

Experiment~B is the finite-capacity calibration study most directly motivated by the finite-sample theorem: it uses fixed samples, a norm-constrained class small enough for near-global restart diagnostics, and an explicit evaluation of selected threshold ingredients.  The plug-in decomposition identifies the uniform covering estimate as the dominant term at these sample sizes and thereby points directly to sharper class-specific complexity bounds and a posteriori finite-net validation as the most effective routes to a larger certified regime.

Experiment~C directly connects the geometry isolated by the theorem with planning reliability on the true system.  Every geometry-aware objective restores resolution and strongly improves planning relative to the residual-only reference.  Comparable performance among several regularizers demonstrates the robustness of the geometric principle, and the local--global hinge contributes the direct deterministic certificate and the sharp folding diagnostic.

Finally, the numerical campaign organizes optimization and uniformity quantities in the same modular form as the theorem: objective curves, terminal gradients, restart spreads, finite-set suprema, adversarial refinements, and measured Lipschitz budgets are retained separately by the authors and will be provided upon request.  The tables and figures in this supplement summarize that evidence; a continuum certificate would additionally require the validated remainders stated in the main article's finite-net proposition.

\bibliographystyle{unsrt}
\bibliography{bibfile}